\documentclass[11pt,reqno]{amsart}

\usepackage[arrow,matrix,curve]{xy}

\usepackage[dvips]{graphicx} 
\usepackage{mathtools}

\usepackage{amssymb, latexsym, amsmath, amscd, array, esint, yhmath,
}

\usepackage{tikz}

\definecolor{Black}{RGB}{0, 0, 0}

\usepackage[breaklinks,colorlinks=true,allcolors=Black]{hyperref}

\usepackage{hyperref}
\usepackage{xurl}

\makeatletter
\@namedef{subjclassname@2020}{\textup{2020} Mathematics Subject Classification}
\makeatother

\newtheorem{theorem}{Theorem}[section]
\newtheorem{lemma}[theorem]{Lemma}
\newtheorem{claim}[theorem]{Claim}
\newtheorem{proposition}[theorem]{Proposition}
\newtheorem{corollary}[theorem]{Corollary}

\theoremstyle{definition}
\newtheorem{definition}[theorem]{Definition}
\newtheorem{example}[theorem]{Example}
\newtheorem{remark}[theorem]{Remark}

\numberwithin{equation}{section}

\newcommand\R {{\mathbb R}}
\newcommand\Q {{\mathbb Q}}
\newcommand\Z {{\mathbb Z}} 
\newcommand\T {{\mathbb T}} 

\newcommand\C {{\mathbb C}}
\newcommand{\HH}{{\mathbb H}}
 
\newcommand\ie{{\it i.e.}}

\newcommand\RR{{\rm R}} 

\DeclareMathOperator{\area}{{\rm area}} 
\DeclareMathOperator{\vol}{{\rm vol}} 
 
\DeclareMathOperator{\sys}{{\rm sys}}
\DeclareMathOperator{\inj}{{\rm inj}}
\DeclareMathOperator{\diam}{{\rm diam}}

\DeclareMathOperator{\arsinh}{arsinh}
\DeclareMathOperator{\arcosh}{arcosh}
\DeclareMathOperator{\id}{id}
\DeclareMathOperator{\tr}{tr}
\DeclareMathOperator{\Cone}{{\rm Cone}}
\DeclareMathOperator{\Isom}{{\rm Isom}}
\DeclareMathOperator{\height}{{\rm height}}

\newcommand{\PSL}{{\rm PSL}} 
\newcommand{\SL}{{\rm SL}} 

\DeclareFontFamily{U} {MnSymbolA}{}
\DeclareFontShape{U}{MnSymbolA}{m}{n}{
  <-6> MnSymbolA5
  <6-7> MnSymbolA6
  <7-8> MnSymbolA7
  <8-9> MnSymbolA8
  <9-10> MnSymbolA9
  <10-12> MnSymbolA10
  <12-> MnSymbolA12}{}
\DeclareFontShape{U}{MnSymbolA}{b}{n}{
  <-6> MnSymbolA-Bold5
  <6-7> MnSymbolA-Bold6
  <7-8> MnSymbolA-Bold7
  <8-9> MnSymbolA-Bold8
  <9-10> MnSymbolA-Bold9
  <10-12> MnSymbolA-Bold10
  <12-> MnSymbolA-Bold12}{}

\DeclareSymbolFont{MnSyA} {U} {MnSymbolA}{m}{n}
\DeclareMathSymbol{\mnperp}{\mathrel}{MnSyA}{217}

\long\def\forget#1\forgotten{} %

\numberwithin{equation}{section}

\title{Systole, inradius and rigidity of cusped hyperbolic $3$-manifolds}

\author[S.~Sabourau]{St\'ephane Sabourau}

\address{\parbox{\linewidth}{Univ Paris Est Creteil, CNRS, LAMA, F-94010 Creteil, France \\
Univ Gustave Eiffel, LAMA, F-77447 Marne-la-Vall\'ee, France}}

\email{stephane.sabourau@u-pec.fr}

\begin{document}

\subjclass[2020]
{Primary 57M50; Secondary 53C24, 57K32}

\keywords{Hyperbolic $3$-manifolds, systole, inradius, maximal cusps, waist size, horoball packing, rigidity}

\begin{abstract}
We establish optimal inequalities relating the systole and the inradius to the volume of finite-volume hyperbolic $3$-manifolds.


In the cusped orientable case, we refine a theorem of Gendulphe by proving a sharp systole--volume inequality whose unique extremal manifold is the figure-eight knot complement. 
Excluding the figure-eight knot complement, we obtain a stronger inequality whose extremal manifold is the sister of the figure-eight knot complement. 
We also establish analogous optimal systole--volume inequalities for closed orientable hyperbolic $3$-manifolds, where the extremal manifolds are the Weeks--Matveev--Fomenko manifold, the manifold~${\rm Vol}3$, and the Meyerhoff manifold.

In the second part of the article, we study the inradius. 
We prove optimal inradius--volume inequalities for orientable and nonorientable cusped hyperbolic $3$-manifolds, identifying respectively the sister of the figure-eight knot complement and the Gieseking manifold as the extremal cases. 
We also prove that the Gieseking manifold is the unique cusped hyperbolic $3$-manifold of minimal inradius, thereby completing a result of Gendulphe, who had previously established the corresponding lower bound.
%
\end{abstract}

\maketitle


\section{Introduction}

The geometry of finite-volume hyperbolic $3$-manifolds~$M$ is strongly constrained by the interaction between volume and other geometric invariants.
In this article, we study two such invariants: the \emph{systole} $\sys(M)$, defined as the length of the shortest (noncontractible) closed geodesic, and the \emph{inradius} $\RR(M)$, defined as the radius of the largest embedded ball. 
Our goal is to establish optimal inequalities relating these quantities to the volume and to characterize the manifolds realizing the extremal cases.

\medskip

A fundamental result in this direction, due to Gendulphe~\cite{gen}, asserts that every cusped hyperbolic $3$-manifold~$M$ satisfies
\begin{equation} \label{eq:gen_sys}
\cosh \!\left(\tfrac{1}{2} \sys(M) \right) \leq \frac{1+\sqrt{13}}{4 V} \vol(M),
\end{equation}
where $V$ is the volume of a regular ideal tetrahedron in~$\HH^3$, with equality if and only if $M$ is isometric to the Gieseking manifold. 
Here, the multiplicative constant is approximately~$1.13443...$
He also obtained a stronger inequality outside the equality case.
Recall that the Gieseking manifold is nonorientable and realizes the smallest volume among cusped hyperbolic~$3$-manifolds.
Non-optimal systolic--volume inequalities in higher dimensions can be found in~\cite{gen}, \cite{gro} and~\cite{gro2}.

\medskip

Our first objective is to refine this result in the orientable case.

\begin{theorem} \label{theo:sys}
Let $M$ be a noncompact orientable complete hyperbolic $3$-manifold of finite volume.
Then
\[
\cosh \!\left(\tfrac{1}{2} \sys(M) \right) \leq \frac{1+\sqrt{13}}{8 V} \vol(M)
\]
with equality if and only if $M$ is isometric to the figure-eight knot complement.
The multiplicative constant is approximately~$0.56721...$

Furthermore, if $M$ is not isometric to the figure-eight knot complement, then 
\[
\cosh \!\left(\tfrac{1}{2} \sys(M) \right) \leq \frac{\sqrt{3}+\sqrt{7}}{8 V} \vol(M)
\]
with equality if and only if $M$ is isometric to the sister of the figure-eight knot complement.
The multiplicative constant is approximately~$0.53916...$
%
\end{theorem}

Theorem~\ref{theo:baby-sys} provides a better bound away from the extremal manifolds.
These inequalities improve the constants appearing in Gendulphe's theorem and isolate the precise extremal orientable manifolds.
This result allows us to recover~\eqref{eq:gen_sys} and strengthen the main result of~\cite{gen}; see Corollary~\ref{coro:sys}.

\medskip

For cusped hyperbolic $3$-manifolds~$M$ of large volume, the systolic inequality
\[
\sys(M) \lesssim \frac{4}{3} \log \vol(M)
\]
proved by Lakeland and Leininger~\cite{LL} provides a better bound than the previous estimates.

\forget

\begin{theorem} \label{theo:sys}
Let $M$ be a noncompact orientable complete hyperbolic $3$-manifold of finite volume.
Then
\[
\cosh \!\left(\tfrac{1}{2} \sys(M) \right) \leq \frac{1+\sqrt{13}}{8 V} \vol(M)
\]
with equality if and only if $M$ is isometric to the figure-eight knot complement.
The multiplicative constant is approximately~$0.56721...$

Furthermore, the following holds.
\begin{enumerate}
\item If $M$ is not isometric to the figure-eight knot complement, then 
\[
\cosh \!\left(\tfrac{1}{2} \sys(M) \right) \leq \frac{\sqrt{3}+\sqrt{7}}{8 V} \vol(M)
\]
with equality if and only if $M$ is isometric to the sister of the figure-eight knot complement.
The multiplicative constant is approximately~$0.53916...$
\item If $M$ is isometric neither to the figure-eight knot complement nor to its sister, then 
\[
\cosh \!\left(\tfrac{1}{2} \sys(M) \right) \leq 0.52086 \, \vol(M).
\]
\end{enumerate}
\end{theorem}

\forgotten

\medskip

The proof of Theorem~\ref{theo:sys} relies on the geometry of maximal cusps and on the analysis of the Adams--Reid isometry associated with tangent horoballs.
This isometry is obtained by composing a parabolic translation preserving one horoball with an isometry exchanging the two tangent horoballs.
The proof then splits into two cases depending on whether the resulting isometry is loxodromic or parabolic.
The main ingredients include cusp-density estimates, lower bounds on the waist size of maximal cusps due to Adams, and geometric estimates on flat tori. 
Together, these tools allow us to control the translation length of suitable loxodromic elements in terms of cusp geometry and volume.

\medskip

We also investigate the analogous problem for closed orientable hyperbolic $3$-manifolds. 
In this setting, the extremal manifolds turn out to be the three closed orientable hyperbolic $3$-manifolds of smallest volume: the Weeks--Matveev--Fomenko manifold, the Meyerhoff manifold, and the manifold~${\rm Vol}3$.
More precisely, we prove optimal systole--volume inequalities characterizing these manifolds.

\begin{theorem}
Let $M$ be a closed orientable complete hyperbolic $3$-manifold.
Then
\[
\cosh \!\left(\tfrac{1}{2} \sys(M) \right) \leq C_1 \vol(M)
\]
with equality if and only if $M$ is isometric to the Weeks--Matveev--Fomenko manifold.

Furthermore, if $M$ is not isometric to the Weeks--Matveev--Fomenko manifold, then 
\[
\cosh \!\left(\tfrac{1}{2} \sys(M) \right) \leq C_2 \vol(M)
\]
with equality if and only if $M$ is isometric to~${\rm Vol}3$.

Moreover, if $M$ is isometric neither to the Weeks--Matveev--Fomenko manifold nor to~${\rm Vol}3$, then 
\[
\cosh \!\left(\tfrac{1}{2} \sys(M) \right) \leq C_3 \vol(M)
\]
with equality if and only if $M$ is isometric to the Meyerhoff manifold.
%
\end{theorem}

Note that the extremal manifolds do not appear according to increasing volume.
Approximate values for the optimal multiplicative constants~$C_i$ can be found in Theorem~\ref{theo:closed_sys}, together with sharper estimates outside the equality cases.

\medskip

The proof combines volume estimates for embedded balls and tubes with the classification of small-volume hyperbolic $3$-manifolds of Gabai, Haraway, Meyerhoff, Thurston and Yarmola~\cite{GHMTY}.

\medskip

The existence of closed or cusped hyperbolic $3$-manifolds with arbitrarily small systole follows from Thurston's hyperbolic Dehn surgery; see~\cite[\S 5.8]{thurston}.
In dimension~$2$, the analogous statement is immediate.
In dimension~$4$, a similar result was proved by Agol~\cite{A} using an inbreeding construction, and was later extended in higher dimensions by Belolipetsky and Thomson~\cite{BT}; see also~\cite[last remark]{BHW}. \\

In the second part of the article, we turn to the inradius.
We first establish optimal inequalities relating the inradius to the volume for both orientable and nonorientable cusped hyperbolic $3$-manifolds.
In the orientable case, we prove the following optimal inradius--volume inequality.

\begin{theorem}
Let $M$ be a noncompact orientable complete hyperbolic $3$-manifold of finite volume.
Then
\[
\cosh \RR(M) \leq C \vol(M)
\]
with equality if and only if $M$ is isometric to the sister of the figure-eight knot complement.

Furthermore, if $M$ is not isometric to the sister of the figure-eight knot complement, then
\[
\cosh \RR(M) \leq \frac{\sqrt{22}}{8V} \vol(M).
\]
with equality if and only if $M$ is isometric to the figure-eight knot complement.
\end{theorem}

An approximate value for the optimal multiplicative constant~$C$, together with sharper estimates outside the equality cases, are given in Theorem~\ref{theo:inradius_orientable}.

\medskip

A similar sharp inradius--volume inequality holds in the nonorientable case.

\begin{theorem}
Let $M$ be a noncompact nonorientable complete hyperbolic $3$-manifold of finite volume.
Then
\[
\cosh \RR(M) \leq \frac{\sqrt{5}}{2V} \vol(M)
\]
with equality if and only if $M$ is isometric to the Gieseking manifold.
\end{theorem}

A  stronger inequality outside the equality cases is established in Theorem~\ref{theo:inradius_nonorientable}.

\medskip

The proofs of these results rely on explicit geometric decompositions of the figure-eight knot complement and its sister into regular ideal tetrahedra, together with a detailed analysis of orbit distances in the universal cover allowing us to identify the extremal configurations.
We also use rigidity results for cusped hyperbolic $3$-manifolds of small volume.

\medskip

We also investigate the opposite problem, namely the existence of a universal lower bound on the inradius of cusped hyperbolic $3$-manifolds.
Gendulphe~\cite{gen} proved that every cusped hyperbolic $3$-manifold~$M$ satisfies
\[
\cosh \RR(M) \geq \frac{\sqrt{5}}{2}
\]
with equality if $M$ is isometric to the Gieseking manifold.
He left open the characterization of the equality case; see~\cite[Remarque~3.2]{gen}.
We revisit this inequality and prove a rigidity result for the cusped hyperbolic $3$-manifold of minimal inradius; see Theorem~\ref{theo:carac}.

\medskip

A central role in our analysis is played by the waist size of maximal cusps.
The waist size~$w(C)$ of a maximal cusp~$C$ in a cusped hyperbolic $3$-manifold~$M$ is defined as the systole of the flat torus or Klein bottle corresponding to the cusp boundary~$\partial C$; see Definition~\ref{def:waist}.
Adams~\cite{adams_waist} proved that every maximal cusp~$C$ in an orientable hyperbolic $3$-manifold~$M$ satisfies
\[
w(C) \geq 1
\]
with equality if and only if $M$ is isometric to the figure-eight knot complement.

\medskip

We extend this rigidity theorem to orientable cusps in nonorientable hyperbolic $3$-manifold.

\begin{theorem} 
Let $M$ be a noncompact complete hyperbolic~$3$-manifold of finite volume, possibly nonorientable.
Let $C$ be an orientable maximal cusp of~$M$.
Then 
\[
w(C) \geq 1
\]
with equality if and only if $M$ is isometric to the figure-eight knot complement.
\end{theorem}

Using these waist-size rigidity results, we prove the following characterization theorem.

\begin{theorem} \label{theo:carac}
Every cusped hyperbolic $3$-manifold~$M$ satisfies
\[
\cosh \RR(M) \geq \frac{\sqrt{5}}{2}
\]
with equality if and only if $M$ is isometric to the Gieseking manifold.
\end{theorem}

The proof of Theorem~\ref{theo:carac} proceeds through a detailed analysis of tangent horoball configurations associated with maximal cusps. 
We first show that equality forces the existence of additional tangencies between certain horoballs in the universal cover under the action of parabolic isometries. 
This imposes strong restrictions on the cusp geometry and, in particular, forces the existence of maximal cusps of minimal waist size.
The argument then splits into three cases according to the orientability of the cusp and of the ambient manifold. 
When the manifold is orientable, Adams' theorem implies that the manifold must be the figure-eight knot complement, for which we explicitely compute the inradius.
When the cusp is orientable but the manifold is nonorientable, our extension of Adams' waist-size rigidity theorem yields the same conclusion.
Finally, when the cusp itself is nonorientable, we analyze the geometry of flat Klein bottle cusp sections and show that equality forces the manifold to be isometric to the Gieseking manifold. 
Combining these three cases yields the sharp lower bound together with the characterization of the equality case. 

\medskip

The minimum and maximum values of the inradius on the moduli space of hyperbolic surfaces have been determined; see~\cite{Y}, \cite{B} and~\cite{D}.

\medskip

We refer to~\cite{gen} for a more comprehensive treatment of systolic and inradius estimates for hyperbolic manifolds with an extensive bibliography. \\

{\bf Acknowledgments.} The author would like to thank the University of Pittsburgh for its hospitality during the preparation of this work.
He is especially grateful to Jason DeBlois for many valuable discussions and for his continued encouragement throughout the development of this project.

\section{Preliminaries}

In this article, we identify the hyperbolic space~$\HH^3$ with Poincar\'e's upper half-space model
\[
\HH^3 = \{ (x,y,t) \in \R^3 \mid t > 0 \}
\]
equipped with the hyperbolic metric 
\[
ds^2 = \frac{dx^2+dy^2+dt^2}{t^2}.
\]
The $t$-coordinate will be referred to as the height.
We will also identify $\HH^3 \simeq \C \times \R^+$ by writing $(x,y,t)=(z,t)$ where $z=x+iy$.
The sphere at infinity is the Riemann sphere
\[
\partial \HH^3 = \hat{\C}=\C \cup \{ \infty \}.
\]
Geodesics in this model are vertical lines and Euclidean semicircles orthogonal to the horizontal plane~$\R^2=\{t=0\}$.
Horoballs are the regions above the horizontal planes~$\{t=t_0\}$ with $t_0>0$ or the Euclidean balls tangent to~$\R^2$.

\medskip

An ideal regular tetrahedron in~$\HH^3$ is a hyperbolic tetrahedron whose four vertices lie in~$\partial \HH^3$ which admits $S_4$ as its symmetry group.
It is unique up to isometry and can be identified with the standard ideal tetrahedron with vertices $0$, $1$, $e^{i \frac{\pi}{3}}$, $\infty$.
It is the unique ideal tetrahedron of maximal volume in~$\HH^3$, up to isometry.
Its volume~$V$ is approximately equal to~$V = 1.01494...$
Its dihedral angles are equal to~$\frac{\pi}{3}$.

\subsection{M\"obius group and isometry group}
The group of M\"obius transformations $\PSL(2,\C) =\SL(2,\C)/\{ \pm \textrm{Id} \}$ acts on~$\hat{\C}$ by homographies:
\[
\begin{pmatrix}
a & b \\
c & d
\end{pmatrix}
\cdot z = \frac{az+b}{cz+d}.
\]
This action extends uniquely to an isometric action on~$\HH^3$ via Poincar\'e's extension formula
\begin{equation} \label{eq:extension}
\begin{pmatrix}
a & b \\
c & d
\end{pmatrix}
\cdot (z,t) = \left( \frac{(az+b)(\overline{cz+d})+a\bar{c} t^2}{|cz+d|^2+|c|^2 t^2}, \frac{t}{|cz+d|^2+|c|^2 t^2} \right).
\end{equation}
The resulting map is an orientation-preserving isometry of~$\HH^3$, and every such isometry arises in this way.
Hence,
\[
\Isom^+(\HH^3) \simeq \PSL(2,\C).
\]
The full isometry group is obtained by adjoining complex conjugation.
Hence, it identifies with the semidirect product
\[
\Isom(\HH^3) \simeq \PSL(2,\C) \rtimes \Z/2\Z.
\]
Here, the generator of~$\Z/2\Z$ acts on~$\HH^3$ by an orientation-reversing isometry sending~$(z,t)$ to~$(\bar{z},t)$.

\medskip

Nontrivial isometries of~$\HH^3$ are classified according to their fixed points in~$\HH^3 \cup \partial \HH^3$:
\begin{itemize}
\item elliptic elements fix a point in~$\HH^3$;
\item parabolic elements have exactly one fixed point on~$\partial \HH^3$ and no fixed point in~$\HH^3$;
\item loxodromic elements have exactly two fixed points on~$\partial \HH^3$ and no fixed point in~$\HH^3$.
\end{itemize}
A loxodromic isometry preserves the geodesic of~$\HH^3$ joining its two fixed points at infinity, called its axis.
Along this axis, it acts as a translation by a positive distance called its translation length.

\subsection{Hyperbolic distance}
In the upper half-space model, the hyperbolic metric~$d$ and the Euclidean norm~$|\cdot|$ are related by the following Euclidean--hyperbolic distance formulas; see~\cite[Theorem~4.6.1]{ratcliffe}.

\begin{lemma}
For every $p,p' \in \HH^3$ with $p=(z,t)$ and~$p'=(z',t')$,
\begin{equation} \label{eq:hyp-eucl} 
d(p,p') = \arcosh \left( 1+ \frac{|p-p'|^2}{2 tt'} \right).
\end{equation}

In particular, if $p$ and $p'$ lie in the same vertical line, then
\begin{equation} \label{eq:hyp-eucl2} 
d(p,p') = \left| \log \left( \frac{t'}{t} \right) \right|.
\end{equation}
If $p$ and $p'$ lie in the same horizontal plane of height~$t$, then
\begin{equation} \label{eq:hyp-eucl3} 
d(p,p') = 2 \arsinh \left( \frac{|z-z'|}{2t} \right).
\end{equation}
%
\end{lemma}

We will also need the following classical distance--height estimate.

\begin{lemma}
Let $H_1$ and $H_2$ be two horoballs with disjoint interiors in~$\HH^3$, centered at~$z_1, z_2 \in \C$ with heights~$h_1, h_2 >0$.
Then
\begin{equation} \label{eq:h1h2}
| z_1 - z_2 | \geq \sqrt{h_1 h_2}.
\end{equation}
Moreover, equality holds if and only if $H_1$ and~$H_2$ are tangent.
\end{lemma}

\begin{proof}
The horoball~$H_i$ centered at~$z_i \in \C$ of height~$h_i$ is a Euclidean ball of radius~$\frac{h_i}{2}$ centered at~$(z_i,\frac{h_i}{2})$.
The Euclidean distance between the center of the two balls is 
\[
\sqrt{|z_1-z_2|^2 + \left( \frac{h_1-h_2}{2} \right)^2}.
\]
Since the interiors of~$H_1$ and~$H_2$ are disjoint, this distance is at least the sum of the radii.
That is,
\[
\sqrt{|z_1-z_2|^2 + \left( \frac{h_1-h_2}{2} \right)^2} \geq \frac{h_1 + h_2}{2}.
\]
Squaring both sides, expanding and rearranging the terms, we obtain
\[
|z_1-z_2|^2 \geq h_1 h_2.
\]
\end{proof}

\subsection{Hyperbolic \texorpdfstring{$3$}{3}-manifolds}
A complete hyperbolic $3$-manifold is a quotient 
\[
M=\HH^3 / \Gamma
\]
by a discrete torsion-free subgroup $\Gamma \leqslant  \Isom(\HH^3)$ acting freely on~$\HH^3$.
By the Mostow--Prasad rigidity theorem, if $M$ has finite volume, then the subgroup~$\Gamma$ is unique up to conjugation.

\medskip

Closed (oriented) geodesics of~$M$ correspond to conjugacy classes of loxodromic elements of~$\Gamma$.
More precisely, if $\gamma \in \Gamma$ is loxodromic, its invariant axis in~$\HH^3$ projects to a closed geodesic in~$M$ whose length equals the translation length (or minimal displacement)
\[
\ell_\gamma = \inf_{p \in \HH^3} d(p,\gamma.p)
\]
which is attained on the axis of~$\gamma$.

\medskip

The translation length of a loxodromic element can be expressed in terms of the trace; see~\cite[Lemma~5.1]{gen}.
(Recall that the trace of an element of~$\PSL(2,\C)$ is defined only up to sign.)

\begin{lemma}
The translation distance~$\ell_\gamma$ of a loxodromic isometry~$\gamma \in \PSL(2,\C) \rtimes \Z/2\Z$ is given by the trace formula
\begin{equation} \label{eq:trace}
2 \cosh\left(\frac{\ell_\gamma}{2} \right) = 
\begin{cases}
\left| \frac{1}{2} \tr(\gamma) -1 \right| + \left| \frac{1}{2} \tr(\gamma) +1 \right| & \text{ if } \gamma \text{ is orientation-preserving,} \\
\sqrt{\tr(\gamma^2) + 2} & \text{ if } \gamma \text{ is orientation-reversing.}
\end{cases}
\end{equation}
\end{lemma}

\subsection{Cusps of noncompact hyperbolic \texorpdfstring{$3$}{3}-manifolds}
A complete hyperbolic $3$-manifolds $M=\HH^3/\Gamma$ of finite volume is noncompact if and only if $\Gamma$ has parabolic elements (and not just loxodromic elements).
In this case, there are finitely many conjugacy classes of parabolic stabilizers
\[
\Gamma_\xi = \{ \gamma \in \Gamma \mid \gamma.\xi=\xi \}
\]
where $\xi \in \partial \HH^3$ is the fixed point of a parabolic element of~$\Gamma$.
These conjugacy classes correspond to the ends of~$M$, which are represented geometrically by cusps.

Geometrically, a cusp~$C$ of~$M$ is obtained as the quotient of a horoball~$H_\xi$ centered at a parabolic fixed point~$\xi$ by its stabilizer:
\[
C=H_\xi/\Gamma_\xi.
\]
Up to conjugation in~$\Isom(\HH^3)$, we may assume that $\xi=\infty$.
Then the parabolic subgroup~$\Gamma_\xi$ acts cocompactly by Euclidean isometries on the horosphere $\partial H_\xi = \{ t = t_0 \}$ with the induced Euclidean metric.
The  boundary~$\partial C$ of the cusp, which coincides with the quotient~$\partial H_\xi / \Gamma_\xi$, is therefore either a flat torus or a flat Klein bottle.

The cusp~$C$ is isometric to the warped product $\partial C \times [t_0,\infty)$ with the metric $\frac{1}{t^2} (g_{\rm eucl} + dt^2)$.
In particular, we have the following area--volume relation for cusps
\begin{equation} \label{eq:AV}
\vol(C) = \frac{1}{2} \area(\partial C).
\end{equation}

Using B\"or\"oczky's horoball packing density theorem~\cite{boro}, the following cusp density estimate was obtained in~\cite{meyerhoff} for a single cusp and extended to multiple cusps in~\cite{adams_cusps}.
As customary, when counting cusps, we count the number of disjoint cusps, that is, the number of ends.

\begin{lemma} \label{lem:density}
Let $C_1,\dots,C_m$ be $m$ cusps of~$M$ with pairwise disjoint interiors.
Then
\begin{equation} \label{eq:density}
\vol(M) \geq \frac{2V}{\sqrt{3}} \left( \vol(C_1) + \cdots + \vol(C_m) \right)
\end{equation}
where $V$ is the volume of an ideal regular tetrahedron in~$\HH^3$.
\end{lemma}

In the universal cover, the cusps of~$M$ correspond to $\Gamma$-orbits of horoballs centered at parabolic fixed points.
Maximal cusps correspond to horoball systems whose interiors are pairwise disjoint but such that some of them are tangent.

In particular, a maximal cusp~$C$ is self-tangent.
In this case, we denote by~$\partial C$ the quotient~$\partial H_\xi/\Gamma_\xi$ of a horoball~$H_\xi$ covering~$C$.
With this convention, maximal cusp boundary~$\partial C$ is a flat torus or a Klein bottle without self-intersection.

\subsection{Waist size of maximal cusps}
Let $M=\HH^3/\Gamma$ be a noncompact complete hyperbolic $3$-manifolds of finite volume.
Let $C$ be a maximal cusp in~$M$.
By maximality, there exist two tangent horoballs in~$\HH^3$ covering~$C$.
After a possible conjugation in~$\Isom(\HH^3)$, we will generally assume that these two horoballs coincide with
\[
H_\infty=\{t \geq 1 \} \quad \mbox{ and } \quad H_0=\{x^2+y^2+(t-\tfrac{1}{2})^2 \leq \tfrac{1}{4} \}
\]
which are centered at~$\infty$ and~$0$, and are tangent at $p_0=(0,0,1)$.

\medskip

The notion of waist size was introduced in~\cite{adams_waist} for orientable cusps.

\begin{definition} \label{def:waist}
The \emph{waist size} of a maximal cusp~$C$, denoted by~$w(C)$, is defined as the systole of the flat torus or Klein bottle corresponding to the cusp boundary~$\partial C$.
That is,
\[
w(C) = \sys(\partial C) =  \min_{p \in \partial H} \min_{\gamma \in \Gamma_H \setminus \{ \id \}} d_{\partial H}(p,\gamma.p)
\]
where $H$ is any horoball covering~$C$ and $\Gamma_H \leqslant \Gamma$ is the (parabolic) stabilizer of~$H$ in~$\Gamma$.

If $\partial C$ is a flat torus, it is enough to take
\[
w(C) = \sys(\partial C) =  \min_{\gamma \in \Gamma_H \setminus \{ \id \}} d_{\partial H}(p,\gamma.p)
\]
where $p$ is any point in~$\partial H$.
\end{definition}

The waist size of~$C$ provides a lower bound on the volume of the cusp.
Indeed, Minkowski's inequality asserts that
\begin{equation} \label{eq:minkowski}
\area(\T^2) \geq \frac{\sqrt{3}}{2} {\sys(\T^2)}^2
\end{equation}
with equality if and only if $\T^2$ is a hexagonal torus.
Similarly, for every flat Klein bottle~$K$
\begin{equation} \label{eq:square}
\area(K) \geq {\sys(K)}^2
\end{equation}
with equality if and only if $K$ is a square Klein bottle.
Combining these inequalities with~\eqref{eq:AV} yields the waist--volume estimates
\begin{equation} \label{eq:waist-volume}
\vol(C) = \frac{1}{2} \area(\partial C) \geq \frac{\sqrt{3}}{4} w(C)^2 \quad \mbox{ if } C \mbox{ is orientable}
\end{equation}
and
\begin{equation} \label{eq:waist-volume2}
\vol(C) = \frac{1}{2} \area(\partial C) \geq \frac{1}{2} w(C)^2 \quad \mbox{ if } C \mbox{ is nonorientable}.
\end{equation}

Adams proved that the waist size of a maximal cusp admits a uniform lower bound and characterized the inequality case; see~\cite{adams_waist} for the first estimate and~\cite{adams_waist2} for the second estimate.

\begin{theorem} \label{theo:waist}
Let $M$ be a noncompact complete hyperbolic $3$-manifolds of finite volume.
Let $C$ be a maximal cusp of~$M$.
Then
\begin{equation} \label{eq:w>1}
w(C) \geq 1
\end{equation}
with equality if and only if $M$ is isometric to the figure-eight knot complement (and $C$ is its unique maximal cusp).

Furthermore, if $M$ is not isometric to the figure-eight knot complement, then
\begin{equation} \label{eq:w>2}
w(C) \geq \sqrt[4]{2}
\end{equation}
with equality if and only if $M$ is isometric to either the $5_2$ knot complement or the manifold obtained by~$(2,1)$ surgery on the Whitehead link.
\end{theorem}

The argument giving the waist size lower bound~\eqref{eq:w>1} also applies to a collection~$(C_i)_{1 \leq i \leq m}$ of cusps where each cusp $C_i$ is maximal in~$M \setminus (\cup_{j \neq i} C_j)$.
In this case, either $C_i$ is self-tangent (that is, $C_i$ is a maximal cusp) or $C_i$ is tangent to~$\cup_{j \neq i} C_j$.

\begin{lemma} \label{lem:waist_collection}
Let $M$ be a noncompact complete hyperbolic $3$-manifolds of finite volume.
Let $(C_i)_{1 \leq i \leq m}$ be a collection of cusps where each cusp $C_i$ is maximal in~$M \setminus (\cup_{j \neq i} C_j)$ for every $1 \leq i \leq m$.
Then,
\[
w(C_i) \geq 1.
\]
\end{lemma}

\begin{proof}
By maximality, for every $1 \leq i \leq m$, there exists a horoball $H_i$ covering~$C_i$ which is tangent to a horoball~$H_j$ covering~$C_j$ for some $j=1,\dots,m$ (possibly with $j=i$).
After a possible conjugation in~$\Isom(\HH^3)$, we can assume that $H_i$ coincides with the horoball~$H_\infty$ centered at~$\infty$ of height~$1$.
In this case, the horoball~$H_j$, which is tangent to~$H_i=H_\infty$ at some point~$p \in \partial H_\infty$, is also of height~$1$.
Now, every parabolic isometry~$\tau \in \Gamma_\infty$ sends~$H_j$ to a horoball~$\tau.H_j$ of height~$1$, tangent to~$H_\infty$ at~$\tau.p$, such that the interiors of~$H_j$ and~$\tau.H_j$ are disjoint.
It follows that the Euclidean distance on~$\partial H_\infty$ between~$p$ and~$\tau.p$ is at least~$1$.
Hence, $w(C_i) \geq 1$.
\end{proof}


\section{Systole of closed orientable hyperbolic \texorpdfstring{$3$}{3}-manifolds}

In this section, we show an optimal systolic inequality similar to Gendulphe's for \emph{closed} orientable hyperbolic $3$-manifolds, instead of cusped hyperbolic $3$-manifolds, and characterize the first three optimal manifolds.

\medskip

These optimal manifolds coincide with the first three closed orientable hyperbolic $3$-manifolds of minimal volume listed below.

\begin{example}
\mbox{ }
\begin{itemize}
\item The Weeks--Matveev--Fomenko manifold~$m003(2,1)$ is the closed orientable hyperbolic $3$-manifold of the smallest volume; see~\cite{GMM}.
\item The Meyerhoff manifold~$m004(5,1)$ is the closed orientable hyperbolic $3$-manifold of the second smallest volume; see~\cite{GHMTY}.
\item ${\rm Vol}3=m007(3,1)$ is the closed orientable hyperbolic $3$-manifold of the third smallest volume; see~\cite{GHMTY}.
\end{itemize}
These three hyperbolic $3$-manifolds are arithmetic, and explicit formulas for their volume can be obtained from Borel's volume formula.
See \cite{CFJR}, \cite{Chinburg}, \cite{JR} for details.
Exact expressions for their systole can likewise be derived by identifying candidate systolic elements in the holonomy representations described in the above references, and then verifying, using SnapPy~\cite{snappy}, that the elements realize the shortest closed geodesics.
Numerical approximations for the volume and the systole are listed in Table~\ref{table:3} below.
\end{example}

\forget

The systole and volume of the first manifold in the list can be computed.

\begin{proposition}
Let $M$ be the Weeks--Matveev--Fomenko manifold~$m003(2,1)$.
Then
\[
\cosh \!\left(\tfrac{1}{2} \sys(M) \right) = 
\]
and 
\[
\vol(M) = \frac{3 \cdot 23^{\frac{3}{2}} \zeta_k(2)}{4 \pi^4} = 0.94270...
\]
where $\zeta_k$ denotes the Dedekind zeta function of the field~$k=\Q(\theta)$ with $\theta^3-\theta+1=0$.

In particular, the systole is approximately~$0.86255...$
\end{proposition}

\forgotten

We have the following optimal systole--volume estimates.

\begin{theorem} \label{theo:closed_sys}
Let $M$ be a closed orientable complete hyperbolic $3$-manifold.
Then
\[
\cosh \!\left(\tfrac{1}{2} \sys(M) \right) \leq C_1 \vol(M)
\]
with equality if and only if $M$ is isometric to the Weeks--Matveev--Fomenko manifold~$m003(2,1)$.
The multiplicative constant is approximately~$C_1=1.10641...$

Furthermore, the following holds.
\begin{enumerate}
\item If $M$ is not isometric to the Weeks--Matveev--Fomenko manifold, then 
\[
\cosh \!\left(\tfrac{1}{2} \sys(M) \right) \leq C_2 \vol(M)
\]
with equality if and only if $M$ is isometric to~${\rm Vol}3=m007(3,1)$.
The multiplicative constant is approximately~$C_2=1.07165...$
\item If $M$ is isometric neither to the Weeks--Matveev--Fomenko manifold nor to~${\rm Vol}3$, then 
\[
\cosh \!\left(\tfrac{1}{2} \sys(M) \right) \leq C_3 \vol(M)
\]
with equality if and only if $M$ is isometric to the Meyerhoff manifold~$m004(5,1)$.
The multiplicative constant is approximately~$C_3=1.06184...$
\item In all other cases,
\[
\cosh \!\left(\tfrac{1}{2} \sys(M) \right) < 1.1011 \, \vol(M)
\]
\end{enumerate}
\end{theorem}

\begin{proof}
Suppose that 
\begin{equation} \label{eq:closed_contradiction}
\cosh \!\left(\tfrac{1}{2} \sys(M) \right) \geq 1.1011 \, \vol(M).
\end{equation}

Let us show that $M$ is one of the three closed orientable hyperbolic $3$-manifolds $m003(2,1)$, $m004(5,1)$, $m007(3,1)$ with the smallest volume.

Every (open) ball $B(\tfrac{1}{2} \sys(M))$ of radius~$\tfrac{1}{2} \sys(M)$ in~$M$ is isometric to any ball of the same radius in the hyperbolic $3$-dimensional space~$\HH^3$.
Hence,
\begin{equation} \label{eq:closed_ball}
\vol(M) \geq \vol B\!\left(\tfrac{1}{2} \sys(M)\right) = \pi  \left( \sinh(\sys(M)) - \sys(M) \right).
\end{equation}

Together, the inequalities~\eqref{eq:closed_contradiction} and~\eqref{eq:closed_ball} yield
\[
\cosh \!\left(\tfrac{1}{2} \sys(M) \right) \geq 1.1011 \, \pi \left( \sinh(\sys(M)) - \sys(M) \right).
\]

This implies
\begin{equation} \label{eq:closed_sys<1}
\sys(M) \leq 1.24428.
\end{equation}

We could apply the inequality~\eqref{eq:closed_contradiction} again to deduce that
\[
\vol(M) \leq 1.0897.
\]
But this bound is not sharp enough and does not allow us to conclude that $M$ is one of the three closed orientable hyperbolic $3$-manifolds with the smallest volume.
Indeed, recall that these three manifolds have volume less than~$1.01495$; see Table~\ref{table:3}

\begin{table}[hbpt]
\begin{center}
\begin{tabular}{ |c|c|c|c| } 
 \hline
 & $\vol$ & $\sys$ & $\cos(\tfrac{1}{2} \sys) \big/ \vol$ \\
  \hline
$m003(2,1)$ & $0.94270...$ & $0.58460...$ & $1.10641...$ \\
  \hline
$m004(5,1)$ & $0.98136...$ & $0.57808...$ & $1.06184...$ \\
   \hline
$m007(3,1)$ & $1.01494...$ & $0.83144...$ & $1.07165...$ \\ 
 \hline
\end{tabular}
\bigskip
\end{center}
\caption{Systole and volume of the three closed orientable hyperbolic $3$-manifolds with the smallest volume.} \label{table:3}
\end{table}

Now, without loss of generality, we can assume that $M$ is not~${\rm Vol}3$, otherwise we are done.
By~\cite{GT15}, every closed orientable hyperbolic $3$-manifold distinct from~${\rm Vol}3$ admits an embedded tube~$T$ of radius~$\frac{1}{2} \log(3)$ along a closed geodesic.
Recall that the volume of an embedded tube of radius~$R$ along a closed geodesic of length~$L$ is equal to $\pi  L  \sinh^2(R)$; see~\cite[\S IX.3]{fenchel}.

In our case, with $L \geq \sys(M)$ and $R=\frac{1}{2} \log(3)$, we deduce that
\[
\vol(M) \geq \vol(T) \geq \frac{\pi}{3} \, \sys(M).
\]

Combined with~\eqref{eq:closed_contradiction}, we obtain
\[
\cosh \!\left(\tfrac{1}{2} \sys(M) \right) \geq 1.011 \, \frac{\pi}{3} \, \sys(M)
\]
which leads to the following alternative
\[
\sys(M) \leq  0.97162 \quad \mbox{ or } \quad \sys(M)  \geq 4.78585.
\]
By the systolic bound~\eqref{eq:closed_sys<1}, only the inequality
\[
\sys(M) \leq  0.97162
\]
can hold.
Applying the inequality~\eqref{eq:closed_contradiction} again, we deduce that
\[
\vol(M) \leq 1.01749.
\]

The closed orientable hyperbolic $3$-manifolds of volume at most~$1.01749$ are precisely $m003(2,1)$, $m004(5,1)$ and $m007(3,1)$; see~\cite{GHMTY}.

Now, SnapPy~\cite{snappy} provides approximate values for the volume and the systole of these three manifolds; see Table~\ref{table:3}.

It follows that the ratio
\begin{equation} \label{eq:closed_ratio}
\frac{\cosh \!\left(\tfrac{1}{2} \sys(M) \right)}{\vol(M)}
\end{equation}
with the largest value is attained by the Weeks--Matveev--Fomenko manifold~$m003(2,1)$, while the second and third largest values are attained by the manifold~${\rm Vol}3=m007(3,1)$ and the Meyerhoff manifold~$m004(5,1)$.
All other closed orientable hyperbolic $3$-manifolds have a ratio~\eqref{eq:closed_ratio} less than~$1.1011$.
\end{proof}

\begin{remark}
Note that the ratio~\eqref{eq:closed_ratio} is not a monotone function of the volume.
\end{remark}

\section{Systole of noncompact orientable hyperbolic \texorpdfstring{$3$}{3}-manifolds}

In this section, we establish an optimal systole--volume estimate for \emph{orientable} cusped hyperbolic $3$-manifolds, which extends Gendulphe's systolic inequality.

\subsection{The Adams--Reid isometry}
Let $M=\HH^3/\Gamma$ be a noncompact orientable complete hyperbolic $3$-manifold.
Fix a maximal cusp~$C$ in~$M$.
By maximality, there exist two tangent horoballs in~$\HH^3$ covering~$C$.
After a possible conjugation in~$\Isom(M)$, we can assume that these two horoballs coincide with the horoballs $H_\infty$ and~$H_0$ centered at~$\infty$ and~$0$ that are tangent at $p_0=(0,0,1)$.

By construction, there exists an isometry $\gamma \in \Gamma$ which sends~$H_0$ to~$H_\infty$.
Observe that $\gamma$ does not lie in the parabolic stabilizer~$\Gamma_\infty \leqslant \Gamma$ of~$H_\infty$, \ie, $\gamma \notin \Gamma_\infty$.
The isometry~$\gamma$ takes~$H_\infty$ to a horoball~$H_\omega$ centered at some point~$\omega \in \partial \HH^3$.
Note that $\omega$ is different from~$0$ and~$\infty$.
The horoball~$H_\omega$ is tangent to~$H_\infty$ at~$\gamma .p_0$.
Thus, both $p_0$ and~$\gamma.p_0$ lie in the horizontal Euclidean plane~$\partial H_\infty$ of height~$1$.

As observed in~\cite{adams_waist}, the horoballs~$H_0$ and~$H_\omega$ lie in different $\Gamma_\infty$-orbit.
Otherwise, there would exist $\tau \in \Gamma_\infty$ such that $\tau.\omega=0$.
The product~$\gamma \tau$ would then switch~$H_0$ and~$H_\omega$, and thus would fix their tangent point, which is impossible for a nontrivial element of~$\Gamma$.
Thus, the points~$p_0$ and~$\gamma.p_0$ project to two different points in the maximal cusp boundary $\T^2 = \partial C = \partial H_\infty / \Gamma_\infty$.

Since the horoballs~$H_0$ and~$H_\omega$ have disjoint interiors and are of Euclidean diameter~$1$, it follows from the previous observation that the flat torus~$\T^2$ contains two disjoint open disks of radius~$\frac{1}{2}$ centered at the projections of~$p_0$ and~$\gamma.p_0$.
In particular, we have
\[
d_{\partial H_\infty}(\gamma .p_0, p_0)  \geq 1.
\]

Replacing $\gamma$ with $\tau \gamma$ for some~$\tau \in \Gamma_\infty$ if necessary, we can further assume that $\gamma .p_0$ is closer to~$p_0$ than to any other point of~$\Gamma_\infty .p_0$ for the Euclidean metric~$d_{\partial H_\infty}$ on~$\partial H_\infty$ (and also for the hyperbolic metric~$d$ on~$\HH^3$).
That is,
\[
d_{\partial H_\infty}(\gamma .p_0, p_0) = d_{\partial H_\infty}(\gamma .p_0, \Gamma_\infty .p_0).
\]
This implies that
\begin{equation} \label{eq:D}
d_{\partial H_\infty}(\gamma .p_0, p_0) = d_{\T^2}(\pi(\gamma .p_0), \pi(p_0)) \leq \diam(\T^2)
\end{equation}
where $\pi:\HH^3 \to M$ is the canonical projection and $\T^2 = \partial C =  \partial H_\infty / \Gamma_\infty$ is the maximal cusp boundary of~$C$.

\medskip

We will refer to this isometry~$\gamma$ as the \emph{Adams--Reid isometry} of~$\Gamma$, since it was first introduced in~\cite{adams-reid} to bound the systole of cusped hyperbolic $3$-manifolds.

\subsection{Systole, diameter and area of a flat torus}

We will need the following straightforward result about a flat torus.

\begin{lemma} \label{lem:torus}
The systole, the diameter and the area of a flat torus~$\T^2$ satisfy the following inequality
\[
{\diam(\T^2)}^2 \leq \frac{1}{4} \, {\sys(\T^2)}^2 + \frac{1}{4} \left( \frac{\area(\T^2)}{\sys(\T^2)} \right)^2
\]
\end{lemma}

\begin{proof}
Every flat torus~$\T^2$ can be written as a quotient~$\T^2=\R^2/\Lambda$, where $\Lambda=\Z u + \Z v$ is a lattice of~$\R^2$ with $u=(s,0)$, $v=(t,h)$ and~$s=\sys(\T^2)$.
Every lattice point has the form~$(ms+nt,nh)$, where $m,n \in \Z$.
Take any point~$(a,b) \in \R^2$.
Choose $n \in \Z$ so that $|b-nh| \leq \frac{|h|}{2}$.
Then choose $m \in \Z$ so that $|a-(ms+nt)| \leq \frac{s}{2}$.
Thus, for every $(a,b) \in \R^2$, there exists a lattice point at distance at most
\[
\frac{1}{2} \sqrt{s^2 + h^2}
\]
from~$(a,b)$.
Note that $A=sh$, where $A=\area(\T^2)$.
Since the flat metric on~$\T^2$ is invariant by translations, this shows that
\[
\diam(\T^2) \leq \frac{1}{2} \sqrt{s^2 + \left( \frac{A}{s} \right)^2}.
\]
\end{proof}

\subsection{Loxodromic case}
Before making any assumption on the Adams--Reid isometry, let us start with the following observation.

\begin{lemma} \label{lem:3.07}
Assume that $M$ is different from the figure-eight knot complement and its sister.
If $\vol(M) \leq 3.07$ then
\[
\cosh \!\left(\tfrac{1}{2} \sys(M) \right) \leq 0.41030 \, \vol(M).
\]
\end{lemma}

\begin{proof}
The list of all noncompact orientable complete hyperbolic $3$-manifolds of volume at most~$3.07$ is given in~\cite{GHMTY}.
There are $14$ manifolds in this list, namely $m003$, $m004$, $m006$, $m007$, $m009$, $m010$, $m011$, $m015$, $m016$, $m017$, $m019$, $m022$, $m023$ and~$m026$.
We can check with SnapPy~\cite{snappy} that all of them, except the figure-knot complement~$m004$ and its sister~$m003$, satisfy the desired inequality.
\end{proof}

In the loxodromic case, we have the following estimate.

\begin{proposition} \label{prop:ARloxodromic}
Assume that $M$ is different from the figure-eight knot complement and its sister.
Suppose that the Adams--Reid isometry $\gamma$ is loxodromic.
Then
\[
\cosh \!\left(\tfrac{1}{2} \sys(M) \right) \leq 0.44647 \, \vol(M).
\]
\end{proposition}

\begin{proof}
By Lemma~\ref{lem:3.07}, we can assume that $\vol(M) \geq 3.07$.

Since $\gamma$ is loxodromic, we derive from the formula~\eqref{eq:hyp-eucl} and the inequality~\eqref{eq:D} that
\[
\cosh( \sys(M) ) \leq \cosh( d(p_0,\gamma .p_0)) = 1+\frac{1}{2} \, d_{\partial H_\infty}(p_0,\gamma .p_0)^2 \leq 1 + \frac{D^2}{2}
\]
where $D=\diam(\T^2)$.
By Lemma~\ref{lem:torus}, we have
\[
D^2 \leq \frac{s^2}{4} + \frac{A^2}{4 s^2}
\]
where $s=\sys(\T^2)$ and $A=\area(\T^2)$.
Therefore, 
\[
\cosh( \sys(M) ) \leq 1+\frac{s^2}{8} + \frac{A^2}{8s^2}.
\]

By the second waist size inequality~\eqref{eq:w>2} -- assuming that $M$ is different from the figure-eight knot complement and its sister -- and Minkowski's inequality~\eqref{eq:minkowski}, we obtain
\[
\sqrt[4]{2} \leq s \leq \frac{\sqrt{2}}{\sqrt[4]{3}} \, \sqrt{A}.
\]
On this interval, the function~$f$ defined by $f(s) = 1+\frac{s^2}{8} + \frac{A^2}{8s^2}$ is first decreasing, then increasing.
It follows that 
\[
\cosh( \sys(M) ) \leq \max \left\{ f(\sqrt[4]{2}), f\left(\frac{\sqrt{2}}{\sqrt[4]{3}} \, \sqrt{A} \right) \right\}.
\]
We need to consider the following two cases.

Suppose that $\cosh( \sys(M) )$ is bounded by~$f(\sqrt[4]{2})$.
By the area--volume cusp relation~\eqref{eq:AV}, we obtain
\[
\cosh( \sys(M) ) \leq 1 + \frac{\sqrt{2}}{8} + \frac{A^2}{8 \sqrt{2}} \leq 1 + \frac{\sqrt{2}}{8} + \frac{3}{8 \sqrt{2} \, V^2} \, \vol(M)^2.
\]
Using the relation $\cosh(2t)=2 \cosh(t)^2-1$, we derive 
\[
\cosh \!\left(\tfrac{1}{2} \sys(M) \right) \leq \sqrt{1 + \frac{\sqrt{2}}{16} + \frac{3}{16 \sqrt{2} V^2} \vol(M)^2} \leq 0.44647 \, \vol(M)
\]
where the last inequality holds for $\vol(M) \geq 3.07$.

Suppose that $\cosh( \sys(M) )$ is bounded by~$f\left(\frac{\sqrt{2}}{\sqrt[4]{3}} \, \sqrt{A} \right)$.
As previously, by the area--volume cusp relation~\eqref{eq:AV}, we obtain
\[
\cosh( \sys(M) ) \leq 1 + \frac{7}{16 \sqrt{3}} \, A \leq 1 + \frac{7}{16 V} \, \vol(M).
\]
Thus,
\[
\cosh \!\left(\tfrac{1}{2} \sys(M) \right) \leq \sqrt{1 + \frac{7}{32 V} \vol(M)} \leq 0.41989 \, \vol(M)
\]
where the last inequality holds for $\vol(M) \geq 3.07$.
\end{proof}

\begin{remark}
The first waist size estimate~$s \geq 1$ -- see inequality~\eqref{eq:w>1} in Theorem~\ref{theo:waist} -- only allows us to derive the first inequality of Theorem~\ref{theo:sys}.
We need the second waist estimate~$s \geq \sqrt[4]{2}$ -- see inequality~\eqref{eq:w>2} in Theorem~\ref{theo:waist} -- for manifolds different from the figure-eight knot complement and its sister in order to derive the full version of Theorem~\ref{theo:sys}.
\end{remark}

\subsection{Parabolic case}

For the rest of this section, we will assume that the Adams--Reid isometry~$\gamma$ is parabolic.
In this case, the isometry $\gamma \in \Gamma \setminus \Gamma_\infty$ satisfies $\gamma.0=\infty$, $\gamma.\infty=\omega$ and~$\tr(\gamma)=\pm 2$.
Therefore, it can be identified with the following matrix of~$\SL(2,\C)$
\[
\gamma = 
\begin{pmatrix}
2 & - \frac{\omega}{2} \\
\frac{2}{\omega} & 0
\end{pmatrix}.
\]
By Poincar\'e's extension formula~\eqref{eq:extension}, the isometry~$\gamma$ sends $p_0=(0,0,1)$ to
\[
\gamma .p_0 = \left( \omega, \left| \tfrac{\omega}{2} \right|^2 \right) \in \partial H_\infty.
\]
This implies that $|\omega| = 2$.
Observe also that $\frac{\omega}{2}$ is the fixed point at infinity of~$\gamma$.

\medskip

Consider a (nontrivial) parabolic isometry $\sigma \in \Gamma_\infty$ of minimal translation distance on~$\partial H_\infty$.
That is,
\[
d_{\partial H_\infty}(p, \sigma .p) = \min_{\tau \in \Gamma_\infty \setminus \{ \id \}} d_{\partial H_\infty}(p, \tau .p)
\]
for any $p \in \partial H_\infty$.
Note that the translation distance $d_{\partial H_\infty}(p, \sigma .p)$ does not depend on the point~$p \in \partial H_\infty$.
This means that the parabolic isometry~$\sigma$ realizes the waist size of~$C$ (\ie, the systole of the flat torus~$\T^2 = \partial C = \partial H_\infty / \Gamma_\infty$).
That is,
\begin{equation} \label{eq:s}
d_{\partial H_\infty}(p, \sigma .p) = \sys(\T^2).
\end{equation}
The parabolic isometry~$\sigma \in \Gamma$ satisfies $\sigma . \infty=\infty$, $\tr(\sigma)=\pm 2$ and $d_{\partial H_\infty}(x,\sigma .x) = s$, where $s=\sys(\T^2)$.
Therefore, it can be identified with the following matrix of~$\SL(2,\C)$
\[
\sigma = 
\begin{pmatrix}
1 & \xi \\
0 & 1
\end{pmatrix}.
\]

Observe that the products
\[
\sigma^{\pm 1} \gamma = 
\begin{pmatrix}
2 \pm 2 \frac{\xi}{\omega} & - \frac{\omega}{2} \\
\frac{2}{\omega} & 0
\end{pmatrix}
\]
cannot both be parabolic.

\medskip

We consider the following dichotomy: both products are loxodromic or one of them is parabolic.

\medskip

In the first case, we have

\begin{proposition} \label{prop:ARparabolic1}
Suppose $\gamma$ is parabolic and one the products~$\sigma \gamma$ or~$\sigma^{-1} \gamma$ is also parabolic.
Then
\[
\cosh \!\left(\tfrac{1}{2} \sys(M) \right) \leq \frac{3}{8 V} \vol(M).
\]
The multiplicative constant is approximately~$0.36947...$
\end{proposition}

\begin{proof}
Replacing $\sigma$ with its inverse if necessary, we can assume that $\sigma \gamma$ is parabolic.
In this case,
\[
\tr(\sigma \gamma) = 2+2 \, \frac{\xi}{\omega} = -2.
\]
That is, $\xi=-2\omega$ and $s=|\xi|=4$.

It follows that
\[
\sigma^{-1} \gamma = 
\begin{pmatrix}
6 & -\frac{\omega}{2} \\
\frac{2}{\omega} & 0
\end{pmatrix}
\]
has trace~$6$.
By the trace formula~\eqref{eq:trace}, we derive
\[
\cosh \left( \frac{\ell_{\sigma^{-1} \gamma}}{2} \right) \leq 3.
\]

Since $\sys(\T^2)=s=4$, we deduce from the cusp density estimate~\eqref{eq:density} and the waist--volume inequality~\eqref{eq:waist-volume} that
\[
\vol(M) \geq \frac{V}{2} s^2 = 8 V.
\]
Hence,
\[
\cosh \!\left(\tfrac{1}{2} \sys(M) \right) \leq \frac{3}{8 V} \vol(M).
\]
\end{proof}

In the second case, we have

\begin{proposition} \label{prop:ARparabolic2}
Suppose that $\gamma$ is parabolic and both products~$\sigma \gamma$ and~$\sigma^{-1} \gamma$ are loxodromic.
Then
\[
\cosh \!\left(\tfrac{1}{2} \sys(M) \right) \leq 0.52086 \, \vol(M).
\]
\end{proposition}

\begin{proof}
Recall that
\[
\sigma^{\pm 1} \gamma = 
\begin{pmatrix}
2 \pm 2 \frac{\xi}{\omega} & -\frac{\omega}{2} \\
\frac{2}{\omega} & 0
\end{pmatrix}.
\]
where $|\omega|=2$.
Since both $\sigma \gamma$ and~$\sigma^{-1} \gamma$ are loxodromic, we obtain from the trace formula~\eqref{eq:trace} that
\[
\cosh \!\left(\tfrac{1}{2} \sys(M) \right) \leq \cosh \left( \frac{\ell_{\sigma^{\pm 1} \gamma}}{2} \right) = \frac{s}{4} +\left| 1 \pm \frac{\xi}{2\omega} \right|
\]
where $s=|\xi|=\sys(\T^2)$.

Applying the inequality
\[
\min |a \pm b| \leq \sqrt{|a|^2+|b|^2}
\]
which can be derived by measuring the distance from~$0$ to~$a \pm b$ in the triangle $0,a,a \pm b$ with an acute angle at~$a$, we deduce that
\begin{equation} \label{eq:s/4}
\cosh \!\left(\tfrac{1}{2} \sys(M) \right) \leq \frac{s}{4} + \sqrt{1+\left( \frac{s}{4} \right)^2}.
\end{equation}

Now, recall that the isometries~$\gamma$ and~$\sigma$ are defined from the horoballs covering a maximal cusp~$C$ of~$M$, where the tangency points between the horoballs correspond to the self-tangency points of~$C$.
We will consider the following three cases.

\medskip

\emph{Case~1.} 
Suppose that $M$ has a single cusp.
We follow the argument of~\cite[Proposition~5.4]{gen}.
As previously observed, the point~$\frac{\omega}{2}$ is the (only) fixed point of the parabolic isometry of~$\gamma$.
This means that there is a horoball~$H_{\omega/2}$ centered at~$\frac{\omega}{2}$ which covers a maximal cusp of~$M$.
By uniqueness of the cusp, this maximal cusp is necessarily~$C$.
Now, since $H_{\omega/2}$ and~$H_\infty$ have disjoint interiors, the height of~$H_{\omega/2}$ is at most~$1$.
Thus, the horoballs~$H^+_{\omega/2}$ of height~$1$ centered at~$\frac{\omega}{2}$, tangent to both~$H_0$ and~$H_\omega$, contains~$H_{\omega/2}$.
Therefore, the translation distance of~$\gamma$ on~$\partial H_{\omega/2}$ is bounded by that on~$\partial H^+_{\omega/2}$.
By definition, the parabolic isometry~$\gamma$ takes the tangent point~$z$ of~$H_0$ and~$H^+_{\omega/2}$ to the tangent point~$\gamma. z$ of~$H_\infty$ and~$H^+_{\omega/2}$.
Sending~$H^+_{\omega/2}$ to~$H_\infty$ (and the horoballs~$H_0$ and~$H_\infty$ to horoballs of height~$1$) by an isometry of~$\HH^3$, we observe that 
\[
d_{\partial H^+_{\omega/2}}(p,\gamma.p) = 1.
\]
It follows that the translation distance of~$\gamma$ on~$\partial H_{\omega/2}$ is at most~$1$ (and at least~$1$ by Theorem~\ref{theo:waist}).
Thus, the waist size of~$C$, viewed as the quotient of~$H_{\omega/2}$ by~$\Gamma_{\omega/2}$, is equal to~$1$.
By the equality case of Theorem~\ref{theo:waist}, this implies that $M$ is isometric to the complement of the figure-eight knot, which is excluded.

\medskip

\emph{Case~2.} 
Suppose that $M$ has exactly two cusps.
Let $C'$ be a maximal cusp in~$M \setminus C$.
By Lemma~\ref{lem:waist_collection} and the waist--volume inequality~\eqref{eq:waist-volume}, the volume of~$C'$ is at least~$\frac{\sqrt{3}}{4}$.
It follows from the cusp density estimate~\eqref{eq:density} that
\[
\vol(M) \geq \frac{2V}{\sqrt{3}} \left( \vol(C) + \vol(C') \right) \geq \frac{2V}{\sqrt{3}} \vol(C) + \frac{V}{2}.
\]
Therefore,
\[
s^2 \leq \frac{2}{\sqrt{3}} \area(\partial C) \leq \frac{4}{\sqrt{3}} \vol(C) \leq \frac{2}{V} \vol(M) -1.
\]
Combined with~\eqref{eq:s/4}, this yields
\begin{equation} \label{eq:case2}
\cosh \!\left(\tfrac{1}{2} \sys(M) \right) \leq \frac{1}{4} \sqrt{\frac{2}{V} \vol(M) -1} + \sqrt{\frac{15}{16} + \frac{1}{8V} \vol(M)}.
\end{equation}

By~\cite{agol_2cusps}, the volume of an orientable complete hyperbolic $3$-manifold with two cusps is at least four times Catalan's constant.
(Recall that Catalan's constant~$\beta(2)$, where $\beta$ is the Dirichlet beta function, is the alternating sum of the reciprocals of the odd square numbers.)
Thus, 
\[
\vol(M) \geq 4 \beta(2) = 3.66386...
\]
Combined with~\eqref{eq:case2}, we derive
\[
\cosh \!\left(\tfrac{1}{2} \sys(M) \right) \leq 0.49182 \, \vol(M).
\]

\medskip

\emph{Case~3.} 
Suppose that $M$ has at least three cusps.
First, choose a cusp~$C'$ which is maximal in~$M \setminus C$, then choose another cusp~$C''$ which is maximal in~$M \setminus (C \cup C')$.
By Lemma~\ref{lem:waist_collection} and the waist--volume inequality~\eqref{eq:waist-volume}, the volume of~$C' \cup C''$ is at least~$\frac{\sqrt{3}}{2}$.
As previously, it follows from the cusp density estimate~\eqref{eq:density} that
\[
\vol(M) \geq \frac{2V}{\sqrt{3}} \left( \vol(C) + \vol(C') + \vol(C'') \right) \geq \frac{2V}{\sqrt{3}} \vol(C) + V.
\]
Therefore,
\[
s^2 \leq \frac{2}{\sqrt{3}} \area(\partial C) \leq \frac{4}{\sqrt{3}} \vol(C) \leq \frac{2}{V} \vol(M) -2.
\]
Combined with~\eqref{eq:s/4}, this yields
\[
\cosh \!\left(\tfrac{1}{2} \sys(M) \right) \leq \frac{1}{4} \sqrt{\frac{2}{V} \vol(M) -2} + \sqrt{\frac{7}{8} + \frac{1}{8V} \vol(M)}.
\]
Now, by Lemma~\ref{lem:3.07}, we can assume that $\vol(M) \geq 3.07$.
In this case, we derive
\[
\cosh \!\left(\tfrac{1}{2} \sys(M) \right) \leq 0.52086 \, \vol(M).
\]
\end{proof}

\subsection{Conclusion}

Combining Lemma~\ref{lem:3.07}, Proposition~\ref{prop:ARloxodromic}, Proposition~\ref{prop:ARparabolic1} and Proposition~\ref{prop:ARparabolic2} together, we obtain the following result.

\begin{theorem} \label{theo:baby-sys}
Let $M$ be a noncompact orientable complete hyperbolic $3$-manifold different from the figure-eight knot complement and its sister.
Then
\[
\cosh \!\left(\tfrac{1}{2} \sys(M) \right) \leq 0.52086 \, \vol(M).
\]
\end{theorem}

In order to derive the full statement of Theorem~\ref{theo:sys}, we simply need to examine the case of the figure-eight knot complement and its sister.
These two examples have been treated in~\cite[\S 8.3]{GR}. 

\begin{example} \label{ex}
\mbox{ }
\begin{enumerate}
\item The figure-eight knot complement~$m004$ is a noncompact orientable complete hyperbolic $3$-manifold.
It has volume~$2 V = 2.02988...$, and its systole satisfies
\[
\cosh \!\left(\tfrac{1}{2} \sys \right) = \frac{1+\sqrt{13}}{4} = 1.15138...
\]
\item The sister~$m003$ of the figure-eight knot complement -- obtained by $(-5,1)$ surgery on the Whitehead link complement -- is a noncompact orientable complete hyperbolic $3$-manifold.
It has volume~$2 V = 2.02988...$, and its systole satisfies
\[
\cosh \!\left(\tfrac{1}{2} \sys \right) = \frac{\sqrt{3}+\sqrt{7}}{4} = 1.09445...
\]
\end{enumerate}
\end{example}
 
Theorem~\ref{theo:sys} follows from Theorem~\ref{theo:baby-sys} and Example~\ref{ex}. \\

As a corollary, we can recover Gendulphe's result~\cite{gen} and improve the multiplicative constant when the hyperbolic manifold differs from the Gieseking manifold.

\begin{corollary} \label{coro:sys}
Let $M$ be a noncompact (orientable or nonorientable) complete hyperbolic $3$-manifold of finite volume.
Then
\[
\cosh \!\left(\tfrac{1}{2} \sys(M) \right) \leq \frac{1+\sqrt{13}}{4 V} \vol(M)
\]
with equality if and only if $M$ is isometric to the Gieseking manifold.
The multiplicative constant is approximately~$1.13443...$

Furthermore, if $M$ is not isometric to the Gieseking manifold, then 
\[
\cosh \!\left(\tfrac{1}{2} \sys(M) \right) \leq 1.04172 \, \vol(M).
\]
\end{corollary}

\begin{proof}
We can assume that $M$ is nonorientable, otherwise the inequalities follow from Theorem~\ref{theo:sys}.

Consider the orientable double cover~$\bar{M}$ of~$M$.
Note that $\vol(\bar{M}) = 2 \vol(M)$.
Since the covering map $\bar{M} \to M$ is locally isometric and $\pi_1$-injective, the systole of~$\bar{M}$ is no less than that of~$M$.
That is, $\sys(M) \leq \sys(\bar{M})$.
By Theorem~\ref{theo:sys}, we deduce that
\begin{equation} \label{eq:<<<}
\cosh \!\left(\tfrac{1}{2} \sys(M) \right) \leq \cosh \!\left(\tfrac{1}{2} \sys(\bar{M}) \right) \leq \frac{1+\sqrt{13}}{8 V} \vol(\bar{M}) \leq  \frac{1+\sqrt{13}}{4 V} \vol(M).
\end{equation}
Moreover, equality holds if and only if $\bar{M}$ is isometric to the figure-eight knot complement, which has volume~$2V$.
In this case, the manifold~$M$ has minimal volume~$V$ among all noncompact complete hyperbolic $3$-manifold.
It follows from~\cite{adams_minvol} that $M$ is the Gieseking manifold.
Conversely, the Gieseking manifold satisfies the equality case of~\eqref{eq:<<<}; see Example~\ref{ex}.

If the nonorientable~$M$ is not isometric to the Gieseking manifold, then its double orientable cover~$\bar{M}$ is not the figure-eight knot complement.
It is not its sister either, since the sister of the figure-eight knot complement does not cover any nonorientable hyperbolic manifold.
We conclude as previously
\[
\cosh \!\left(\tfrac{1}{2} \sys(M) \right) \leq \cosh \!\left(\tfrac{1}{2} \sys(M) \right) \leq 0.52086 \, \vol(\bar{M}) \leq  1.04172 \, \vol(M).
\]
\end{proof}

\section{Inradius and volume of noncompact hyperbolic \texorpdfstring{$3$}{3}-manifolds}

In this section, we establish optimal inradius--volume estimates for cusped hyperbolic $3$-manifolds both in the orientable and nonorientable cases; see Theorem~\ref{theo:inradius_orientable} and Theorem~\ref{theo:inradius_nonorientable}.
We also show that the extremal cases correspond to the figure-eight knot complement and the Gieseking manifold.

\subsection{Injectivity radius, inradius and orbit distances}

Recall that the \emph{inradius}~$\RR(M)$ of a complete Riemannian manifold~$M$ without conjugate points is defined as the radius of the largest embedded ball in~$M$.
That is,
\[
\RR(M) = \sup_{x \in M} \inj(x)
\]
where $\inj(x)$ is the injectivity radius of~$M$ at~$x$.

\medskip

Let us show the following well-known characterization of the injectivity radius and inradius in terms of orbit distances.

\begin{lemma} \label{lem:inrad-dist}
Let $M$ be a complete Riemannian manifold without conjugate points.
Denote by~$\Gamma$ the deck transformation group of~$M$ acting on the universal cover~$\bar{M}$.
For every $x \in M$ and every $p \in \bar{M}$ projecting to~$x$, we have
\[
\inj(x) = \frac{1}{2} \min_{\gamma \in \Gamma \setminus \{ {\rm id} \}} d(p,\gamma.p).
\]
In particular,
\[
\RR(M) = \frac{1}{2} \sup_{p \in \bar{M}} \min_{\gamma \in \Gamma \setminus \{ {\rm id} \}} d(p,\gamma.p).
\]
\end{lemma}

\begin{proof}
Denote by $\pi:\bar{M} \to M$ the canonical projection.

Let $r>\frac{1}{2} \min_{\gamma \in \Gamma \setminus \{ {\rm id} \}} d(p,\gamma.p)$.
Then there exists $\gamma \in \Gamma \setminus \{ \id \}$ with $d(p,\gamma.p) <2r$.
Hence the balls~$B(p,r)$ and~$B(\gamma.p,r)$ intersect.
Choose $q$ in the intersection.
By construction, $q$ and~$\gamma^{-1}.q$ both lie in~$B(p,r)$ and have the same image under~$\pi$.
Thus, $\pi$ is not injective on~$B(p,r)$.
Therefore, $\inj(x) \leq r$.
Hence,
\[
\inj(x) \leq \frac{1}{2} \min_{\gamma \in \Gamma \setminus \{ {\rm id} \}} d(p,\gamma.p).
\]

Let $r<\frac{1}{2} \min_{\gamma \in \Gamma \setminus \{ {\rm id} \}} d(p,\gamma.p)$.
Then for every $\gamma \in \Gamma \setminus \{ \id \}$, the balls $B(p,r)$ and~$B(\gamma.p,r)$ are disjoint.

Let $q_1, q_2 \in B(p,r)$ with $\pi(q_1)=\pi(q_2)$.
Then there exists $\gamma \in \Gamma$ such that $q_2=\gamma.q_1$.
Since $q_1 \in B(p,r)$, we have $q_2 \in B(\gamma.p,r)$.
Thus, $q_2 \in B(p,r) \cap B(\gamma.p,r)$.
This forces $\gamma=\id$ and hence~$q_2=q_1$.
Therefore $\pi$ is injective on~$B(p,r)$.

Since $M$ has no conjugate points, the injectivity radius at~$x$ coincides with the largest radius of a ball centered at~$p$ on which the covering map is injective.
Thus, $\inj(x) \geq r$.
Hence,
\[
\inj(x) \geq \frac{1}{2} \min_{\gamma \in \Gamma \setminus \{ {\rm id} \}} d(p,\gamma.p).
\]
\end{proof}

\subsection{Inradius of noncompact orientable hyperbolic \texorpdfstring{$3$}{3}-manifolds}

We can compute the inradius of the figure-eight knot complement.

\begin{proposition} \label{prop:inradius_orientable}
\mbox{ }
\begin{enumerate}
\item The inradius of the figure-eight knot complement~$M$ satisfies
\[
\RR(M) = \arcosh \left( \frac{\sqrt{22}}{4} \right).
\]
That is, $\RR(M) = 0.57940...$
\item The inradius of the sister of the figure-eight knot complement~$M$ satisfies
\[
\RR(M) > \arcosh \left( \frac{\sqrt{22}}{4} \right).
\]
\end{enumerate}
\end{proposition}

We will need the following straightforward result.

\begin{lemma}
The height of the center of the regular ideal simplex~$\Delta$ with ideal vertices $0$, $1$, $e^{i \frac{\pi}{3}}$ and~$\infty$ is equal to~$\frac{\sqrt{6}}{3}$.
\end{lemma}

\begin{proof}
Let $L$ be the vertical geodesic of~$\HH^3$ given by $p_t=(\frac{1}{2}, \frac{\sqrt{3}}{6},t)$ with $t >0$.
By symmetry, the center of~$\Delta$ is the unique point of~$L$ which is equidistant from the faces with vertices~$(0,1,\infty)$ and $(0,1,e^{i\frac{\pi}{3}})$.

A direct computation shows that the distance between $p_t$ and the vertical face contained in the plane $y=0$ (corresponding to the face~$(0,1,\infty)$) is attained at the point $(\frac{1}{2},0,\sqrt{1+\frac{1}{12t^2}})$ and is equal to 
\begin{equation} \label{eq:center1}
\arsinh \left( \frac{1}{2 t \sqrt{3}} \right).
\end{equation}

Similarly, the face $(0,1,e^{i\frac{\pi}{3}})$ lies in the hemisphere of Euclidean radius~$\frac{\sqrt{3}}{3}$ centered at~$(\frac{1}{2}, \frac{\sqrt{3}}{6},0)$.
The hyperbolic distance between $p_t$ and this face is equal to
\begin{equation} \label{eq:center2}
\log \left( \sqrt{3} t \right).
\end{equation}

The distances in~\eqref{eq:center1} and~\eqref{eq:center2} coincide precisely when $t=\frac{\sqrt{6}}{3}$.
Thus, the center of~$\Delta$ is the point~$(\frac{1}{2}, \frac{\sqrt{3}}{6},\frac{\sqrt{6}}{3})$.
\end{proof}

Let us proceed to the proof of Proposition~\ref{prop:inradius_orientable}.

\begin{proof}[Proof of Proposition~\ref{prop:inradius_orientable}]
The figure-eight knot complement and its sister are obtained by gluing together two regular ideal simplices $A$ and~$B$ of different colors, each with vertices~$1$, $2$, $3$ and~$4$, along their faces according to the identifications given in Table~\ref{table:8}; see~\cite{KRR}. 
Their isometry groups are respectively isometric to the dihedral group~$D_4$ and $\Z_2 \times \Z_4$.
They both switch the simplices~$A$ and~$B$.
Moreover, the stabilizers of each simplex act transitively on the vertices of the two simplices.

\forget

\begin{table}[hbp]
\centering
\begin{tabular}{rccc} 
 & $A$ & & $B$ \\
$\sigma_1$: & $(1,2,3)$ & $\leftrightarrow$ & $(3,2,1)$ \\
$\sigma_2$: & $(1,2,4)$ & $\leftrightarrow$ & $(1,4,2)$ \\
$\sigma_3$: & $(1,3,4)$ & $\leftrightarrow$ & $(3,4,2)$ \\
$\sigma_4$: & $(2,3,4)$ & $\leftrightarrow$ & $(4,1,3)$ \\
 & & & 
\end{tabular}
\caption{Identifications for the figure-eight knot complement.} \label{table:8}
\end{table}

\begin{table}[hbp]
\centering
\begin{tabular}{ccc} 
$A$ & & $B$ \\
$(1,2,3)$ & $\leftrightarrow$ & $(4,1,2)$ \\
$(1,2,4)$ & $\leftrightarrow$ & $(3,4,1)$ \\
$(1,3,4)$ & $\leftrightarrow$ & $(1,3,2)$ \\
$(2,3,4)$ & $\leftrightarrow$ & $(2,4,3)$ \\
& &
\end{tabular}
\caption{Identifications for the sister of the figure-eight knot complement. ???NOT REQUIRED} \label{table:sister}
\end{table}

\forgotten

\begin{table}[hbp]
\parbox{.48\linewidth}{
\centering
\begin{tabular}{rccc} 
 & $A$ & & $B$ \\
$\sigma_1$: & $(1,2,3)$ & $\leftrightarrow$ & $(3,2,1)$ \\
$\sigma_2$: & $(1,2,4)$ & $\leftrightarrow$ & $(1,4,2)$ \\
$\sigma_3$: & $(1,3,4)$ & $\leftrightarrow$ & $(3,4,2)$ \\
$\sigma_4$: & $(2,3,4)$ & $\leftrightarrow$ & $(4,1,3)$ \\
 & & & 
\end{tabular}
}
\parbox{.48\linewidth}{
\centering
\begin{tabular}{ccc} 
$A$ & & $B$ \\
$(1,2,3)$ & $\leftrightarrow$ & $(4,1,2)$ \\
$(1,2,4)$ & $\leftrightarrow$ & $(3,4,1)$ \\
$(1,3,4)$ & $\leftrightarrow$ & $(1,3,2)$ \\
$(2,3,4)$ & $\leftrightarrow$ & $(2,4,3)$ \\
& &
\end{tabular}
}
\caption{Identifications for the figure-eight knot complement and its sister.}  \label{table:8}
\end{table}

Lifting this decomposition to the universal cover yields a tiling of~$\HH^3$ by regular ideal simplices, where two simplices of the same color are never adjacent.
Up to isometry, we can assume that the simplex~$A$ coincides with the regular ideal simplex~$\Delta$ with ideal vertices $v_1=0$, $v_2=1$, $v_3=e^{i \frac{\pi}{3}}$ and~$v_4=\infty$.
Denote by~$p_\Delta$ the center of~$\Delta$.
Define the Voronoi cell~$\mathcal{V}$ of~$p_\Delta$ as
\[
\mathcal{V}=\{ p \in \HH^3 \mid d(p,p_\Delta) \leq d(p,\gamma.p_\Delta) \mbox{ for every } \gamma \in \Gamma \}
\]
where $\Gamma$ is the deck transformation group of the figure-eight knot complement.
The Voronoi cell~$\mathcal{V}$ consists of~$\Delta$ together with the four cones over the faces of~$\Delta$ from the centers of the simplices of the tiling adjacent to~$\Delta$.
Thus, the minimal distance between~$p_\Delta$ and any other point of its $\Gamma$-orbit is equal to the distance between~$p_\Delta$ and the center~$p_{\Delta'}$ of a simplex~$\Delta'$ of the same color as~$\Delta$ sharing an edge with~$\Delta$.
Up to isometry, we may assume that this common edge is vertical.
Since the dihedral angle of a regular ideal simplex is equal to~$\frac{\pi}{3}$, this yields six possible choices for~$\Delta'$ (all of them work).
We may further assume that the ideal vertices of~$\Delta'$ are $1$, $2$, $1+e^{i \frac{\pi}{3}}$ and~$\infty$.
Thus, the centers~$p_\Delta$ and~$p_{\Delta'}$ have the same height~$\frac{\sqrt{6}}{3}$ and differ by a horizontal translation of Euclidean length~$1$.
Using the Euclidean--hyperbolic distance formula~\eqref{eq:hyp-eucl3}, we obtain
\[
\cosh d(p_\Delta,p_{\Delta'}) = \frac{7}{4}.
\]
Hence
\[
\min_{\gamma \in \Gamma \setminus \{ \id \}} \frac{1}{2} \, d(p_\Delta,\gamma.p_\Delta) = \frac{1}{2} \arcosh \left( \frac{7}{4} \right) = \arcosh \left( \frac{\sqrt{22}}{4} \right).
\]
This yields a lower bound on the inradius of both the figure-eight knot complement and its sister.

Assume that $M$ is the figure-eight knot complement.
Let us show that for every $p \in \HH^3$, there exists $\gamma \in \Gamma \setminus \{ \id \}$ such that 
\[
\frac{1}{2} \, d(p,\gamma.p) \leq \arcosh \left( \frac{\sqrt{22}}{4} \right).
\]

Up to isometry, we may assume that $p$ lies in the region of~$\Delta$ consisting of the points that are closer to the geodesic ray from~$p_\Delta$ to~$\infty$ than to any other geodesic ray from~$p_\Delta$ to another ideal vertex of~$\Delta$ (or at equal distance).
In particular, the height of~$p$ is at least~$\frac{\sqrt{6}}{3}$.

The composition $\gamma=\sigma_2^{-1} \circ \sigma_3$ of the two isometries of~$\Gamma$ corresponding to the second and third transformations in Table~\ref{table:8} is a transformation of~$\Gamma$ fixing~$\infty$ and sending $0$, $1$, $e^{i \frac{\pi}{3}}$ to $e^{i \frac{2 \pi}{3}}$, $e^{i \frac{\pi}{3}}$, $\sqrt{3} i$.
This parabolic isometry 
acts on~$\HH^3$ as a horizontal translation by the vector~$e^{i \frac{2 \pi}{3}}$.
By the Euclidean--hyperbolic distance formula~\eqref{eq:hyp-eucl3}, for every point~$p$ of height at least~$\frac{\sqrt{6}}{3}$, we have
\[
d(p,\gamma.p) \leq \arcosh \left( \frac{7}{4} \right)
\]
as required.

Assume that $M$ is the sister of the figure-eight knot complement.
Although we can show that  $\arcosh \left( \frac{\sqrt{22}}{4} \right)$ is a local maximum for the injectivity radius at~$p_\Delta$, the previous argument fails.
Instead, we can determine the isometries $\gamma_1,\dots,\gamma_k$ of~$\Gamma$ that send~$A$ to its nearby simplices.
This allows us to write 
\[
\min_{\gamma \in \Gamma \setminus \{ \id \}} d(p,\gamma.p) = \min_{1 \leq i \leq k} d(p,\gamma.p)
\]
for any point $p \in A$ (which can always be assumed by symmetry of~$M$).
Then, we decompose~$A$ into convex domains separated by hypersurfaces defined by the equations $d(p,\gamma_i.p) = d(p,\gamma_j.p)$ for $1 \leq i \neq j \leq k$.
By convexity of the distance function, the inradius of~$M$ is achieved at a vertex of this decomposition.
These vertices can be determined by solving a system of polynomial equations.
However, the resulting expressions are considerably more involved and do not admit an explicit form, unlike in the case of the figure-eight knot complement.
Therefore, we will omit the detail of the computation here.
For our purposes, it suffices to rely on the estimate $\RR(M) = 0.64393...$ provided by SnapPy~\cite{snappy}, which implies that 
\[
\RR(M) > \arcosh \left( \frac{\sqrt{22}}{4} \right).
\]
\end{proof}

We will also need the following description of noncompact hyperbolic $3$-manifolds of small volume.

\begin{lemma} \label{lem:2.56}
\mbox{ }
\begin{enumerate}
\item The only cusped orientable complete hyperbolic $3$-manifolds of volume at most~$2.56897$ are the figure-eight knot complement and its sister. \label{2.56i}
\item The only cusped nonorientable complete hyperbolic $3$-manifolds of volume at most~$1.28448$ is the Gieseking manifold. \label{2.56ii}
\end{enumerate}
\end{lemma}

\forget
\begin{lemma}
Let $M$ be a noncompact complete hyperbolic $3$-manifold.
\begin{enumerate}
\item If $M$ is orientable and $\vol(M) \leq 2.56897$ then $M$ is isometric either to the figure-eight knot complement or its sister.
\item If $M$ is nonorientable and $\vol(M) \leq 1.28448$ then $M$ is isometric to the Gieseking manifold.
\end{enumerate}
\end{lemma}
\forgotten

\begin{proof}
Let $M$ be a cusped orientable complete hyperbolic $3$-manifold with $\vol(M) \leq 2.56897$.
%
By~\cite{adams_cusps}, the volume of an $n$-cusped (orientable or nonorientable) complete hyperbolic $3$-manifold is at least~$n V$, where $V = 1.01494...$
Thus, the manifold~$M$ has at most two cusps.
By~\cite{agol_2cusps}, the volume of an orientable complete hyperbolic $3$-manifold with two cusps is at least four times Catalan's constant, that is approximately $3.6638$.
Thus, the manifold~$M$ has only one cusp.
Now, the list of all one-cusped orientable complete hyperbolic $3$-manifolds with volume bounded by~$2.848$ is given in~\cite{GMM}.
There are ten manifolds in this list and one of them must be the manifold~$M$.
They all have volume greater than~$2.56897$, except for the figure-eight knot complement~$m004$ and its sister~$m003$, which both have volume~$2V = 2.02988...$
Therefore, the manifold~$M$ is one of these two exceptions.

\medskip

Let $M$ be a cusped nonorientable complete hyperbolic $3$-manifold with $\vol(M) \leq 1.28448$.
Its orientable double cover~$\bar{M}$ has volume at most~$2.56897$.
By the previous discussion, the manifold~$\bar{M}$ must be isometric either to the figure-eight knot complement or its sister.
In both cases, the volume of~$\bar{M}$ is equal to~$2V$.
Thus, the volume of~$M$ is equal to~$V$, which is the minimal volume of a cusped hyperbolic $3$-manifold.
By uniqueness of the cusped hyperbolic $3$-manifold of minimal volume, see~\cite{adams_minvol}, the manifold~$M$ is isometric to the Gieseking manifold.
\end{proof}

The following optimal inradius--volume  estimate is the analogous of Theorem~\ref{theo:sys} with (half) the systole replaced by the inradius.

\begin{theorem} \label{theo:inradius_orientable}
Let $M$ be a noncompact orientable complete hyperbolic $3$-manifold of finite volume.
Then
\[
\cosh(\RR(M)) \leq C \vol(M)
\]
with equality if and only if $M$ is isometric to the sister~$m003$ of the figure-eight knot complement.
The optimal multiplicative constant is approximately~$C=0.59835...$

Furthermore, the following holds.
\begin{enumerate}
\item If $M$ is not isometric to the sister of the figure-eight knot complement, then
\[
\cosh(\RR(M)) \leq \frac{\sqrt{22}}{8V} \vol(M).
\]
with equality if and only if $M$ is isometric to the figure-eight knot complement.
The multiplicative constant is approximately~$0.57767...$
\item If $M$ is isometric neither to the figure-eight knot complement, nor to its sister, then 
\[
\cosh(\RR(M)) \leq 0.52519 \, \vol(M).
\]
\end{enumerate}
\end{theorem}

\begin{proof}
Suppose that 
\begin{equation} \label{eq:contradiction2}
\cosh(\RR(M)) > 0.52519 \, \vol(M).
\end{equation}

The largest embedded (open) ball~$B(\RR(M))$ in~$M$ is isometric to any ball of the same radius in the hyperbolic $3$-dimensional space~$\HH^3$.
Hence,
\begin{equation} \label{eq:ball2}
\vol(M) \geq \vol B(\RR(M)) = \pi \left( \sinh(2 \RR(M)) - 2 \RR(M) \right)
\end{equation}

Together, the inequalities~\eqref{eq:contradiction2} and~\eqref{eq:ball2} yield
\[
\cosh(\RR(M)) > 0.52519 \, \pi \left( \sinh(2 \RR(M)) - 2 \RR(M) \right).
\]

This implies
\begin{equation} \label{eq:sys<2}
\RR(M) < 0.8131.
\end{equation}

Applying the inequality~\eqref{eq:contradiction2} again, we deduce that
\begin{equation} \label{eq:vol<2}
\vol(M) < 2.56896.
\end{equation}

By Lemma~\ref{lem:2.56}.\eqref{2.56i}, the manifold~$M$ is isometric to either the figure-eight knot complement or its sister.
In this case, the result follows from Proposition~\ref{prop:inradius_orientable}.
\end{proof}

\subsection{Inradius of noncompact nonorientable hyperbolic \texorpdfstring{$3$}{3}-manifolds}

Let us start with the following example.

\begin{example} \label{ex:gieseking}
As shown in~\cite[\S3.2]{gen}, the inradius of the Gieseking manifold~$M$ is equal to
\[
\RR(M) = \arcosh \left( \frac{\sqrt{5}}{2} \right).
\]
Numerically, $\RR(M) = 0.48121...$

Recall that the Gieseking manifold is obtained by pairwise identifying the faces of a regular ideal simplex.
Under these identifications, the six edges are glued together to form a single geodesic line.
The manifold has a single (nonorientable) cusp and admits a unique embedded ball of maximal radius.
This ball is centered on this geodesic line at the point where the maximal cusp is self-tangent.
\end{example}

The following optimal inradius--volume estimate is the analogous of Theorem~\ref{theo:inradius_orientable}  for nonorientable hyperbolic manifolds.
The proof is similar.

\begin{theorem} \label{theo:inradius_nonorientable}
Let $M$ be a noncompact nonorientable complete hyperbolic $3$-manifold of finite volume.
Then
\[
\cosh(\RR(M)) \leq \frac{\sqrt{5}}{2V} \vol(M)
\]
with equality if and only if $M$ is isometric to the Gieseking manifold.
The multiplicative constant is approximately~$1.11803...$

Furthermore, if $M$ is not isometric the Gieseking manifold, then 
\[
\cosh(\RR(M)) \leq 0.95178 \, \vol(M).
\]
\end{theorem}

\begin{proof}
Suppose that 
\begin{equation} \label{eq:contradiction3}
\cosh(\RR(M)) > 0.95178 \, \vol(M).
\end{equation}

The largest embedded (open) ball~$B(\RR(M))$ in~$M$ is isometric to any ball of the same radius in the hyperbolic $3$-dimensional space~$\HH^3$.
Hence,
\begin{equation} \label{eq:ball3}
\vol(M) \geq \vol B(\RR(M)) = \pi \left( \sinh(2 \RR(M)) - 2 \RR(M) \right).
\end{equation}

Together, the inequalities~\eqref{eq:contradiction3} and~\eqref{eq:ball3} yield
\[
\cosh(\RR(M)) > 0.95178 \, \pi \left( \sinh(2 \RR(M)) - 2 \RR(M) \right).
\]

This implies
\[
\RR(M) < 0.65535.
\]

Applying the inequality~\eqref{eq:contradiction3} again, we deduce that
\[
\vol(M) < 1.28448.
\]

By Lemma~\ref{lem:2.56}.\eqref{2.56ii}, the manifold~$M$ is isometric to the Gieseking manifold.
In this case, the result follows from Example~\ref{ex:gieseking}.
\end{proof}

\section{Inradius rigidity}

In this section, we characterize cusped hyperbolic $3$-manifolds with minimal inradius.

\begin{theorem} \label{theo:rigidity}
Let $M$ be a noncompact complete hyperbolic~$3$-manifold of finite volume.
Then,
\begin{equation*} \label{eq:rigidity}
\cosh(\RR(M))  \geq \frac{\sqrt{5}}{2}
\end{equation*}
with equality if and only if $M$ is isometric to the Gieseking manifold.
\end{theorem}

The proof of this rigidity result is carried out in several steps.
We begin by characterizing manifolds of minimal inradius in terms of horoball tangencies; see Theorem~\ref{theo:tau=1}.
Next, we extend Adams' waist size rigidity theorem to orientable maximal cusps in nonorientable manifolds; see Theorem~\ref{theo:adams_nonorientable}.
We then address the case of nonorientable cusps; see Theorem~\ref{theo:nonorientable_cusp}.
We conclude the section with several elementary estimates regarding the inradius of flat Klein bottles and hyperbolic cusps that are used throughout the proof.

\subsection{Inradius and tangent horoballs}

Let $M = \HH^3 / \Gamma$ be a noncompact complete hyperbolic~$3$-manifold of finite volume.
Fix a maximal cusp~$C$ in~$M$.
Up to isometry, we can assume that $C$ is covered by the horoballs~$H_\infty$ and~$H_0$ centered at~$\infty$ and~$0$, and tangent at~$p_0=(0,0,1)$.

\medskip

Let us start with a simple observation from~\cite[Lemma~2.3]{adams_waist}.
Every parabolic isometry~$\tau \in \Gamma_\infty$ fixes~$H_\infty$ and sends~$H_0$ to a horoball tangent to~$H_\infty$ disjoint from~$H_0$.
This implies that the Euclidean distance between~$p_0$ and~$\tau.p_0$ for the induced Euclidean metric on~$\partial H_\infty$ is at least~$1$. 
That is,
\begin{equation} \label{eq:tau>1}
d_{\partial H_\infty}(p_0,\tau.p_0) \geq 1.
\end{equation}
Therefore, the waist size of an orientable maximal cusp~$C$ is at least~$1$; see Definition~\ref{def:waist}.
That is,
\begin{equation} \label{eq:wC>1}
w(C) \geq 1.
\end{equation}

An optimal lower bound on the inradius of a cusped hyperbolic $3$-manifolds has been established by Gendulphe in~\cite[Theorem~3.1]{gen}.
We go over the argument and expand on the equality case.

\begin{theorem} \label{theo:tau=1}
Let $M$ be a noncompact complete hyperbolic~$3$-manifold of finite volume.
Then,
\begin{equation} \label{eq:coshR>sqrt(5)/2}
\cosh(\RR(M))  \geq \frac{\sqrt{5}}{2}.
\end{equation}

Moreover, if equality holds, then there exists a parabolic isometry $\tau \in \Gamma_\infty$ such that
\begin{equation} \label{eq:Hinfinity}
d_{\partial H_\infty}(p_0,\tau.p_0) = 1,
\end{equation}
where $p_0 \in \HH^3$ is the point of tangency between the horoballs~$H_\infty$ and~$H_0$ covering a maximal cusp~$C$ of~$M$.
In other words, the horoballs~$H_0$ and~$\tau.H_0$ are tangent.
\end{theorem}

\begin{proof}
The $\Gamma$-orbit of~$p_0$ is formed of tangency points between horoballs covering~$C$.
The height of these points is either~$1$, when they lie in~$\partial H_\infty$, or at most~$\frac{1}{2}$, otherwise.
The tangency points between~$H_\infty$ and two full-size horoballs are at Euclidean distance at least~$1$ on~$\partial H_\infty$.
In particular, the minimal distance between~$p_0$ and the points of~$\Gamma.p_0$, different from~$p_0$, lying in~$\partial H_\infty$ is attained exactly at the tangency points between $H_\infty$ and a full-size horoball tangent to~$H_0$.
By the hyperbolic--Euclidean formula~\eqref{eq:hyp-eucl3}, this minimal distance equals
\begin{equation} \label{eq:min1}
\min_{\substack{p \in \Gamma.p_0 \cap H_\infty \\ p \neq p_0}} d(p_0,p) = \arcosh \left( \frac{3}{2} \right).
\end{equation}

Similarly, the minimal distance between~$p_0$ and the points of~$\Gamma.p_0$ not lying in~$\partial H_\infty$ is attained exactly at the tangency points between $H_0$ and a full-size horoball tangent to~$H_0$.
This minimal distance also equals
\begin{equation} \label{eq:min2}
\min_{p \in \Gamma.p_0 \setminus H_\infty} d(p_0,p) = \arcosh \left( \frac{3}{2} \right).
\end{equation}

Therefore, the distance from~$p_0$ to any other point of its $\Gamma$-orbit is at least $\arcosh( \frac{3}{2} )$.
By the characterization of the inradius in terms of orbit distances, see Lemma~\ref{lem:inrad-dist}, we have
\[
\RR(M) \geq \frac{1}{2} \arcosh \left( \frac{3}{2} \right) = \arcosh \left( \frac{\sqrt{5}}{2} \right),
\]
which proves~\eqref{eq:coshR>sqrt(5)/2}.

Suppose now that equality holds.
Then $H_0$ must be tangent to a full-size horoball~$H$ of its $\Gamma$-orbit, otherwise the inequalities in~\eqref{eq:min1} and~\eqref{eq:min2} would be strict.

We can assume that $H_0$ is not tangent to any of its images $\tau.H_0$ by a parabolic isometry $\tau \in \Gamma_\infty$, otherwise $d(p_0,\tau.p_0)=1$ as desired, since $\tau.H_0$ is a full-size horoball.

In this case, choose a point~$q=\gamma.p_0$ in the $\Gamma$-orbit of~$p_0$ corresponding to the tangency point between $H_\infty$ and a full-size horoball~$H$ tangent to~$H_0$.
By construction, the point $q$ is the tangent point between $\gamma.H_0$ and~$\gamma.H_\infty$.
This horoball pair coincides with~$\{H, H_\infty\}$.
Since $\gamma$ is not parabolic, this means that $\gamma.H_0=H_\infty$ and $\gamma.H_\infty=H$.

Thus, when the point~$p_t=(0,e^t)$ moves upward along the vertical geodesic above the center of~$H_0$, its image~$\gamma.p_t$ moves downward along the vertical geodesic above the center of~$H$.
Up to a rotation around the vertical axis, we may assume 
\[
\gamma.p_t=u_t=(1,e^{-t}).
\]

Similarly, choose a point $q'=\eta.p_0$ in the $\Gamma$-orbit of~$p_0$ corresponding to the tangency point between $H_0$ and a full-size horoball~$H$ tangent to~$H_0$.
The image~$\eta.p_t$ lies in the geodesic joining the centers of~$H_0$ and~$H$.
This geodesic coincides with the semicircle of height~$\frac{1}{2}$ between these two centers.
Up to a rotation around the vertical axis, the point~$\eta.p_t$ at distance~$t$ from $\eta.p_0=(\frac{1}{2},\frac{1}{2})$ coincides with
$
\left( \frac{1}{2} \pm \frac{1}{2} \tanh(t), \frac{1}{2 \cosh(t)} \right).
$
This can be checked by using the Euclidean--hyperbolic distance formula~\eqref{eq:hyp-eucl}.
Among these two points, the point
\[
v_t = \left( \frac{1}{2} - \frac{1}{2} \tanh(t), \frac{1}{2 \cosh(t)} \right)
\]
is closer to~$p_t$.

By construction, for $t>0$ small enough, all other points of the $\Gamma$-orbit of~$p_t$ lie at distance greater than 
\[
\min \{ d(p_t,u_t), d(p_t,v_t) \}
\]
from~$p_t$.
A direct computation gives
\[
d(p_t,u_t) = d(p_t,v_t) = \arcosh \left( \cosh(2t) + \frac{1}{2} \right).
\]
Therefore, for $t>0$ small enough, the ball of radius
\[
\frac{1}{2} \arcosh \left( \cosh(2t) + \frac{1}{2} \right) > \arcosh \left( \frac{\sqrt{5}}{2} \right)
\]
centered at the projection of~$p_t$ in~$M$ is embedded.
Hence a contradiction.
\end{proof}

\begin{remark} \label{rem:w>1}
The equality~\eqref{eq:Hinfinity} implies that the waist size of~$C$ is at most~$1$.
If $C$ is orientable, this forces the waist size to be exactly~$1$; see~\eqref{eq:wC>1}.
\end{remark}

\forget

{\color{gray}

\begin{proof}
Let $D_\infty(r)$ and~$D_0(r)$ be the Euclidean disks of radius~$r \geq \frac{1}{2}$ centered at~$x_0$ in~$\partial H_\infty$ and~$\partial H_0$.
Consider the geodesic cones~$\Cone_\infty(D_\infty(r))$ and~$\Cone_0(D_0(r))$ in~$\HH^3$ based at~$\infty$ and~$0$ over these two disks.
Since $r \geq \frac{1}{2}$, the cone~$\Cone_\infty(D_\infty(r))$ contains the horoball~$H_0$.
Similarly, the cone~$\Cone_0(D_0(r))$ contains the horoball~$H_\infty$.

Let us show that the intersection $\Cone_\infty(D_\infty(r)) \cap \Cone_0(D_0(r))$ is contained in the Dirichlet domain of~$\Gamma$ centered at~$x_0$.
Let $x$ be a point in the intersection of the two cones.
Taking the symmetry of~$\HH^3$ with respect to the unit Euclidean hemisphere~$S^+$ centered at~$0$ switching~$H_\infty$ and~$H_0$ if necessary, we can assume that the point~$x$ lies outside this hemisphere.
By construction, the point~$x_0$ minimizes the distance to~$x$ among the points in the $\Gamma_\infty$-orbit of~$x_0$.
The points in the $\Gamma$-orbit of~$x_0$ which do not belong to its $\Gamma_\infty$-orbit have a height at most~$\frac{1}{2}$ because these points are tangent points between horoballs of~$\mathcal{C}$ distinct from~$H_\infty$.
Now, the maximal distance between~$x_0$ and a point~$x$ of~$\Cone_\infty(D_\infty(r))$ outside~$H_\infty$ and the unit hemisphere~$S^+$ is attained when $x$ lies in the intersection of~$\partial \Cone_\infty(D_\infty(r))$ and~$\partial S^+$.
For instance, this occurs for~$x=(r,\sqrt{1-r^2}) \in \HH^3$.
Thus, by the Euclidean--hyperbolic distance formula~\eqref{eq:hyp-eucl}, we derive
\[
d(x_0,x) \leq \arcosh \left( \sqrt{1-r^2} \right).
\]
Now, since the point~$x$ lies in~$\Cone_\infty(D_\infty(r))$ outside the hemisphere~$S^+$, its height is at least~$\sqrt{1-r^2}$.
It follows from~\eqref{eq:hyp-eucl} that 
\[
d(x,\{ {\rm height} \leq \tfrac{1}{2} \}) \geq \arcosh \left( \frac{ \frac{5}{4} -r^2}{\sqrt{1-r^2}} \right).
\]
where $\{ {\rm height} \leq \frac{1}{2} \}$ is the set of points of~$\HH^3$ of height at most~$\frac{1}{2}$.

\end{proof}

}

\forgotten

\subsection{Orientable cusps in nonorientable manifolds}

Adams~\cite{adams_waist} proved that if a noncompact orientable complete hyperbolic $3$-manifold admits a maximal cusp with waist size~$1$ then it is isometric to the figure-eight knot complement; see Theorem~\ref{theo:waist}.
We extend this result to the case of an \emph{orientable} maximal cusp in \emph{nonorientable} manifolds.

\begin{theorem} \label{theo:adams_nonorientable}
Let $M$ be a noncompact (orientable or nonorientable) complete hyperbolic~$3$-manifold of finite volume.
Let $C$ be an orientable maximal cusp of~$M$.
Then 
\[
w(C) \geq 1
\]
with equality if and only if $M$ is isometric to the figure-eight knot complement.
\end{theorem}

%

The following lemma will be used in the proofs of both Theorem~\ref{theo:adams_nonorientable} and Theorem~\ref{theo:nonorientable_cusp}.
The first part of this lemma can be found in~\cite[Lemma~2.7]{adams_waist}. 
We present here a different proof.

\begin{proposition} \label{prop:tangent}
Let $H$, $H'$, $H''$ be three horoballs of height~$h$, $h'$, $h''$ centered at~$\xi, \xi', \xi'' \in \partial \HH^3 \setminus \{ \infty\}$.
Let $d=|\xi'-\xi|$ be the Euclidean distance between the centers of~$H$ and~$H'$.
Let $\gamma$ be an isometry sending~$H$ to~$H_\infty$.
Denote by~$\xi_\gamma=\gamma.\infty$, $\xi'_\gamma=\gamma.\xi'$, $\xi''_\gamma=\gamma.\xi''$ the centers of~$\gamma.H_\infty$, $\gamma.H'$, $\gamma.H''$.
Then $\gamma.H'$ is a horoball of height~$\frac{hh'}{d^2}$ whose center is at Euclidean distance~$\frac{h}{d}$ from that of~$\gamma.H_\infty$.

\clearpage

Moreover, the oriented angle between the centers of~$\gamma.H_\infty$, $\gamma.H'$, $\gamma.H''$ equals that between the centers of~$H$, $H'$, $H''$, up to orientation.
More precisely, 
\[
\widehat{\xi'_\gamma \xi_\gamma \xi''_\gamma} = 
\begin{cases}
- \widehat{\xi' \xi \xi''} & \text{if } \gamma \text{ is orientation-preserving} \\
\phantom{-} \widehat{\xi' \xi \xi''} & \text{otherwise}.
\end{cases}
\]
See Figure~\ref{fig:gamma}.

\begin{figure}[htbp!]
\centering
\vspace*{-0.7cm}
\hspace*{-0.5cm}
\def\svgwidth{0.75\textheight}
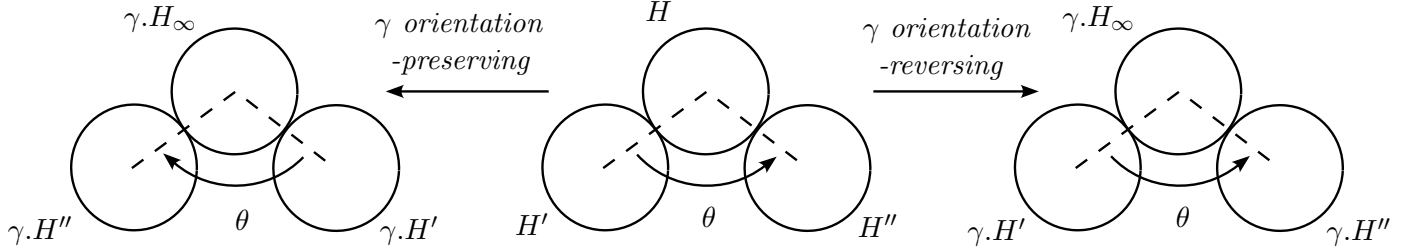
\caption{Action of~$\gamma$ sending~$H$ to~$H_\infty$.} \label{fig:gamma}
\end{figure}
\end{proposition}

\begin{remark}
The isometry~$\gamma \in \Gamma$ in Proposition~\ref{prop:tangent} is uniquely defined up to left multiplication by an element of~$\Gamma_\infty$.
More precisely, if $\gamma' \in \Gamma$ also sends~$H$ to~$H_\infty$, then $\gamma'=\sigma \gamma$ for some $\sigma \in \Gamma_\infty$.
\end{remark}

\begin{remark} \label{rem:updown}
Recall that isometries preserve horoball tangencies.
In particular, if $H'$ is tangent to~$H$, then $\gamma.H'$ is of height~$1$.
By construction, the isometry~$\gamma$ also maps
\begin{itemize}
\item the oriented geodesic~$(\xi,\infty)$ pointing upward from~$\xi$ to the oriented geodesic~$(\infty,\xi_\gamma)$ pointing downward toward~$\xi_\gamma$;
\item the oriented geodesics~$(\xi',\infty)$ and~$(\xi'',\infty)$ pointing upward from~$\xi'$ and~$\xi''$ to the oriented geodesics~$(\xi'_\gamma,\xi_\gamma)$ and~$(\xi''_\gamma,\xi_\gamma)$ pointing inward toward the horoball~$H_{\xi_\gamma}=\gamma.H_\infty$ to~$\xi_\gamma$;
\item the oriented geodesics~$(\xi,\xi')$ and~$(\xi,\xi'')$, pointing outward from~$H_\xi$ toward~$\xi'$ and~$\xi''$, to the oriented geodesics~$(\infty,\xi'_\gamma)$ and~$(\infty,\xi''_\gamma)$ pointing downward toward~$\xi'_\gamma$ and~$\xi''_\gamma$;
\item the oriented geodesic~$(\xi',\xi'')$ from~$\xi'$ to~$\xi''$ to the oriented geodesic~$(\xi'_\gamma,\xi''_\gamma)$ from~$\xi'_\gamma$ to~$\xi''_\gamma$.
\end{itemize}
\end{remark}

\begin{proof}[Proof of Proposition~\ref{prop:tangent}]
Let $K'$ be the unique horoball centered at~$\xi'$ tangent to~$H$.
By the distance--height formula~\eqref{eq:h1h2}, the height of~$K'$ equals~$\frac{d^2}{h}$.
Since $\gamma$ preserves tangencies, the horoball~$\gamma.K'$ is tangent to~$\gamma.H=H_\infty$, and hence has height~$1$.
For horoballs with the same center, the distance between their boundary is the logarithm of the ratio of their heights, see~\eqref{eq:hyp-eucl2}, and this ratio is preserved under~$\gamma$.
Thus, since $\height(H')=h'$, $\height(K')=\frac{d^2}{h}$ and $\height(\gamma.K')=1$, we obtain
\[
\height(\gamma.H') = \frac{\height(H')}{\height(K')} \times \height(\gamma.K') = \frac{hh'}{d^2}.
\]

To compute the distance between the centers, let $K'_+$ be the full-size horoball centered at~$\xi'$, and let $H_-$ be the horoball centered at~$\xi$ tangent to~$K'_+$.
By the distance--height formula~\eqref{eq:h1h2}, the height of~$H_-$ equals~$d^2$.
Since $\gamma$ sends~$H$ to~$H_\infty$ and preserves the distance between their boundaries, the horoball~$\gamma.H_\infty$ has the same height~$h$ as~$H$. 
Similarly, since $H$ and $H_-$ have the same center and $\gamma$ preserves the distance between their boundaries, the height of~$\gamma.H_-$ is equal to~$\frac{h}{d^2}$. 
Moreover, since~$K'_+$ is tangent to~$H_-$, the horoball~$\gamma.K'_+$ is tangent to the horizontal horoball~$\gamma.H_-$, and thus also has height~$\frac{d^2}{h}$.
Finally, applying the distance--height formula~\eqref{eq:h1h2}, to the tangent horoballs $\gamma.K'_+$ and~$\gamma.H_-$ gives
\[
|\xi'_\gamma-\xi_\gamma|=\frac{h}{d}
\]
as claimed.

The isometry~$\gamma \in \Gamma$ satisfies~$\gamma.\xi=\infty$ and $\gamma.\infty=\xi_\gamma$.
Hence, it induces a conformal or anticonformal homeomorphism of $\partial \HH^3 \simeq \widehat{\C}$ of the form
\[
\frac{az+b}{cz+d} = \xi_\gamma + \frac{b-a \xi}{c(z-\xi)}
\qquad \text{ or } \qquad
\frac{a\bar{z}+b}{c\bar{z}+d} = \xi_\gamma + \frac{b-a \bar{\xi}}{c(\bar{z}-\bar{\xi})}
\]
according to whether $\gamma$ is orientation-preserving or orientation-reversing.

In the orientation-preserving case, 
\[
\arg \left( \frac{\xi''_\gamma - \xi_\gamma}{\xi'_\gamma - \xi_\gamma} \right) = \arg  \left( \frac{\xi' - \xi}{\xi'' - \xi} \right) = - \arg  \left( \frac{\xi'' - \xi}{\xi' - \xi} \right),
\]
while in the orientation-reversing case, 
\[
\arg \left( \frac{\xi''_\gamma - \xi_\gamma}{\xi'_\gamma - \xi_\gamma} \right) = \arg  \left( \frac{\overline{\xi'} - \bar{\xi}}{\overline{\xi''} - \bar{\xi}} \right) = \arg  \left( \frac{\xi'' - \xi}{\xi' - \xi} \right).
\]
This proves the angle formula.
%
\end{proof}

We can now proceed to the proof of Theorem~\ref{theo:adams_nonorientable}.

\begin{proof}[Proof of Theorem~\ref{theo:adams_nonorientable}]
As noted in~\eqref{eq:wC>1}, the waist size of an orientable cusp is at least~$1$; see Remark~\ref{rem:w>1}.
It therefore remains to consider the equality case where the waist size of~$C$ is exactly~$1$.
We may assume that $M$ is nonorientable, otherwise the result follows directly from Adams' waist size theorem; see Theorem~\ref{theo:waist}.
The maximal cusp~$C$ of~$M$ lifts to a cusp~$\bar{C}$ in the orientable double cover~$\bar{M}$ of~$M$.
The key point is that, although the lift of a maximal cusp is not maximal in general, the assumption that the waist size of~$C$ is equal to~$1$ implies that~$\bar{C}$ is in fact maximal in~$\bar{M}$.
We also verify that the waist size of~$\bar{C}$ remains equal to~$1$.
Applying Theorem~\ref{theo:waist}, we deduce that~$\bar{M}$ is isometric to the figure-eight knot complement.
It follows that~$M$ is isometric to the Gieseking manifold.

\medskip

As usual, we assume that the maximal cusp~$C$ in~$M$ is covered in~$\HH^3$ by the horoballs~$H_0$ and~$H_\infty$ centered at~$0$ and~$\infty$, and tangent at~$p_0=(0,0,1)$.
We analyze the cusp diagram obtained by projecting all full-size horoballs (\ie, those tangent to~$H_\infty$) to~$\partial H_\infty$, together with the $\Gamma$-equivalence classes -- labelled by letters -- of oriented geodesics in~$\HH^3$, using a method developed in~\cite{adams_waist}.
To prove that $\bar{C}$ is a maximal cusp in~$\bar{M}$, it suffices to show that there exists an orientation-preserving isometry~$\eta \in \Gamma$ sending the horoball~$H_\infty$ covering~$C$ to a full-size horoball tangent to~$H_\infty$.

\medskip

Label by~$A$ the oriented vertical geodesic from~$0$ to~$\infty$ (and all its images under~$\Gamma$).
Since $C$ is orientable and $w(C)=1$, there exists $\tau \in \Gamma_\infty$ acting on~$\partial H_\infty$ as a Euclidean translation of length~$1$.
Label by~$B$ and color in blue the oriented geodesic from~$0$ to~$\tau.0$ (and all its images under~$\Gamma$).
Figure~\ref{fig:1} depicts the translates of the full-size horoball 
\[
H_A^+=H_0
\]
by~$\tau^{\pm 1}$, along with the vertical $A$-geodesics pointing upward from the centers of these horoballs and the $B$-geodesics between consecutive translates of these horoball centers.
By convention, the translation vector of~$\tau$ horizontally points to the right. 

\medskip

\begin{figure}[htbp!]
\vspace{0.4cm}
\centering
\def\svgwidth{0.25\textheight}
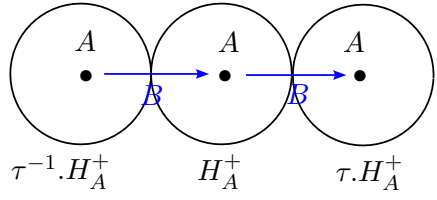
\vspace{0.5cm}
\caption{$H_A^+$ and its nearby $\tau$-translates.} \label{fig:1}
\end{figure}

Fix an isometry~$\gamma \in \Gamma$ sending~$H_A^+$ to~$H_\infty$ and apply Proposition~\ref{prop:tangent} to $H_A^+$, $\tau.H_A^+$, $\tau^{-1}.H_A^+$.
This yields three aligned full-size horoballs
\[
H_A^-=\gamma.H_\infty, \quad H_B^-=\gamma \tau.H_A^+, \quad H_B^+= \gamma \tau^{-1}.H_A^+
\]
together with an $A$-geodesic pointing downward toward the center of~$H_A^-$, a $B$-geodesic pointing downward toward the center of~$H_B^-$, a $B$-geodesic pointing upward from the center of~$H_B^+$, and two (non-vertical) $A$-geodesics from the centers of~$H_B^{\pm}$ to the center of~$H_A^-$; see Remark~\ref{rem:updown} and Figure~\ref{fig:2}. 

\medskip

\begin{figure}[htbp!]
\vspace{3.4cm}
\centering
\def\svgwidth{0.16\textheight}
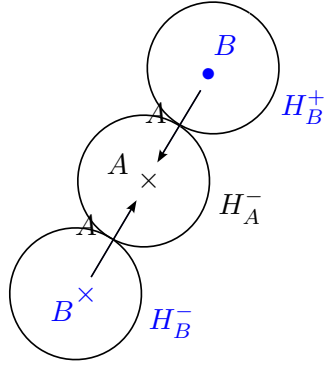
\caption{Image of $H_A^+$ and of its nearby $\tau$-translates under~$\gamma$ (up to rotation).} \label{fig:2}
\end{figure}

The horoballs~$H_B^\pm$ are distinct from~$\tau^{\pm 1}. H_A^-$ (and hence disjoint from them).
Otherwise, the two pairs of horoballs $H_B^{\pm}$  and~$\tau^{\pm 1}. H_A^+$ would coincide and the vertical $B$-geodesics above their centers, with one pointing upward and the other downward, would be identified under the parabolic isometry~$\tau^2$.
This would imply the existence of a deck transformation taking a geodesic to itself, switching its orientation, which is impossible since $\Gamma$ acts freely on~$\HH^3$.

\medskip

Now, each horoball $H_B^+$, $H_A^-$, $H_B^-$ generates a horizontal row of full-size horoballs formed by its $\tau$-translates.
Up to a horizontal reflection, we may assume that the row generated by~$H_B^+$ lies directly above that of~$H_A^-$ and that, consequently, the row generated by~$H_A^-$ lies directly above that of~$H_B^-$; see Figure~\ref{fig:3}.
Label by~$D$ (and not~$C$ to avoid any confusion with the cusp) and color in orange the oriented geodesic from the center of~$H_B^+$ to its translate by~$\tau$ (and all its images by~$\Gamma$).

\medskip

\begin{figure}[htbp!]
\centering
\vspace*{2.5cm}
\def\svgwidth{0.28\textheight}
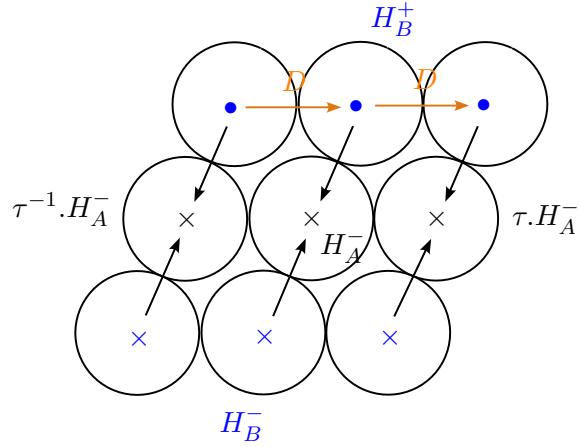
\vspace*{0.3cm}
\caption{Horizontal rows generated by $H_B^+$, $H_A^-$, $H_B^-$} \label{fig:3}
\end{figure}

Figure~\ref{fig:3} illustrates the tangencies that necessarily occur, but there could be others.
For instance, the horoball~$H_B^-$ could be tangent not only to~$H_A^-$, but also to its translates $\tau.H_A^-$ or~$\tau^{-1}.H_A^-$. \\

Label by~$\bar{A}$ the $A$-geodesics with their opposite orientation and the same with the other letters.

\begin{claim} \label{claim:AB}
The classes of~$A$-geodesics and $B$-geodesics are disjoint.
Likewise, the classes of~$A$-geodesics and $\bar{B}$-geodesics are disjoint.
\end{claim}

\begin{proof}
By contradiction, if the classes of $A$ and $B$ geodesics intersect, then they coincide.
Hence, there exists $\sigma \in \Gamma$ sending the $B$-geodesic pointing downward toward the center of~$H_B^-$ to the $A$-geodesic pointing downward toward the center of~$H_A^-$.
By construction, the isometry~$\sigma$ fixes~$\infty$ and is therefore parabolic.
Moreover, since the cusp~$C$ covered by~$H_\infty$ is orientable and $\sigma$ sends the full-size horoball~$H_B^-$ to its tangent full-size horoball~$H_A^-$, the induced translation on~$\partial H_\infty$ has length~$1$ and points in the direction of the aligned horoballs~$H_B^-$, $H_A^-$, $H_B^+$.
It follows that $\sigma.H_B^-=H_A^-$ and $\sigma.H_A^-=H_B^+$; see Figure~\ref{fig:3}. 
In particular, the isometry~$\sigma$ sends the $A$-geodesic joining the centers of~$H_B^-$ and~$H_A^-$ to the $A$-geodesic joining the centers of~$H_B^+$ and~$H_A^-$, with opposite orientation.
This is impossible since $\Gamma$ acts freely on~$\HH^3$.

Similarly, if the classes of $A$ and $\bar{B}$ geodesics coincide, then there exists $\bar{\sigma} \in \Gamma$ sending the $B$-geodesic pointing upward from the center of~$H_B^+$ to the $A$-geodesic pointing downward toward the center of~$H_A^-$.
As previously, we deduce that $\bar{\sigma}.H_B^+=H_A^-$ and $\bar{\sigma}.H_A^-=H_B^-$, see Figure~\ref{fig:3}, and derive a contradiction for analogous reasons.
\end{proof}

Fix an isometry~$\eta \in \Gamma$ taking~$H_B^+$ to~$H_\infty$.
By construction, the isometry~$\eta$ sends the $B$-geodesic pointing upward from the center of~$H_B^+$ to a downward-pointing $B$-geodesic.
After composing~$\eta$ with a parabolic isometry of~$\Gamma_\infty$ if necessary, we may assume that $\eta$ sends the upward-pointing $B$-geodesic from the center of~$H_B^+$ to the $B$-geodesic pointing downward toward the center of~$H_B^-$.
In particular, $\eta.H_\infty=H_B^-$.

\medskip

Applying Proposition~\ref{prop:tangent} to $H_B^+$, $\tau.H_B^+$, $\tau^{-1}.H_B^+$  and the isometry~$\eta$ yields three aligned horoballs 
\[
H_B^-=\eta.H_\infty, \quad H_D^-=\eta \tau.H_B^+, \quad H_D^+= \eta \tau^{-1}.H_B^+
\] 
together with a $D$-geodesic pointing downward toward the center of~$H_D^-$, a $D$-geodesic pointing upward from the center of~$H_D^+$ and two (non-vertical) $B$-geodesics from the centers of~$H_D^{\pm }$ to the center of~$H_B^-$; see Remark~\ref{rem:updown} and Figure~\ref{fig:4}. 

\medskip

\begin{figure}[htbp!]
\centering
\vspace*{1cm}
\hspace*{-4,2cm}
\def\svgwidth{0.4\textheight}
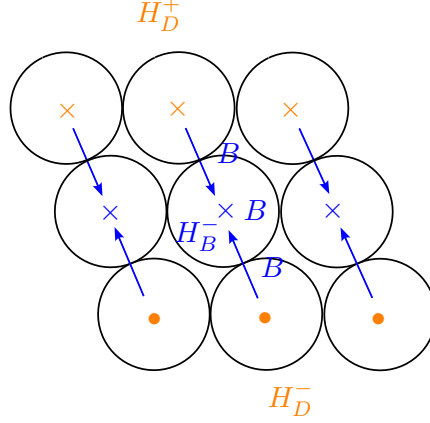
\vspace*{0.4cm}
\caption{Horizontal rows generated by $H_D^+$, $H_B^-$, $H_D^-$ (up to horizontal reflection).} \label{fig:4}
\end{figure}

As previously, the horoballs~$H_D^\pm$, which are tangent to~$H_B^-$ and aligned with it, are distinct from the horoballs~$\tau^{\pm 1}.H_B^-$ tangent to~$H_B^-$.
Otherwise the translation~$\tau^2$ would take a vertical $D$-geodesic to another vertical $D$-geodesic with opposite orientation, which is impossible.
Thus, the horoballs~$H_D^\pm$ lie on either side of the horizontal row of horoballs generated by~$H_B^-$; see Figure~\ref{fig:4}. 
Since the row generated by~$H_A^-$ lies directly above that of~$H_B^-$, one of the horoballs~$H_D^{\pm}$, say~$H_D^\varepsilon$ with $\varepsilon=\pm$, coincides with a $\tau$-translate of~$H_A^-$ and therefore has an $A$-geodesic pointing downward toward its center, while the other horoball~$H_D^{-\varepsilon}$ has an $A$-geodesic pointing upward from its center; compare Figure~\ref{fig:3} and Figure~\ref{fig:4}.

\medskip

Thus, there are four horizontal rows of full-size horoballs directly below each other: the first generated by~$H_B^+$ with $B$-geodesics pointing upward from the centers, the second generated by~$H_A^-$ with $A$-geodesics pointing downward toward the centers, the third generated by~$H_B^-$ with $B$-geodesics pointing downward toward the centers and the fourth generated by~$H_D^{-\varepsilon}$ with $A$-geodesics pointing upward from the centers.
In other words, the row generated by~$H_D^{-\varepsilon}$ with $A$-geodesics pointing upward from the centers lies directly below the last row in Figure~\ref{fig:3}.

\medskip

Observe also that since there is an $A$-geodesic joining the centers of~$H_B^-$ and~$H_A^-$ (see Figure~\ref{fig:3}), and a $\bar{B}$-geodesic joining the centers of~$H_B^-$ and~$H_D^\pm$ (see Figure~\ref{fig:4}), the horoball~$H_D^\varepsilon$ is distinct from~$H_A^-$, otherwise we would obtain a contradiction with Claim~\ref{claim:AB}.
Therefore, $H_B^-$ and more generally every horoball of the row generated by~$H_B^-$ is tangent to two horoballs of the row generated by~$H_A^-$.
Thus, the oriented angle~$\theta$ at the center of $H_B^-$ between the centers of~$\tau.H_B^-$ and~$H_A^-$ is equal to~$\frac{\pi}{3}$ or~$\frac{2\pi}{3}$.
This forces the four horizontal rows generated by $H_B^+$, $H_A^-$, $H_B^-$ and~$H_D^{-\varepsilon}$ to form a hexagonal packing; see Figure~\ref{fig:5}.

\medskip

\begin{figure}[htbp!]
\centering
\vspace*{0.7cm}
\def\svgwidth{0.6\textheight}
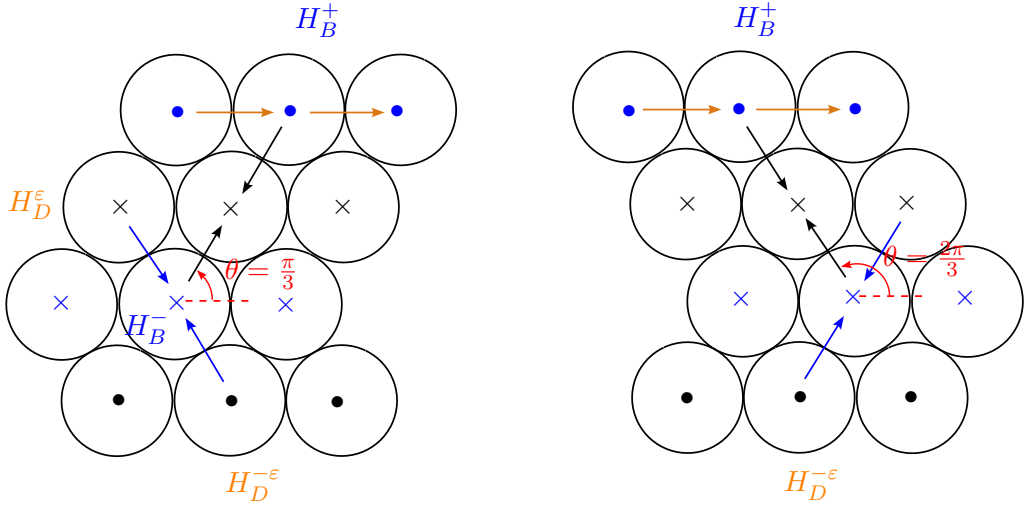
\vspace*{0.4cm}
\caption{Cusp diagram with~$\theta=\frac{\pi}{3}$ and~$\theta=\frac{2\pi}{3}$} \label{fig:5}
\end{figure}

\begin{claim}
The isometry~$\eta$ is orientation-preserving.
\end{claim}

\begin{proof}
By contradiction, suppose that $\eta$ is orientation-reversing.
Apply Proposition~\ref{prop:tangent} to $H_B^+$, $\tau^{-1}.H_B^+$, $H_A^-$ with the isometry~$\eta$.
Note that the angle at the center of $H_B^+$ between the centers of~$\tau^{-1}.H_B^+$ and~$H_A^-$ is equal to~$\theta$.
Since $\eta$ is orientation-reversing, the angle at the center of $H_B^-=\eta.H_B^+$ between the centers of~$H_D^+=\eta \tau^{-1}.H_B^-$ and~$\eta.H_A^-$ is equal to~$\theta$.

If $\theta=\frac{\pi}{3}$ then, whether $H_D^+$ agrees with~$H_D^\varepsilon$ or~$H_D^{-\varepsilon}$, the horoball~$\eta.H_A^-$ coincides with $\tau^{-1}.H_B^-$ or~$\tau.H_B^-$; see Figure~\ref{fig:5}.
Similarly, if $\theta=\frac{2\pi}{3}$ then, whether $H_D^+$ agrees with~$H_D^\varepsilon$ or~$H_D^{-\varepsilon}$, the horoball~$\eta.H_A^-$ coincides with $\tau.H_B^-$ or~$\tau^{-1}.H_B^-$; see Figure~\ref{fig:5}.

In both cases, the horoball~$\eta.H_A^-$ admits a $B$-geodesic pointing downward toward its center.
However, by Remark~\ref{rem:updown}, the deck transformation~$\eta$ sends the $A$-geodesic between the centers of~$H_B^+$ and~$H_A^-$ to an  $A$-geodesic pointing downward to~$\eta.H_A^-$.
Hence a contradiction with Claim~\ref{claim:AB}.
\end{proof}

Since $\eta$ is orientation-preserving, it belongs to the deck transformation group~$\bar{\Gamma}$ of the orientable double cover~$\bar{M}$.
Moreover, the transformation~$\eta$ sends~$H_\infty$ to a full-size tangent horoball, namely~$H_B^-$.
Thus, the  cusp~$\bar{C}$ is maximal in~$\bar{M}$.

Moreover, because $C$ is orientable, the parabolic subgroups of~$\Gamma$ and~$\bar{\Gamma}$ fixing~$\infty$ coincide.
Since both cusps $C$ and~$\bar{C}$ are covered by~$H_\infty$, we deduce that 
\[
w(\bar{C})=w(C)=1.
\]

By Adams' waist size theorem applied to~$\bar{C}$, see Theorem~\ref{theo:waist}, the manifold~$\bar{M}$ is isometric to the figure-eight knot complement.
Hence, $M$ is isometric to the Gieseking manifold.
Since the unique cusp of the Gieseking manifold is nonorientable, this contradicts the assumption that $C$ is orientable.
\end{proof}

\forget

\begin{corollary}
Let $M$ be a noncompact nonorientable complete hyperbolic~$3$-manifold of finite volume with an orientable cusp.
Then
\[
\cosh(\RR(M))  > \frac{\sqrt{5}}{2}.
\]
\end{corollary}

\begin{proof}
Fix an orientable maximal cusp~$C$ in~$M$.
By contradiction, suppose that 
\[
\cosh(\RR(M))  = \frac{\sqrt{5}}{2}.
\]
By Theorem~\ref{theo:tau=1} and since $C$ is orientable, the waist size of~$C$ is equal to~$1$; see Remark~\ref{rem:w>1}.
It follows from Theorem~\ref{theo:adams_nonorientable} that $M$ is isometric to the figure-eight knot complement.
By Proposition~\ref{prop:inradius_orientable}, this implies that
\[
\cosh(\RR(M)) = \frac{\sqrt{22}}{4}.
\]
Hence a contradiction.
\end{proof}

\forgotten

\subsection{Nonorientable cusps in nonorientable manifolds}

Our next result covers the remaining rigidity case.

\begin{theorem} \label{theo:nonorientable_cusp}
Let $M$ be a noncompact nonorientable complete hyperbolic~$3$-manifold of finite volume with a nonorientable cusp.
If
\[
\cosh(\RR(M))  = \frac{\sqrt{5}}{2},
\]
then $M$ is isometric to the Gieseking manifold.
\end{theorem}

Let us start with the following useful result.

\begin{lemma} \label{lem:C-bar}
Let $M$ be a noncompact nonorientable complete hyperbolic $3$-manifold of finite volume.
Denote by~$\bar{M}$ the orientable double cover of $M$.
Let $C$ be a nonorientable maximal cusp of~$M$.
Then the cusp~$\bar{C}$ of~$\bar{M}$ covering~$C$ is maximal.
\end{lemma}

\begin{proof}
Denote by~$\Gamma$ and~$\bar{\Gamma}$ the deck transformation groups of~$M$ and~$\bar{M}$.
Let $H$ be a horoball covering the cusp~$C$.
Then $H$ also covers~$\bar{C}$.
Since $C$ is maximal in~$M$, there exists $\gamma \in \Gamma$ such that $H$ and $\gamma(H)$ are tangent.

If $\gamma$ is orientation-preserving, then $\gamma \in \bar{\Gamma}$, and the same tangency occurs in the
$\bar{\Gamma}$-orbit of~$H$.
Therefore the cusp~$\bar C$ covered by~$H$ is maximal in~$\bar{M}$.

If $\gamma$ is orientation-reversing, observe that since $C$ is nonorientable, there exists an orientation-reversing isometry~$\sigma \in \Gamma$ fixing~$H$.
By construction, the isometry~$\gamma'=\gamma \sigma$ is orientation-preserving, that is $\gamma' \in \bar{\Gamma}$, and sends~$H$ to the same tangent horoball as~$\gamma$.
We may therefore apply the previous argument to~$\gamma'$ and conclude that $\bar{C}$ is a maximal cusp of~$\bar{M}$.
\end{proof}

We now prove the characterization of the Gieseking manifold in terms of its inradius.

\begin{proof}[Proof of Theorem~\ref{theo:nonorientable_cusp}]
Let $C$ be a nonorientable maximal cusp of~$M$.
As usual, we can assume that $C$ is covered by the tangent horoballs~$H_\infty$ and~$H_0$.
By Lemma~\ref{lem:C-bar}, the cusp~$\bar{C}$ of~$\bar{M}$ covering~$C$ is maximal. 

\medskip

Suppose first that there exists an orientation-preserving isometry~$\alpha \in \Gamma_\infty$ with
\[
d_{\partial H_\infty}(p,\alpha.p) = 1
\]
for some (and hence every) $p \in \partial H_\infty$.
Then, $\alpha \in \bar{\Gamma}_\infty$ and the waist size of the maximal cusp~$\bar{C}$ equals~$1$.
By Adams' waist size theorem, see Theorem~\ref{theo:waist}, the orientable manifold~$\bar{M}$ is isometric to the figure-eight knot complement.
Hence, $M$ is isometric to the Gieseking manifold. 

\medskip

Suppose now that every nontrivial orientation-preserving isometry in~$\Gamma_\infty$ acts on~$\partial H_\infty$ as a Euclidean translation of length~$>1$.
Let $\tau \in \Gamma_\infty$ be the parabolic isometry satisfying
\begin{equation*} 
d_{\partial H_\infty}(p_0,\tau.p_0) = 1
\end{equation*}
where $p_0=(0,0,1) \in \HH^3$ is the tangency point between~$H_\infty$ and~$H_0$, as provided by Theorem~\ref{theo:tau=1}.
By assumption, $\tau$ must be orientation-reversing.
Since $C$ is nonorientable, the parabolic subgroup~$\Gamma_\infty$ is isomorphic to the Klein bottle group.
Consequently, $\tau$ acts on~$\partial H_\infty$ as a glide reflection.
Moreover, by~\eqref{eq:tau>1}, the isometry~$\tau$ realizes the minimal displacement among all glide reflections in~$\Gamma_\infty$.

\medskip

Label by~$A$ the oriented vertical geodesic in~$\HH^3$ from~$0$ to~$\infty$ (and all its images under~$\Gamma$).
Label also by~$B$ the oriented geodesic from~$0$ to~$\tau.0$ (and all its images under~$\Gamma$).
Figure~\ref{fig:111} depicts the images of the  full-size horoball 
\[
H_A^+=H_0
\]
by~$\tau^{\pm 1}$, along with the vertical $A$-geodesics pointing upward from the centers of these horoballs and the $B$-geodesics between consecutive images of these horoball centers. 
The tangent points between consecutive images of~$H_A^+$ by~$\tau^{\pm 1}$ lie in the axis of the glide reflection induced by~$\tau$.
After conjugation by an isometry of~$\HH^3$ if necessary, we may assume that the translation vector of~$\tau$ horizontally points to the right and that the center of~$H_A^+$ lies above the axis of~$\tau$. 


\medskip

\begin{figure}[htbp!]
\centering
\vspace*{1.4cm}
\hspace*{-0.5cm}
\def\svgwidth{0.4\textheight}
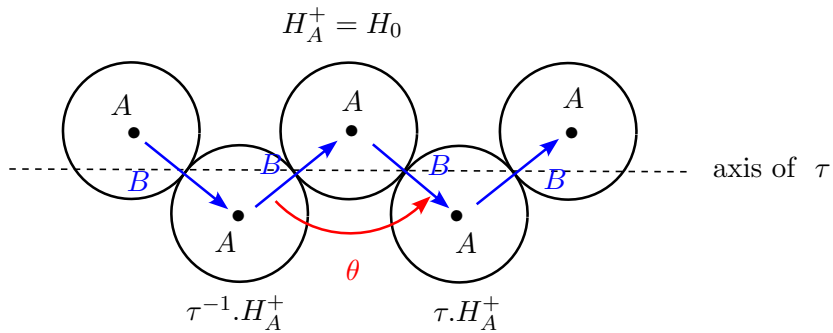
\vspace*{0.3cm}
\caption{$H_A^+$ and its nearby $\tau$-images.} \label{fig:111}
\end{figure}

Let $\gamma \in \Gamma$ be an isometry sending~$H_A^+$ to~$H_\infty$. 
Such an isometry is uniquely defined up to left multiplication by an element of~$\Gamma_\infty$.
The isometry~$\gamma$ takes~$H_\infty$ to a horoball
\[
H_A^- = H_\omega
\]
tangent to~$H_\infty$ at the point~$\gamma . p_0$. 
After left composition by an isometry of~$\Gamma_\infty$, we can assume that the center of~$H_A^-$ lies below the axis of~$\tau$ as close to it as possible.

Define 
\[
\tau'_\varepsilon=\gamma \tau^\varepsilon \gamma^{-1} \in \Gamma_\infty
\]
the conjugate of~$\tau^\varepsilon$ by~$\gamma$, where $\varepsilon = \pm 1$.
The isometry $\tau'_\varepsilon \in \Gamma_\infty$ is an orientation-reversing transformation fixing~$H_\infty$ with
\begin{equation} \label{eq:tau'}
d_{\partial H_\infty}(\gamma.p_0,\tau'_\varepsilon.\gamma.p_0) = 1.
\end{equation}

Similarly to~$\tau$, the isometry~$\tau'_\varepsilon$ acts on~$\partial H_\infty$ as a glide reflection of minimal displacement in~$\Gamma_\infty$.
Thus, the translation vectors of~$\tau$ and~$\tau'_\varepsilon$ coincide, up to the sign.
Choose $\varepsilon = \pm 1$ so that their translation vectors coincide and set 
\[
\tau'=\tau'_\varepsilon.
\]

By construction, the $A$-geodesic pointing upward from the center of~$H_A^+$ is sent by~$\gamma$ to an $A$-geodesic pointing downward toward the center of~$H_A^-$.
Now, by~\eqref{eq:tau'}, the full-size horoball~$H_A^-$ with an $A$-geodesic pointing downward to its center is tangent to its images by~$\tau'^{\pm 1}$.
As previously, the tangent points between consecutive images of~$H_A^-$ by~$\tau'^{\pm 1}$ lie in the axis of the glide reflection; see Figure~\ref{fig:222}.
This forces the axes of~$\tau$ and~$\tau'$ to be distinct, otherwise $H_A^-$ would coincide with the image of~$H_A^+$ by an iterate of~$\tau$ and the classes of $A$-geodesics and $\bar{A}$-geodesics would agree, which is impossible.
By our choice of~$\gamma$, the center~$\omega$ of~$H_A^-$ lies between the axis of~$\tau$ and~$\tau'$ (below the axis of~$\tau$ and above the axis of~$\tau'$). 

Since the translation vectors of~$\tau$ and~$\tau'$ agree and the centers of~$H_A^+$ and~$H_A^-$ are at distance~$1$ from~$\tau^{\pm 1}.H_A^+$ and~$\tau'.^{\pm 1}.H_A^-$, the isosceles triangles with vertices the centers of $H_A^+$, $\tau^{\pm 1}.H_A^+$ and $H_A^-$, $\tau'^{\pm 1}.H_A^-$ have the same side lengths and hence are isometric.
Observe also that the main vertices~$H_A^{\pm}$ of these triangles lie above their basis.
It follows that the oriented angle at the center of~$H_A^-$ between the centers of~$\tau'^{-1}.H_A^-$ and~$\tau'.H_A^-$ is congruent to the oriented angle at the center of~$H_A^+$ between the centers of~$\tau^{-1}.H_A^+$ and~$\tau.H_A^+$; see Figures~\ref{fig:111} and~\ref{fig:222}.
Denote by~$\theta$ this oriented angle.


\medskip

\begin{figure}[htbp!]
\centering
\vspace*{1.3cm}
\hspace*{-0.5cm}
\def\svgwidth{0.4\textheight}
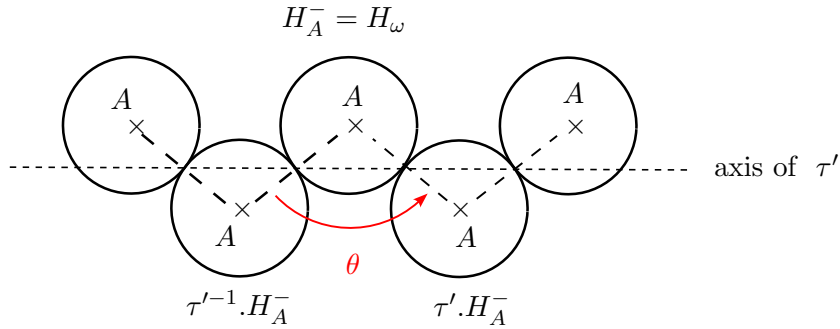
\vspace*{0.3cm}
\caption{$H_A^-$ and its nearby $\tau'$-images.} \label{fig:222}
\end{figure}

We may assume that $\theta$ is different from~$\frac{\pi}{3}$, otherwise the horoballs~$\tau^{\pm 1}.H_A^+$ would be tangent; see Figure~\ref{fig:111}.
Thus, the translation~$\tau^2$, taking~$\tau^{-1}.H_A^+$ to~$\tau.H_A^+$, would be of Euclidean length~$1$, which is excluded by assumption.
In particular, the horoballs~$\tau^{\pm 2}.H_A^-$ are not tangent to~$H_A^-$.

\medskip

Applying Proposition~\ref{prop:tangent} (together with Remark~\ref{rem:updown}) to the horoballs~$H_A^+$, $\tau.H_A^+$, $\tau^{-1}.H_A^+$ and the isometry~$\gamma$ yields three full-size horoballs 
\[
H_A^-=\gamma.H_\infty, \qquad H_B^-=\gamma \tau.H_A^+, \qquad H_B^+= \gamma \tau^{-1}.H_A^+
\] 
together with a $B$-geodesic pointing downward toward the center of~$H_B^-$, a $B$-geodesic pointing upward from the center of~$H_B^+$ and two (non-vertical) $A$-geodesics from the centers of~$H_B^{\pm }$ to the center of~$H_A^-$; see Figure~\ref{fig:333}.
The angle at the center of~$H_A^-$ between the centers of~$H_B^-$ and~$H_B^+$ coincide with~$\theta$, up to the sign; see Figure~\ref{fig:222}.

\begin{figure}[htbp!]
\centering
\vspace*{-0.5cm}
\hspace*{-0.5cm}
\def\svgwidth{0.75\textheight}
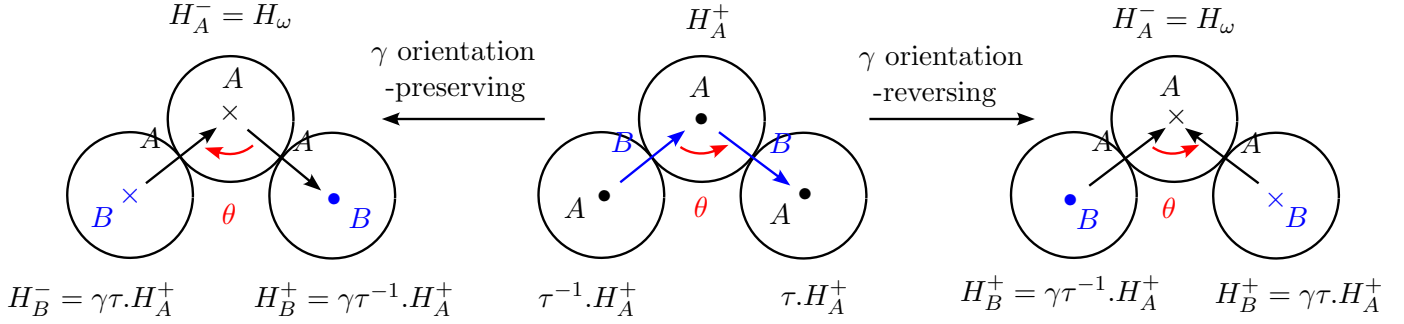
\vspace*{0,2cm}
\caption{Image of $H_A^+$ and of its nearby $\tau$-images under~$\gamma$ (up to isometry).} \label{fig:333}
\end{figure}

The pair of horoballs~$H_B^\pm$ tangent to~$H_A^-$ is distinct from the pair~$\tau'^{\pm 1}.H_A^-$, which is also tangent to~$H_A^-$.
Otherwise, the classes of~$B$-geodesics and $\bar{B}$-geodesics would coincide with the class of $A$-geodesics, which is impossible. 

\medskip

{\bf Three configurations.}
We are therefore led to consider the following configurations of the pair~$H_B^\pm$ around~$H_A^-$, relative to $\tau'^{\pm1}.H_A^-$:
\begin{enumerate}
\item $H_B^+$ and~$H_B^-$ lie on the same side of~$H_A^-$ with respect to~$\tau'^{\pm 1}.H_A^-$; \label{conf6} 
\item $H_B^+$ and~$H_B^-$ lie on opposite sides of~$H_A^-$ with respect to~$\tau'^{\pm 1}.H_A^-$; \label{conf5}
\item exactly one of the horoballs~$H_B^{\pm}$ coincides with~$\tau'^{\pm 1}.H_A^-$.  \label{conf1} \\
\end{enumerate}

\textbf{Configuration~\eqref{conf6}.}
Suppose that $H_B^+$ and~$H_B^-$ lie on the same side of~$H_A^-$ with respect to~$\tau'^{\pm 1}.H_A^-$.
Reflect~$\tau'^{\pm 1}.H_A^-$ across the line passing through the centers of~$H_A^-$ and~$\tau'^{\pm 2}.H_A^-$.
The resulting images are tangent to~$H_A^-$ and disjoint from all $\tau'$-iterates of~$H_A^-$.
Moreover, the angle at the center of~$H_A^-$ between the centers of these reflected horoballs equals (up to sign) the angle at the center of~$H_A^-$ between the centers of~$\tau'^{\pm 1}.H_A^-$.
The pair of reflected horoballs is the only pair of horoballs with these properties; see Figure~\ref{fig:000}.

\medskip

\begin{figure}[htbp!]
\centering
\vspace*{3.3cm}
\hspace*{-0.5cm}
\def\svgwidth{0.35\textheight}
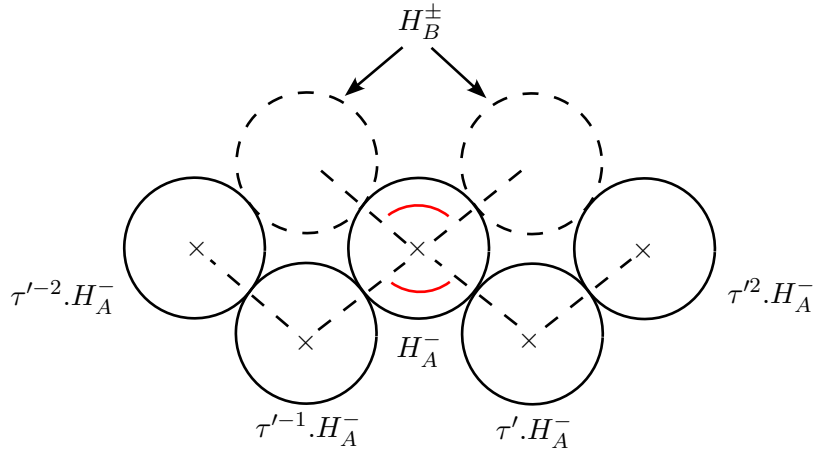
\vspace*{0.3cm}
\caption{$H_B^{\pm}$ lie in the same side.} \label{fig:000}
\end{figure}

Thus, $H_B^{\pm}$ must coincide with this reflected pair.
In particular, the translation~$\tau'^2$ sends $H_B^-$ to~$H_B^+$, or vice versa.
This is impossible, since the vertical $B$-geodesics at their centers point in opposite directions. \\

\textbf{Configuration~\eqref{conf5}.}
Suppose that $H_B^+$ and~$H_B^-$ lie on opposite sides of~$H_A^-$ with respect to~$\tau'^{\pm 1}.H_A^-$.
Without loss of generality, assume that $H_B^+$ lies above the row of horoballs formed of the $\tau'$-iterates of~$H_A^-$ (and $H_B^-$ lies below).
The other case is analogous.
The horoballs $\tau'^{\pm 1}.H_B^-$, which are tangent to $\tau'^{\pm 1}.H_A^-$, also lie above the row of the $\tau'$-iterates of~$H_A^-$.
The unit segments joining the centers of $\tau'.H_A^-$ and $\tau'.H_B^-$, and of~$\tau'^{-1}.H_A^-$ and $\tau'^{-1}.H_B^-$, are parallel since they differ by a translation of~$\tau'^2$.
Figure~\ref{fig:888} illustrates the relevant segments between the centers of~$H_A^-$, $\tau'^{\pm 1}.H_A^-$, $\tau'^{\pm 1}.H_B^-$ and~$H_B^+$.

\medskip

\begin{figure}[htbp!]
\centering
\vspace*{1,9cm}
\hspace*{-0.5cm}
\def\svgwidth{0.25\textheight}
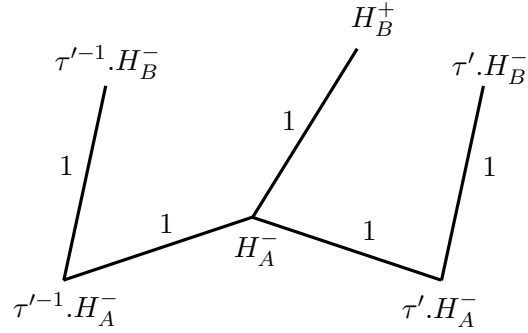
\vspace*{0.3cm}
\caption{Unit segments between the centers of~$H_A^-$, $\tau'^{\pm 1}.H_A^-$, $\tau'^{\pm 1}.H_B^-$ and~$H_B^+$.} \label{fig:888}
\end{figure}

Since the center of~$H_B^+$ lies at Euclidean distance at least~$1$ from the centers of $\tau'^{\pm 1}.H_B^-$, the segment joining the centers of~$H_A^-$ and~$H_B^+$ must be parallel to the segment joining the centers of~$\tau'.H_A^-$ and~$\tau'.H_B^-$; see Figure~\ref{fig:888}.
Thus, the two quadrilaterals with vertices the centers of
\[
H_A^-, \tau'.H_A^-, \tau'.H_B^-, H_B^+ \quad \mbox{ and } \quad H_A^-, H_B^+, \tau'^{-1}.H_B^-, \tau'^{-1}.H_A^-
\] 
have all sides of length~$1$, and are therefore unit rhombi (not necessarily isometric).

\medskip

The segments joining the centers of~$H_A^-$ and~$H_B^\pm$ are symmetric with respect to a horizontal axis.
Indeed, one is the image of the segment joining the centers of~$\tau'^{-1}.H_A^-$ and~$\tau'^{-1}.H_B^-$ under a translation, while the other is the image of the same segment under the glide reflection~$\tau'$.
In particular, the centers of~$H_B^\pm$ lie on a line~$L$ orthogonal to the axis of~$\tau'$.

Since $H_B^\pm$ lie on opposite sides of~$H_A^-$ with respect to~$\tau'^{\pm 1}.H_A^-$, the distance between their centers is at least~$\sqrt{3}$, and the same holds for the centers of~$\tau'^{\pm 1}.H_A$.
Thus, the translation length of the glide reflection~$\tau'$ is at least~$\frac{\sqrt{3}}{2}$.

\medskip

Let $\tau_\mnperp \in \Gamma_\infty$ be the nontrivial translation whose vector is orthogonal to the axis of~$\tau'$ (and~$\tau$), points from the center of~$H_B^-$ toward the center of~$H_B^+$, and has minimal length.
This isometry sends~$H_B^-$ to a horoball located beyond~$H_B^+$ along the line~$L$.
Otherwise, it would send~$H_B^-$ either to~$H_B^+$ or, in certain special cases, to~$H_A^-$.
In the latter case, composing the transformation with itself yields an element of~$\Gamma_\infty$ sending~$H_B^-$ to~$H_B^+$.
This is impossible since the vertical geodesics above their centers belong to different ($B$ and~$\bar B$) classes.
Hence the translation length of~$\tau_\mnperp$ is at least~$\sqrt{3}+1$.

However, this estimate is not sufficient.
Recall now that the horizontal row formed by the $\tau$-iterates of~$H_A^+$, or by one of its images under~$\Gamma_\infty$ (all with $A$-geodesics pointing upward), must lie between~$H_B^-$ and~$\tau_\mnperp.H_B^-$.
This implies that the translation length of~$\tau_\mnperp$ is at least~$2 \sqrt{3}+1$.

\medskip

We conclude that  the Klein bottle~$K$ corresponding to the maximal cusp boundary of~$C$ is isometric to the quotient of~$\R^2$ by the subgroup generated by
\begin{equation} \label{eq:ab}
(x,y) \mapsto \left( x+a,-y \right) \quad \text{and} \quad (x,y) \mapsto (x,y+b).
\end{equation}
with parameters
\[
a=\tfrac{\sqrt{3}}{2} \quad \mbox{ and } \quad b \geq 2 \sqrt{3}+1,
\] 
see Section~\ref{sec:Klein}.
By Proposition~\ref{prop:diskK}, the Klein bottle~$K$ contains an embedded Euclidean disk of radius~$\frac{\sqrt{3}}{2}$.
Proposition~\ref{prop:ballC} then implies that the cusp~$C$, and hence~$M$, contains an embedded ball of radius
\[
\frac{1}{2} \log \left( 1+\sqrt{3} \right) > \arcosh \left( \frac{\sqrt{5}}{2} \right).
\]

\medskip

\textbf{Configuration~\eqref{conf1}.}
Suppose that $H_B^-$ coincides with $\tau'^{-1}.H_A^-$.
(The other cases are treated similarly.)
Then $H_B^+$ cannot coincide with~$\tau'.H_A^-$, otherwise the pairs~$H_B^{\pm}$ and~$\tau'^{\pm 1}.H_A^-$ would agree.
This implies that the angle at the center of~$H_A^-$ between the centers of~$H_B^+$ and~$H_B^-$ is equal to~$\theta$ (and not~$-\theta$), and that the isometry~$\gamma$ is orientation-preserving; see Figures~\ref{fig:222} and~\ref{fig:333}.
Recall that there are two non-vertical $A$-geodesics from the centers of~$H_B^\pm$ to the center of~$H_A^-$ forming an angle~$\theta$; see Figure~\ref{fig:333}.
Since the vertical geodesics pointing downward to the centers of~$H_B^-$ and~$\tau'.H_A^-$ coincide, the classes of $A$-geodesics and $B$-geodesics agree; see Figures~\ref{fig:222} and~\ref{fig:333}.

The cusp diagram around~$H_A^-$ is depicted in Figure~\ref{fig:444}.

\medskip

\begin{figure}[htbp!]
\centering
\vspace*{2.4cm}
\hspace*{-0.5cm}
\def\svgwidth{0.3\textheight}
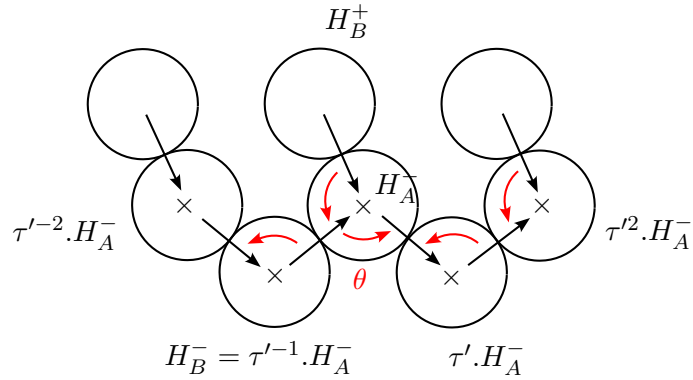
\vspace*{0.4cm}
\caption{$H_B^-$ coincides with $\tau'^{-1}.H_A^-$.} \label{fig:444}
\end{figure}

We will use following bounds on the angle of~$\theta$ in Case~2 below.

\begin{claim}
We have 
\[
\frac{\pi}{2} \leq \theta \leq \frac{5 \pi}{6}.
\]
\end{claim}

\begin{proof}
The angle at the center of~$H_A^-$ between the centers of~$\tau'.H_A^-$ and~$H_B^+$ is equal to~$2\pi-2\theta$.
This angle must be at least~$\frac{\pi}{3}$, otherwise these two horoballs, both tangent to~$H_A^-$, would intersect.
Thus, $\theta \leq \frac{5 \pi}{6}$.

Consider the isosceles trapezoid with vertices the centers of~$\tau'^{-2}.H_A^-$, $H_B^-$, $H_A^-$, $H_B^+$.
The angles at the centers of~$H_B^-$ and~$H_A^-$ are both equal to~$\theta$.
It follows that $\theta \geq \frac{\pi}{2}$, otherwise $\tau'^{-2}.H_A^-$ and~$H_B^+$ would intersect.
\end{proof}

Both $H_A^+$ and~$H_B^+$ have $A$-geodesics pointing upward from their centers.
Hence, there exists an isometry~$\alpha \in \Gamma_\infty$ sending~$H_A^+$ to~$H_B^+$.
Define
\[
\sigma = \alpha \tau \alpha^{-1}.
\]
Then $\sigma$ acts on~$\partial H_\infty$ as a glide reflection with the same translation vector as~$\tau$ and~$\tau'$, of length~$\sin \frac{\theta}{2}$.
In particular, $\sigma^{\pm 2}.H_B^+$ is tangent to~$\tau'^{\pm 2}.H_A^-$.
Since $\tau.H_A^+$ is tangent to both~$H_A^+$ and~$\tau^2.H_A^+$, it follows that $\sigma.H_B^+$ is tangent to both~$H_B^+$ and~$\sigma^2.H_B^+$.
Similarly, $\sigma^{-1}.H_B^+$ is tangent to both~$H_B^+$ and~$\sigma^{-2}.H_B^+$.

\medskip

We will consider two cases depending whether the center of~$\sigma.H_B^+$ lies below or above the axis of~$\sigma$.
 (We show that the first case does not occur.) \\

\noindent \textbf{Case 1.} Assume that the center of~$\sigma.H_B^+$ lies below or on the axis of~$\sigma$.
Since the translations~$\sigma^2$ and~$\tau'^2$ agree, the triangles with vertices the centers of $H_A^-$, $\tau'^{\pm 1}.H_A^-$ and $H_B^+$, $\sigma^{\pm 1}. H_B^+$ have the same side lengths and hence are isometric.
Under the assumption that the center of~$\sigma.H_B^+$ lies below or on the axis of~$\sigma$, these two isometric triangles are related by a translation of unit vector; see Figure~\ref{fig:555}.
Thus, $\sigma.H_B^+$ is tangent to~$H_B^-=\tau'.H_A^-$, and $\sigma^{-1}.H_B^+$ is tangent to~$\tau'^{-1}.H_A^-$.

\medskip

\begin{figure}[htbp!]
\centering
\vspace*{3,7cm}
\hspace*{-0.5cm}
\def\svgwidth{0.3\textheight}
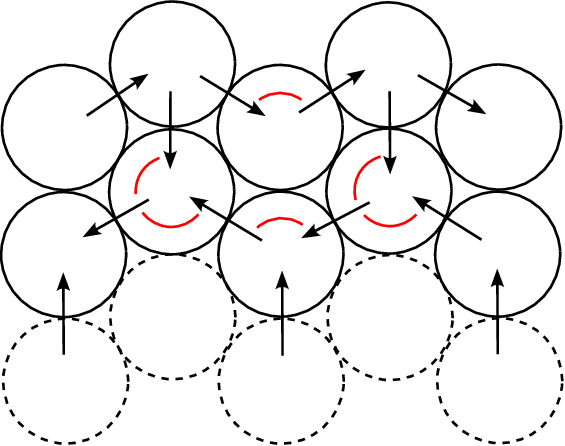
\vspace*{0.1cm}
\caption{Center of~$\sigma.H_B^+$ lying below the axis of~$\sigma$.} \label{fig:555}
\end{figure}

Since $\sigma.H_B^+$ is tangent to~$H_B^-$ and lies between~$H_A^-$ and~$\tau'^2.H_A^-$, the angle~$\theta$ at the center of~$H_B^-$ between the centers of~$\tau'.H_B^-$ and~$\tau'^{-1}.H_B^-$ is at least~$\frac{2\pi}{3}$.
On the other hand, in the diamond with vertices the centers of~$H_A^-$, $H_B^+$, $\sigma^{-1}.H_B^+$, $\tau'^{-1}.H_A^-$, the angle~$\theta$ at the center of~$H_A^-$ must be at most~$\frac{2\pi}{3}$, otherwise $H_A^-$ and~$\sigma^{-1}.H_B^+$ would intersect.
Therefore, $\theta=\frac{2\pi}{3}$ and $\sigma.H_B^+$ is tangent to~$H_A^-$, $\tau'.H_A^-$ and~$\tau'^2.H_A^-$.
Taking their image under~$\tau'^{-1}$ and~$\tau'^{-2}$, it follows that $H_A^-$ is tangent to the full-size horoballs~$\sigma.H_B^+$, $\tau'^{-1} \sigma.H_B^+$ and~$\tau'^{-2} \sigma.H_B^+ = \sigma^{-1}.H_B^+$ with an $A$-geodesic pointing upward from their center.

Let $\eta \in \Gamma$ be an isometry sending~$H_A^-$ to~$H_\infty$.
The isometry~$\eta$ takes~$H_\infty$ to a horoball tangent to~$H_\infty$ with an $A$-geodesic pointing upward from its center.
After left composition by an isometry of~$\Gamma_\infty$, we may assume that~$\eta.H_\infty=H_B^+$.
Now, the six full-size horoballs tangent to~$H_A^-$, namely $\tau'^{-1}.H_A^-$, $\tau'^{-1} \sigma.H_B^+$, $\tau'.H_A^-$, $\sigma.H_B^+$, $H_B^+$, $\sigma^{-1}.H_B^+$, are sent by~$\eta$ to six full-size horoballs around~$H_B^+$.
The horoball~$H_B^+$ has two $A$-geodesics from its center to the adjacent horoballs $\sigma.H_B^+$ and~$H_A^-$ tangent to~$H_B^+$.
By Proposition~\ref{prop:tangent} (together with Remark~\ref{rem:updown}), these two $A$-geodesics must be the images under~$\tau_\perp$ of two $A$-geodesics pointing downward to the centers of two adjacent full-size horoballs tangent to~$H_A^-$.
Of the six full-size horoballs tangent to~$H_A^-$, exactly two have an $A$-geodesic pointing downward to their center, but these two horoballs, $\tau'^{-1}.H_A^-$ and~$\tau'.H_A^-$, are not adjacent.
Hence a contradiction.  \\

\noindent \textbf{Case 2.}
Assume that the center of~$\sigma.H_B^+$ lies above the axis of~$\sigma$; see Figure~\ref{fig:666}.

\medskip

\begin{figure}[htbp!]
\centering
\vspace*{3,2cm}
\hspace*{-0.5cm}
\def\svgwidth{0.3\textheight}
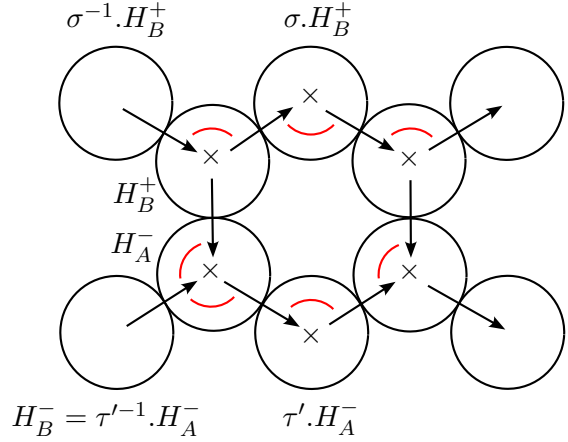
\vspace*{0.4cm}
\caption{Center of~$\sigma.H_B^+$ lying above the axis of~$\sigma$.} \label{fig:666}
\end{figure}

Consider the hexagon~$\mathcal{H}$ whose vertices are the centers of the horoballs $H_B^+$, $H_A^-$, $\tau'.H_A^-$, $\tau'^2.H_A^-$, $\sigma^2.H_B^+$, $\sigma.H_B^+$.
Since the full-size horoballs centered at these vertices are tangent along consecutive pairs, the sides of the hexagon~$\mathcal{H}$ have unit length.
Furthermore, since the translations~$\sigma^2$ and ~$\tau'^2$ coincide, the isosceles triangles with vertices the centers of $H_A^-$, $\tau'.H_A^-$, $\tau'^2.H_A^-$ and $H_B^+$, $\sigma.H_B^+$, $\sigma^2.H_B^+$ have the same side lengths, and hence are isometric.
Under the assumption that the center of~$\sigma.H_B^+$ lies above the axis of~$\sigma$, it follows that the opposite sides of~$\mathcal{H}$ are pairwise parallel and have equal length.
Therefore, the hexagon~$\mathcal{H}$ is centrally symmetric; see Figure~\ref{fig:666}.
By this symmetry, the angles of~$\mathcal{H}$ at $\tau'.H_A^-$, $\tau'^2.H_A^-$, $\sigma.H_B^+$, $H_B^+$ are equal to~$\theta$ and those at $H_A^-$ and~$\sigma^2.H_B^+$ are equal to~$2\pi-2\theta$.

\medskip

Recall that every flat Klein bottle $K=K(a,b)$ is determined by geometric parameters~$(a,b)$ corresponding to isometries of~$\R^2$ generating its deck transformation group; see~\eqref{eq:ab} and Section~\ref{sec:Klein}.

\begin{claim} \label{claim:ab}
The parameters~$(a,b)$ of the Klein bottle~$K$ corresponding to the maximal cusp boundary of~$C$ are given by 
\[
a=\sin \tfrac{\theta}{2} \quad \mbox{ and } \quad b = 4 \sin \tfrac{\theta}{2} \sin \theta.
\] 
\end{claim}



\begin{proof}
Recall first that the translation vectors of $\tau$, $\tau'$, $\sigma$ are of length~$\sin \frac{\theta}{2}$.
Furthermore, this is the minimal length among all glide reflections in~$\Gamma_\infty$ with the same axes.
Hence, 
\[
a=\sin \tfrac{\theta}{2}.
\]

The distance between the axes of~$\tau'$ and~$\sigma$ is equal to the $y$-coordinate of the vector between the centers of~$H_B^-$ and~$H_B^+$.
Therefore,
\[
d(\textrm{axis}(\tau'),\textrm{axis}(\sigma)) = 2 \sin \tfrac{\theta}{2} \sin \theta.
\]

Suppose there is glide reflection of~$\Gamma_\infty$ whose axis lies strictly between those of~$\tau'$ and~$\sigma$.
Then, composing this glide reflection with a suitable iterate of~$\tau'$ or~$\sigma$ would yield a nontrivial translation in the $y$-direction of length less than the distance between the axes of~$\tau'$ and~$\sigma$.
The image under this translation of the horizontal row of horoballs formed by the $\tau'$-iterates of~$H_A^-$ would intersect the row given by the $\sigma$-iterates of~$H_B^+$, which is impossible.

It follows that the minimal distance between the axes of the glide reflections in~$\Gamma_\infty$ is 
\[
\tfrac{b}{2}=2 \sin \tfrac{\theta}{2} \sin \theta.
\]
\end{proof}

In general, Claim~\ref{claim:ab} does not yield an embedded Euclidean disk or radius greater than $\frac{1}{4} (1+\sqrt{5})$ on the cusp boundary~$K$ as we would like; see Remark~\ref{rem:final}.
This condition holds only when 
\[
\tfrac{3\pi}{5}= 1.88495... \leq \theta \leq 2.33128
\]
and in particular for $\theta = \frac{2\pi}{3} = 2.09439...$
In this range, Proposition~\ref{prop:ballC} implies that $\cosh \RR(M) > \frac{\sqrt{5}}{2}$ as desired.

\medskip

For other values of~$\theta$, we will locate a large embedded ball not entirely contained in the cusp~$C$.

\begin{claim} \label{claim:1.65}
If $1.65 \leq \theta < \frac{5\pi}{6}$, then $\cosh \RR(M) > \frac{\sqrt{5}}{2}$.
\end{claim}

\begin{proof}
Consider the tiling of the plane by the hexagons in the $\Gamma_\infty$-orbit of~$\mathcal{H}$.
Since we may assume that $\theta \neq \frac{2\pi}{3}$, these hexagons are not regular.
In particular, there is no room for a full size-size horoball to lie inside one of them without intersecting the interior of the full-size horoballs centered at their vertices.

Let $H$ be a horoball (not of full-size) contained in one of these hexagons.
By Proposition~\ref{prop:hexagon}, the center of~$H$ is at distance at most
\[
D=
\begin{cases}
\frac{1}{2 \cos \frac{\theta}{2}} & \text{if } \frac{\pi}{2} \leq \theta \leq \frac{2\pi}{3} \\
\sqrt{\frac{1}{4} + \sin^2 \theta} & \text{if } \frac{2\pi}{3} \leq \theta \leq \frac{5\pi}{6}
\end{cases}
\]
from the center of a full-size horoball.
By the distance--height estimate~\eqref{eq:h1h2}, it follows that
\begin{equation} \label{eq:heightH}
\height(H) \leq D^2.
\end{equation}

Let $p_t$ be the point of height~$t$ above the center~$O$ of the hexagon~$\mathcal{H}$.
We will show that for some~$t>0$ and every $\gamma \in \Gamma \setminus \{ \id \}$,
\[
d(p_t,\gamma.p_t) > 2 \arcosh \left( \frac{\sqrt{5}}{2} \right).
\]

Suppose that $\gamma$ fixes~$H_\infty$, that is, $\gamma \in \Gamma_\infty$.
The point~$p_1$ is the projection of~$p_t$ to~$\partial H_\infty$.
It lies at equal distance from the axes of~$\tau'$ and~$\sigma$.
By Proposition~\ref{prop:diskK}, this implies that
\[
|\gamma.p_t-p_t| = |\gamma.p_1-p_1| \geq 2r
\]
where
\[
r = \frac{1}{2} \min \left\{ 2 \sin \tfrac{\theta}{2}, \sin \tfrac{\theta}{2} \sqrt{1+4 \sin^2 \theta} \right\}.
\]
By the Euclidean--hyperbolic distance formula~\eqref{eq:hyp-eucl3}, this yields
\begin{equation} \label{eq:gamma1}
d(p_t,\gamma.p_t) \geq 2 \arsinh \left( \frac{r}{t} \right) > 2 \arcosh \left( \frac{\sqrt{5}}{2} \right)
\end{equation}
whenever $t<2r$.

\medskip

Suppose that $\gamma$ sends~$H_\infty$ to a horoball~$H$ which is not of full-size.
By the height bound~\eqref{eq:heightH}, the image of~$p_1 \in \partial H_\infty$ under~$\gamma$ has height at most~$D^2$.
Similarly,
\[
\height(\gamma.p_t) \leq \frac{D^2}{t}.
\]
Using the Euclidean--hyperbolic distance formula~\eqref{eq:hyp-eucl2}, we obtain
\begin{equation} \label{eq:gamma2}
d(p_t,\gamma.p_t) \geq \log \left( \frac{t^2}{D^2} \right) > 2 \arcosh \left( \frac{\sqrt{5}}{2} \right)
\end{equation}
whenever $t > \frac{1+\sqrt{5}}{2} \, D$.

\medskip

Before moving on to the next case, observe that 
\[
\frac{1+\sqrt{5}}{2} D < 2r.
\]
Indeed, for $\frac{\pi}{2} \leq \theta \leq \frac{2\pi}{3}$, this is equivalent to
\[
\tan \tfrac{\theta}{2} > \frac{1+\sqrt{5}}{2}
\]
which holds since $\tan \frac{\theta}{2} \geq 1$.
For $\frac{2\pi}{3} \leq \theta \leq \frac{5\pi}{6}$, this is equivalent to
\[
\sin \tfrac{\theta}{2} > \frac{1+\sqrt{5}}{4}
\]
which also holds since $\sin \frac{\theta}{2} \geq \frac{\sqrt{3}}{2}$.

Thus, for $t<2r$ sufficiently close to~$2r$, both inequalities~\eqref{eq:gamma1} and~\eqref{eq:gamma2} are satisfied.

\medskip

Suppose that $\gamma$ sends~$H_\infty$ to a full-size horoball~$H$.
Let $d$ denote the Euclidean distance on~$\partial H$ between the projection~$\gamma.p_1$ of~$\gamma.p_t$ to~$\partial H$ and the tangent point of~$H$ and~$H_\infty$.
Then $d \geq R$, otherwise the center of~$\gamma^{-1}.H_\infty$ would be at distance less than~$D$ from~$O$, which is absurd.

By rotational symmetry around the vertical geodesic~$\Delta$ though the center of~$H$, we may assume that $p_t$ and $\gamma.p_t$ lie in the same vertical half-plane bounded by $\Delta$.
More precisely, let $q_1 \in \partial H$ be the point at Euclidean distance~$d$ on~$\partial H$ from the tangency point of~$H$ and~$H_\infty$, measured along~$\partial H$, and lying in the previous half-plane.
Then define~$q_t$ as the point obtained from~$q_1$ by moving a signed hyperbolic distance~$\log(t)$ along the geodesic orthogonal to~$\partial H$ at~$q_1$, directed into~$H$.
It follows that
\[
d(p_t,\gamma_t) \geq d(p_t,q_t).
\]
Since we are only seeking a lower bound, we may further assume that the center of~$H$ is as close as possible to~$O$.

To compute~$q_t$ explicitly, we introduce coordinates in which the center of~$H$ is at the origin in~$\C$ and the point~$O$ is represented by~$D$.
In these coordinates, the point~$q_t$ is obtained by applying the inversion with respect to the hemisphere of Euclidean radius~$1$ centered at the origin of~$\C$ to the point~$(d,t)$.
Therefore,
\[
\gamma.p_t = \left( \frac{d}{d^2+t^2}, \frac{t}{d^2+t^2} \right).
\]
Recall that $p_t=(D,t)$ in these coordinates.
By the Euclidean--hyperbolic distance formula~\eqref{eq:hyp-eucl}, we obtain
\[
d(p_t,q_t) = \arcosh \left( 1+\frac{\left( D - \frac{d}{d^2+t^2} \right)^2 + \left( t - \frac{t}{d^2+t^2} \right)^2}{\frac{2 t^2}{d^2+t^2}} \right).
\]
Denote this expression by~$f(D,d,t)$.
One can check that the function $d \mapsto f(D,d,t)$ is increasing for $d \geq D$ and $D^2+t^2 >1$.
Taking $t$ close to~$2r \geq \sqrt{2}$ ensures that $t>1$.
Hence, $d(p_t,q_t) \geq f(D,D,t)$.

\medskip

If $\frac{\pi}{2} \leq \theta \leq \frac{2\pi}{3}$, then $D=\frac{1}{2 \cos \frac{\theta}{2}}$ and $2r = 2 \sin \frac{\theta}{2}$.
The function $\theta \mapsto f(D,D,2r)$ is increasing on this interval and equals $2 \arcosh(\frac{\sqrt{5}}{2})$ for some $\theta_0 \in (1.64,1.65)$.
Thus, for $\theta \geq 1.65$, choosing $t$ sufficiently close to~$2r$ yields
\[
d(p_t,\gamma.p_t) \geq d(p_t,q_t) > 2 \arcosh \left( \frac{\sqrt{5}}{2} \right).
\]

If $\frac{2\pi}{3} \leq \theta < \frac{5\pi}{6}$, then $D=\sqrt{\frac{1}{4} + \sin^2 \theta}$ and $2r = \sin \frac{\theta}{2} \sqrt{1+ 4 \sin^2 \theta}$.
The function $\theta \mapsto f(D,D,2r)$ is decreasing on this interval and equals $2 \arcosh(\frac{\sqrt{5}}{2})$ at $\theta = \frac{5\pi}{6}$.
Thus, for $\theta<\frac{5\pi}{6}$, the same conclusion follows.
\end{proof}

Our next result shows that $\theta$ can only take restricted values.

\begin{claim} \label{claim:except2}
The only possible value of~$\theta$ in~$(\frac{\pi}{2},\frac{3\pi}{5})$ is
\[
\arccos\left( \frac{1-\sqrt{2}}{2} \right) \simeq 1.77941...
\]
In particular, the angle~$\theta$ does not lie in~$(\frac{\pi}{2},1.65)$.
\end{claim}

\begin{proof}
Let $\eta \in \Gamma$ be an isometry sending~$\sigma^2.H_B^+$ to~$H_\infty$.
The isometry~$\eta$ takes~$H_\infty$ to a horoball tangent to~$H_\infty$ with an $A$-geodesic pointing downward to its center.
After left composition by an isometry of~$\Gamma_\infty$, we may assume that~$\eta.H_\infty=H_A^-$.

Apply Proposition~\ref{prop:tangent} (together with Remark~\ref{rem:updown}) to $\sigma^2.H_B^+$ and its three full-size tangent horoballs $\sigma.H_B^+$, $\tau'^2.H_A^-$, $\sigma^3.H_B^+$.
This yields one of the cusp diagrams around~$H_A^-$ depicted in Figure~\ref{fig:777} with two $A$-geodesics pointing downward to the centers of~$\eta \tau'^2.H_A^-$ and~$\eta \sigma^3.H_B^+$, an $A$-geodesic pointing upward from the center of~$\eta \sigma.H_B^+$, and three non-vertical $A$-geodesics from~$\eta \sigma.H_B^+$ and~$\eta \sigma^3.H_B^+$ to~$H_A^-$, and from~$\eta \sigma^2.H_B^+$ to~$\eta \tau'^{-2}.H_A^-$.

\medskip

\begin{figure}[htbp!]
\centering
\vspace*{0.7cm}
\hspace*{-0.5cm}
\def\svgwidth{0.75\textheight}
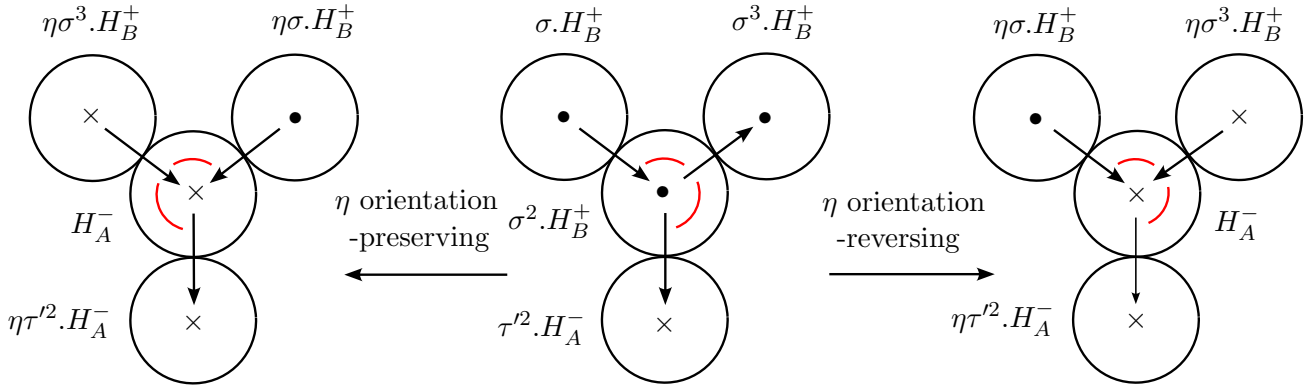
\caption{Cusp diagram around~$\sigma^2.H_B^+$: a and b} \label{fig:777}
\end{figure}

Since the hexagons in the $\Gamma_\infty$-orbit of~$\mathcal{H}$ are not regular, no full-size horoball can lie in their interior. 
Thus, the only full-size horoballs tangent to~$H_A^-$ are $H_B^+$, $\tau'^{\pm 1}.H_A^-$.
It follows that $\eta$ sends $\sigma.H_B^+$ to~$H_B^+$, since it is the only full-size horoball tangent to~$H_A^-$ with an upward-pointing $A$-geodesic.
It also follows that $\eta$ sends $\sigma^3.H_B^+$ to~$\tau'^{-1}.H_A^-$, and $\tau'^2.H_A^-$ to~$\tau'.H_A^-$, since this is the only way for the non-vertical $A$-geodesics around~$\sigma^2.H_B^+$ to match those around~$H_A^-$ under~$\eta$.
In particular, the isometry~$\eta$ is orientation-preserving and
\[
\eta \sigma.H_B^+=H_B^+, \qquad \eta \sigma^3.H_B^+ = \tau'^{-1}.H_A^-, \qquad \eta \tau'^2 . H_A^- = \tau' . H_A^-.
\]

Now, apply Proposition~\ref{prop:tangent} again to~$\sigma^2.H_B^+$, its two full-size tangent horoballs~$\sigma.H_B^+$ and~$\tau'^2.H_A^-$, and the full-size horoball~$H_A^-$ whose center lies on the equidistant bisector of~$\sigma.H_B^+$ and~$\tau'^2.H_A^-$ at distance $1-2 \cos \theta$ from the center of~$\sigma^2.H_B^+$.
This yields a horoball~$H$ of height $\frac{1}{(1-2 \cos \theta)^2}$ at distance $\frac{1}{1-2 \cos \theta}$ from the center of~$H_A^-$ along the ray toward~$\sigma^2.H_B^+$.

Similarly, apply Proposition~\ref{prop:tangent} with~$\eta^{-1}$ to~$H_A^-$, its two full-size tangent horoballs~$\tau'.H_A^-$ and~$H_B^+$, and the full-size horoball~$\sigma^2.H_B^+$ whose center lies on the equidistant bisector of~$\tau'.H_A^-$ and~$H_B^+$ at distance $1-2 \cos \theta$ from the center of~$H_A^-$.
This yields a horoball~$H'$ of height $\frac{1}{(1-2 \cos \theta)^2}$ at distance $\frac{1}{1-2 \cos \theta}$ from the center of~$\sigma^2.H_B^+$ along the ray toward~$H_A^-$.

The two horoballs~$H$ and~$H'$ cannot intersect transversally (without being tangent or equal).
This excludes the values of~$\theta$ such that
\[
0 < \left| \frac{1}{1-2 \cos \theta}  - \frac{1}{2} (1-2 \cos \theta) \right|< \frac{1}{2(1-2 \cos \theta)^2} 
\]
Hence, the only possible value of~$\theta$ in~$(\frac{\pi}{2},\frac{3\pi}{5})$ is $\arccos\left( \frac{1-\sqrt{2}}{2} \right)$.
\end{proof}

By Claims~\ref{claim:1.65} and~\ref{claim:except2}, the only remaining possibilities are the two extremal cases $\theta = \frac{\pi}{2}$ and $\theta = \frac{5\pi}{6}$.

\begin{claim}
If $\theta = \frac{\pi}{2}$ or~$\frac{5\pi}{6}$, then $\cosh \RR(M) > \frac{\sqrt{5}}{2}$.
\end{claim}

\begin{proof}
We begin by showing that in both cases there exists a regular ideal octahedron~$\mathcal{O}$ whose vertices are the centers of horoballs in the packing of~$C$, and such that the horoballs corresponding to adjacent vertices are tangent.

Suppose that $\theta=\frac{\pi}{2}$.
The horoball packing of~$C$ given by the cusp diagram is a square horoball packing, where $H_A^-$ and~$\sigma^2.H_B^+$ are tangent; see Figure~\ref{fig:pi/2}.
Let $\eta \in \Gamma$ be an isometry sending~$H_A^-$ to~$H_\infty$.
Apply Proposition~\ref{prop:tangent} with~$\eta$ to~$H_A^-$ and the four full-size horoballs $\sigma.H_B^+$, $\tau'^{\pm 2}.H_A^-$, $\tau'.H_B^+$ whose centers lie at Euclidean distance~$\sqrt{2}$ from the center of~$H_A^-$.
This yields four horoballs of height~$\frac{1}{2}$ around~$\eta.H_\infty$ (tangent to it), and by $\Gamma_\infty$-periodicity, around every full-size horoball in the cusp diagram.
These horoballs of height~$\frac{1}{2}$ are centered at the vertices of the dual lattice of the cusp diagram.
Indeed, by the distance--height estimate~\eqref{eq:h1h2}, any other horoball disjoint from the interiors of the full-size horoballs has height less than~$\frac{1}{2}$.
Let~$H$ be the horoball of height~$\frac{1}{2}$ tangent to~$H_A^-$, $\tau'.H_A^-$, $\tau'^2.H_A^-$, $\sigma^2.H_B^+$.
We then define~$\mathcal{O}$ to be the regular ideal octahedron whose vertices are the centers of~$H_\infty$, $H_A^-$, $\tau'.H_A^-$, $\tau'^2.H_A^-$, $\sigma^2.H_B^+$, $H$.

\medskip

\begin{figure}[htbp!]
\centering
\vspace*{4.2cm}
\hspace*{-0.5cm}
\def\svgwidth{0.35\textheight}
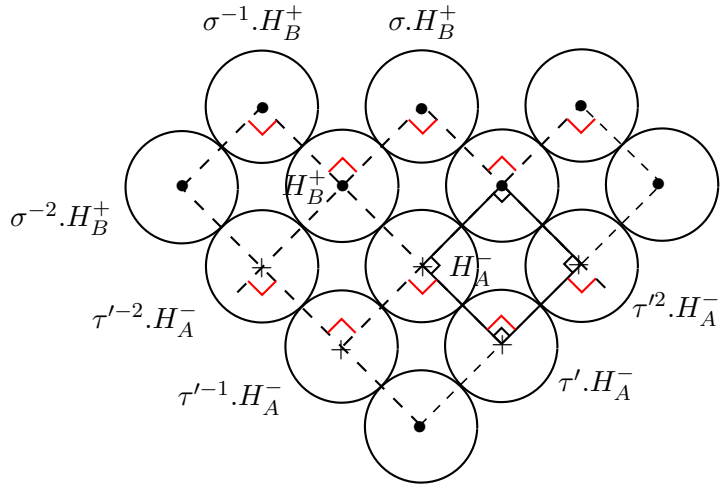
\vspace*{0cm}
\caption{Horoball packing for $\theta=\frac{\pi}{2}$} \label{fig:pi/2}
\end{figure}

Suppose that $\theta=\frac{5\pi}{6}$.
The cusp diagram contains four horoballs forming a square of size~$1$, namely $H_B^+$, $\tau'.H_A^-$, $\tau'^2.H_A^-$, $\sigma.H_B^+$; see Figure~\ref{fig:5pi/6}.
Let $\eta \in \Gamma$ be an isometry sending~$H_B^+$ to~$H_\infty$
Apply Proposition~\ref{prop:tangent} to~$H_B^+$ and the horoball~$\tau'^2.H_A^-$ whose center is at distance~$\sqrt{2}$ from that of~$H_B^+$.
Using again the distance--height estimate, $\Gamma_\infty$-periodicity and the configuration in Figure~\ref{fig:5pi/6}, we obtain a horoball~$H$ of height~$\frac{1}{2}$ lying in the center of this square and tangent to the full-size horoballs centered at its vertices.
In this case, we define~$\mathcal{O}$ as the regular ideal octahedron whose vertices are the centers of~$H_\infty$, $H_B^+$, $\tau'.H_A^-$, $\tau'^2.H_A^-$, $\sigma.H_B^+$, $H$.

\medskip

\begin{figure}[htbp!]
\centering
\vspace*{4.2cm}
\hspace*{-0.5cm}
\def\svgwidth{0.45\textheight}
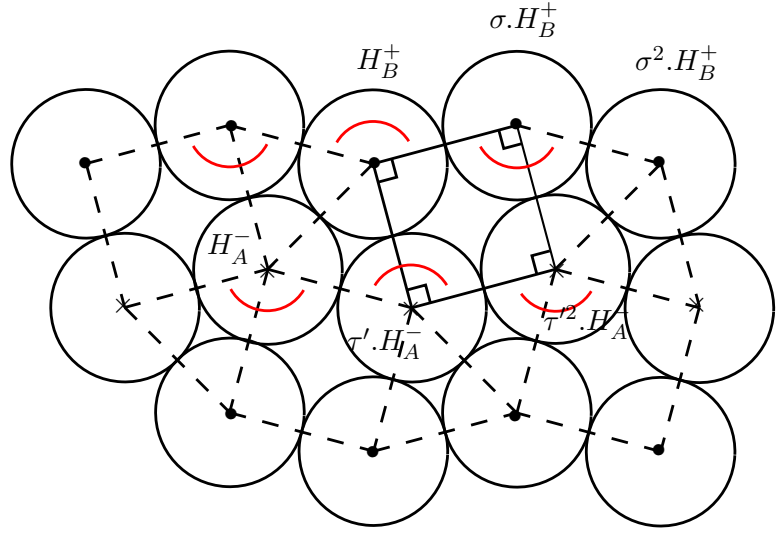
\vspace*{0cm}
\caption{Horoball packing for $\theta=\frac{5\pi}{6}$} \label{fig:5pi/6}
\end{figure}

In both case, consider the six maximal horoballs $H_1,\dots,H_6$ centered at the vertices of~$\mathcal{O}$.
Let $\gamma \in \Gamma$ be a nontrivial isometry.
The collection $\mathcal{C}=\cup_{1 \leq i \leq 6} \{ H_i, \gamma.H_i \}$ forms a horoball packing, since distinct horoballs have disjoint interiors.
The associated Delaunay decomposition (dual to the Voronoi decomposition) has $\mathcal{O}$ as the cell determined by $H_1,\dots,H_6$, and $\gamma.\mathcal{O}$ as the cell determined by $\gamma.H_1, \dots, \gamma.H_6$.
Distinct Delaunay cells have disjoint interiors.
Hence, $\mathcal{O}$ and~$\gamma.\mathcal{O}$ have disjoint interior.
Indeed, if they intersected, they would coincide and $\gamma$ would fix the center~$p_0$ of~$\mathcal{O}$, which is impossible since $\Gamma$ acts freely.
Therefore, the distance between~$p_0$ and~$\gamma.p_0$ is at least twice the distance between $p_0$ and the faces of~$\mathcal{O}$.
By Lemma~\ref{lem:octahedron}, it follows that
\[
\RR(M) \geq d(p_0,\gamma.p_0) \geq 2 \arcosh \left( \frac{\sqrt{6}}{2} \right) = 1.31695...
\]
which implies $\cosh \RR(M) > \frac{\sqrt{5}}{2}$.
\end{proof}

For the sake of completeness, we prove the following straightforward result.

\begin{lemma} \label{lem:octahedron}
The distance between the center of a regular ideal octahedron~$\mathcal{O}$ of~$\HH^3$ and its faces is equal to
\[
\arcosh \left( \frac{\sqrt{6}}{2} \right).
\]
\end{lemma}

\begin{proof}
In~$\HH^3$, we can realize~$\mathcal{O}$ as the ideal octahedron with vertices $0$, $\infty$, $\pm 1$, $\pm i$.
By symmetry, its center is the point~$p_0=(0,0,1)$.
The nearest-point projection of~$p_0$ to the face with ideal vertices $1$, $i$, $\infty$ lies in the vertical geodesic given by~$q_t=(\frac{1}{2},\frac{1}{2},t)$.
By the Euclidean--hyperbolic distance formula~\eqref{eq:hyp-eucl}, we obtain
\[
d(p_0,q_t) = \arcosh \left( \frac{t^2+\frac{3}{2}}{2t} \right).
\]
This distance is minimal for $t=\frac{\sqrt{6}}{2}$, and the corresponding minimal distance is $\arcosh \left( \frac{\sqrt{6}}{2} \right)$.
\end{proof}
This concludes the proof of Theorem~\ref{theo:nonorientable_cusp}.
\end{proof}

Putting all these results together, we now proceed to the proof Theorem~\ref{theo:rigidity}.

\begin{proof}[Proof of Theorem~\ref{theo:rigidity}]
Let $M$ be a noncompact, nonorientable, complete hyperbolic~$3$-manifold of finite volume.
By Theorem~\ref{theo:tau=1}, we may assume that $M$ satisfies the equality case in the inradius lower bound.
Namely,
\[
\cosh(\RR(M)) = \frac{\sqrt{5}}{2}.
\]
Fix a maximal cusp~$C$ of~$M$.

Suppose first that $C$ is orientable. 
As noted in Remark~\ref{rem:w>1}, its waist size is equal to $1$.
By the equality case in Theorem~\ref{theo:adams_nonorientable}, it follows that $M$ is isometric to the figure-eight knot complement.
In this situation, Proposition~\ref{prop:inradius_orientable} yields
\[
\cosh(\RR(M)) = \frac{\sqrt{22}}{4}
\]
which contradicts the previous assumption on the inradius of~$M$.

Therefore, we conclude that $C$ is nonorientable.
In this case, Theorem~\ref{theo:nonorientable_cusp} implies that $M$ is isometric to the Gieseking manifold.
\end{proof}

\forget

\textbf{Configuration~\eqref{conf1}.}
Suppose that $H_B^-$ coincides with $\tau'.H_A^-$.
(The other cases are treated similarly. PROBLEM WHEN $H_B^-$ COINCIDES WITH $\tau'^{-1}.H_A^-$)
Then $H_B^+$ differs from~$\tau'^{-1}.H_A^-$, otherwise the pairs~$H_B^{\pm}$ and~$\tau'^{\pm 1}.H_A^-$ would coincide.
This forces the angle at the center of~$H_A^-$ between the centers of~$H_B^-$ and~$H_B^+$ to be equal to~$\theta$ (and not~$-\theta$); see Figures~\ref{fig:333} and~\ref{fig:444}.

\medskip

\begin{figure}[htbp!]
\centering
\includegraphics[height=50mm]{Fig444.pdf}
\caption{$H_B^-$ coincides with $\tau'.H_A^-$.} \label{fig:444}
\end{figure}

Since the vertical geodesics pointing downward to the centers of~$H_B^-$ and~$\tau'.H_A^-$ agree, the classes of $A$-geodesics and $B$-geodesics coincide.
Now, both $H_A^+$ and~$H_B^+$ have $A$-geodesics pointing upward from their centers.
Hence, there exists an isometry~$\alpha \in \Gamma_\infty$ sending~$H_A^+$ to~$H_B^+$.
Define
\[
\sigma = \alpha \tau \alpha^{-1}.
\]
Then $\sigma$ acts on~$\partial H_\infty$ as a glide reflection with the same translation vector as~$\tau$ and~$\tau'$.
In particular, $\sigma^{\pm 2}.H_B^+$ is tangent to~$\tau'^{\pm 2}.H_A^-$.
Since $\tau.H_A^+$ is tangent to both~$H_A^+$ and~$\tau^2.H_A^+$, it follows that $\sigma.H_B^+$ is tangent to~$H_B^+$ and~$\sigma^2.H_B^+$.
Similarly, $\sigma^{-1}.H_B^+$ is tangent to~$H_B^+$ and~$\sigma^{-2}.H_B^+$.

\medskip

Let $\eta \in \Gamma$ be an isometry sending~$H_A^-$ to~$H_\infty$.
The isometry~$\eta$ takes~$H_\infty$ to a horoball tangent to~$H_\infty$ with an $A$-geodesic pointing upward from its center.
After left composition by an isometry of~$\Gamma_\infty$, we can assume that~$\eta.H_\infty=H_B^+$.
Applying Proposition~\ref{prop:tangent} (together with Remark~\ref{rem:updown}) with~$\eta$ to $H_A^-$ and its three full-size tangent horoballs $H_B^+$ and~$\tau'^{\pm 1}.H_A^-$ yields the cusp diagram around~$H_B^+$ depicted in Figure~\ref{fig:777} with two $A$-geodesics pointing upward from the centers of~$\eta.H_B^+$ and~$\eta \tau'.H_A^-$, an $A$-geodesic pointing downward toward the center of~$\eta \tau'^{-1}.H_A^-$, and three non-vertical $A$-geodesics from~$\eta.H_B^+$ to~$H_B^+$ and $H_B^+$ to~$\eta \tau'^{\pm 1}.H_A^-$.

\medskip

\begin{figure}[htbp!]
\centering
\includegraphics[height=90mm]{Fig777.pdf}
\caption{Cusp diagram around~$H_B^+$: a and b} \label{fig:777}
\end{figure}

Furthermore, the oriented angles at the center of~$H_B^+$ between the centers of~$\eta \tau'^{-1}.H_A^-$ and $\eta \tau'.H_A^-$, and $\eta \tau'.H_A^-$ and~$\eta.H_B^+$ are both equal to~$\theta$ or~$-\theta$, depending whether $\eta$ is orientation-reversing or not.

\begin{lemma}
The hexagon~$\mathcal{H}$ with vertices the centers of the horoballs $H_B^+$, $H_A^-$, $H_B^-$, $\tau'^2.H_A^-$, $\sigma^2.H_B^+$, $\sigma.H_B^+$ is regular.
\end{lemma}

\begin{proof}
Since the full-size horoballs centered at the vertices of~$\mathcal{H}$ are tangent, the sides of the hexagon have unit length.

Since the translations~$\sigma^2$ and ~$\tau'^2$ agree, the isosceles triangles~$\Delta^-$ and~$\Delta^+$ with vertices the centers of $H_A^-$, $\tau'^{\pm 1}.H_A^-$ and $H_B^+$, $\sigma^{\pm 1}.H_B^+$ have the same side lengths and hence are isometric.

\medskip

\emph{Case 1.} Assume that the center of~$\sigma.H_B^+$ lies above or on the axis of~$\sigma$.
Then the opposite sides of~$\mathcal{H}$ are pairwise parallel and of the same length.
Hence the hexagon~$\mathcal{H}$ is centrally symmetric; see Figure~\ref{fig:666}.
Thus, the angles of~$\mathcal{H}$ at the center of~$H_A^-$, $H_B^-$, $\sigma^2.H_B^+$, $\sigma.H_B^+$ are equal to~$\theta$.
The same holds with the other two hexagons~$\tau'^{-1}.\mathcal{H}$ and~$\tau'^{-2}.\mathcal{H}$ around the center of~$H_A^-$.
In particular, the angle at the center of~$H_B^+$ between the centers of $\sigma^{-1}.H_B^+$ and~$H_A^-$ is equal to~$\theta$.

\medskip

\begin{figure}[htbp!]
\centering
\includegraphics[height=50mm]{Fig666.jpg}
\caption{Center of~$\sigma.H_B^+$ lying above the axis of~$\sigma$.} \label{fig:666}
\end{figure}

Suppose that these isometric hexagons are not regular.
Then there is no room for a full-size horoball to lie inside them.
(Recall we only consider horoballs covering cusps of~$M$.)
Thus, the only full-size horoballs tangent to~$H_B^+$ are $H_A^-$, $\sigma^{\pm 1}.H_B^+$.
This means that $\eta$ sends $\tau'^{-1}.H_A^-$ to~$H_A^-$ since it is the only full-size horoball tangent to~$H_B^+$ with a downward-pointing $A$-geodesic.
This also means that $\eta$ sends $H_B^+$ to~$\sigma^{-1}.H_B^-$, and $\tau'.H_A^-$ to~$\sigma.H_B^+$, since this is the only way for the nonvertical $A$-geodesics around~$H_A^-$ to match those around~$H_B^+$ under~$\eta$.
Thus, the isometry~$\eta$ is orientation-reversing.
Furthermore, the angle at the center of~$H_B^+$ between the centers of~$H_A^-$ and~$\sigma.H_B^+$ is equal to the angle~$\theta$ at the center of~$H_A^-$ between the centers of~$\tau'^{-1}.H_A^-$ and~$\tau'.H_A^-$.
Therefore, the angles at the center of~$H_A^-$ between the centers of its full-size tangent horoballs are all equal to~$\theta$.
That is, $\theta=\frac{2\pi}{3}$.
Hence the hexagon~$\mathcal{H}$ is regular.

\medskip

\emph{Case 2.} Assume that the center of~$\sigma.H_B^+$ lies below the axis of~$\sigma$.
Since the translations~$\sigma^2$ and~$\tau'^$ agree, the triangles with vertices the centers of $H_A^-$, $\tau'^{\pm 1}.H_A^-$ and $H_B^+$, $\sigma^{\p[m 1}. H_B^+$ have the same side lengths and hence are isometric.
Under the assumption that the center of~$\sigma.H_B^+$ lies below the axis of~$\sigma$, these two isometric triangles are related by a translation of unit vector; see Figure~\ref{fig:555}.
Thus, $\sigma.H_B^+$ is tangent to~$H_B^-=\tau'.H_A^-$, and $\sigma^{-1}.H_B^+$ is tangent to~$\tau'^{-1}.H_A^-$.

\medskip

\begin{figure}[htbp!]
\centering
\includegraphics[height=50mm]{Fig555.jpg}
\caption{Center of~$\sigma.H_B^+$ lying below the axis of~$\sigma$.} \label{fig:555}
\end{figure}

Since $\sigma.H_B^+$ is tangent to~$H_B^-$ and lies between~$H_A^-$ and~$\tau'^2.H_A^-$, the angle~$\theta$ at the center of~$H_B^-$ between the centers of~$\tau'.H_B^-$ and~$\tau'^{-1}.H_B^-$ is at least~$\frac{2\pi}{3}$.
On the other hand, in the diamond with vertices the centers of~$H_A^-$, $H_B^+$, $\sigma^{-1}.H_B^+$, $\tau'^{-1}.H_A^-$, the angle~$\theta$ at the center of~$H_A^-$ must be at most~$\frac{2\pi}{3}$, otherwise $H_A^-$ and~$\sigma^{-1}.H_B^+$ would intersect.
Therefore, $\theta=\frac{2\pi}{3}$ and $\sigma.H_B^+$ is tangent to~$H_A^-$, $\tau'.H_A^-$ and~$\tau'^2.H_A^-$.
Taking their image under~$\tau'^{-1}$ and~$\tau'^{-2}$, it follows that $H_A^-$ is tangent to the full-size horoballs~$\sigma.H_B^+$, $\tau'^{-1} \sigma.H_B^+$ and~$\tau'^{-2} \sigma.H_B^+ = \sigma^{-1}.H_B^+$ with an $A$-geodesic pointing upward from their center.

???Now, the six full-size horoballs tangent to~$H_A^-$, namely $\tau'^{-1}.H_A^-$, $\tau' \sigma^{-1}.H_B^+$, $\tau'.H_A^-$, $\sigma.H_B^+$, $H_B^+$, $\sigma^{-1}.H_B^+$, are sent by~$\eta$ to six full-size horoballs around~$H_B^+$.
The horoball~$H_B^+$ has two $A$-geodesics from its center to the adjacent horoballs $\sigma.H_B^+$ and~$H_A^-$ tangent to~$H_B^+$.
These two $A$-geodesics must be the images under~$\eta$ of two $A$-geodesics pointing downward to the centers of two adjacent full-size horoballs tangent to~$H_A^-$.
Of the six full-size horoballs tangent to~$H_A^-$, exactly two have an $A$-geodesic pointing downward to their center, but these two horoballs, $\tau'^{-1}.H_A^-$ and~$\tau'.H_A^-$, are not adjacent.
Hence a contradiction. 
\end{proof}

Since the hexagon~$\mathcal{H}$ is regular, see Figure~\ref{fig:666}, the translation vectors of the glide reflections induced by~$\tau'$ and~$\sigma$ have length~$\frac{\sqrt{3}}{2}$, and the distance between their axes is~$2$.
Furthermore, the cusp diagram. shows that no glide reflection induced by an element of~$\Gamma_\infty$ has its axis between those of~$\tau'$ and~$\sigma$.
Hence, the Klein bottle~$K$ corresponding to the maximal cusp boundary of~$C$ is isometric to the quotient of~$\R^2$ by the subgroup generated by
\[
(x,y) \mapsto \left( x+\tfrac{\sqrt{3}}{2},-y \right) \quad \text{and} \quad (x,y) \mapsto (x,y+4).
\]

In other words, its parameters are $(a,b)=(\frac{\sqrt{3}}{2},4)$; see Section~\ref{sec:Klein}.
By Proposition~\ref{prop:diskK}, the Klein bottle~$K$ contains an embedded Euclidean disk of radius~$\frac{\sqrt{3}}{2}$.
Proposition~\ref{prop:ballC} then implies that the cusp~$C$, and hence~$M$, contains an embedded ball of radius
\[
\frac{1}{2} \log \left( 1+\sqrt{3} \right) > \arcosh \left( \frac{\sqrt{5}}{2} \right).
\]
\end{proof}

\forgotten

\subsection{Inradius of Klein bottles and cusps} \label{sec:Klein}

The maximal radius of a Euclidean open disk embedded in a flat torus~$\T^2$ is equal to~$\frac{1}{2} \sys(\T^2)$.
We extend this result to flat Klein bottles as follows. 

\medskip

Every flat Klein bottle is isometric to the quotient~$K(a,b)$ of~$\R^2$ by the group~$\Gamma_{a,b} = \langle u,v \rangle$ generated by
\begin{align*}
u:(x,y) & \mapsto (x+a,-y) \\
v:(x,y) & \mapsto (x,y+b)
\end{align*}
where $a>0$ and~$b>0$.
The parameters~$(a,b)$ are uniquely defined.

\begin{proposition} \label{prop:diskK}
The Klein bottle~$K(a,b)$ contains an embedded Euclidean open disk of radius~$r$ if and only if
\[
a \geq r, \quad \frac{b}{2} \geq r, \quad a^2 + \left(\frac{b}{2} \right)^2 \geq 4r^2.
\]

Thus, the maximal radius of a Euclidean open disk embedded in~$K(a,b)$ is
\[
\frac{1}{2} \min \left\{ 2a,b,\sqrt{a^2+\left( \frac{b}{2} \right)^2} \right\}.
\]
\end{proposition}

\begin{proof}
The Klein bottle~$K(a,b)$ contains an embedded Euclidean open disk of radius~$r$ if and only if there exists $p=(x,y) \in \R^2$ such that
\[
\min_{\gamma \in \Gamma_{a,b} \setminus \{ \id \}} d(p,\gamma(p)) \geq 2r.
\]

The orientation-preserving transformations of~$\Gamma_{a,b}$ are of the form~$u^{2n} v^m$ with $n,m \in \Z$, that is,
\[
u^{2n} v^m:(x,y) \mapsto (x+2na,y+mb).
\]
Thus, 
\begin{equation} \label{eq:min+}
\min_{\gamma \in \Gamma^+_{a,b} \setminus \{ \id \}} d(p,\gamma(p)) = \min\{2a,b\}
\end{equation}
where $\Gamma^+_{a,b} \leqslant \Gamma_{a,b}$ is the subgroup of orientation-preserving transformations.

The orientation-reversing transformations of~$\Gamma_{a,b}$ are of the form $\gamma=u^{2n+1}v^m$ with $n,m \in \Z$, that is,
\[
u^{2n+1}v^m:(x,y) \mapsto (x+(2n+1)a,-y+mb).
\]
Thus, 
\[
d(p,\gamma(p)) = \sqrt{(2n+1)^2 a^2 + (mb-2y)^2}
\]
and
\[
\min_{\gamma \in \Gamma_{a,b} \setminus \Gamma^+_{a,b}} d(p,\gamma(p)) = \sqrt{a^2+d(2y,b\Z)^2}.
\]
Unlike~\eqref{eq:min+}, this expression depends on the $y$-coordinate of~$p$.
It is maximal when $y \equiv \frac{b}{4} \, \text{mod} \, \frac{b}{2}$.

In this case, 
\[
\min_{\gamma \in \Gamma_{a,b} \setminus \{ \id \}} d(p,\gamma(p)) = \min \left\{ 2a,b,\sqrt{a^2+\left( \frac{b}{2} \right)^2} \right\}.
\]
This minimum is greater or equal to~$2r$ if and only if
\[
a \geq r, \quad \frac{b}{2} \geq r, \quad a^2 + \left(\frac{b}{2} \right)^2 \geq 4r^2.
\]
\end{proof}

The following result relates the inradius of a cusp to the inradius of its boundary.

\begin{proposition} \label{prop:ballC}
Let $C$ be a cusp (orientable or nonorientable) of a complete finite-volume hyperbolic $3$-manifold.
Suppose that the cusp boundary~$\partial C$ contains an embedded Euclidean disk of radius $r$.
Then $C$ contains an embedded hyperbolic open ball of radius at least
\[
\frac{1}{2} \log(1+2r).
\]
\end{proposition}

\begin{proof}
We may assume that $C$ is covered by the horoball~$H_\infty$ in~$\HH^3$ and that the embedded Euclidean disk of radius~$r$ in~$\partial C$ lifts to a disk centered at $p_0=(0,0,1) \in \HH^3$.
By assumption, 
\[
d_{\partial H_\infty}(p_0,\gamma.p_0) \geq 2r
\]
for every $\gamma \in \Gamma_\infty$.

Define $p_t=(0,0,e^t) \in H_\infty$ for $t \geq 0$.
The cusp~$C$ contains an embedded open ball of radius~$R$ centered at the projection of~$p_t$ provided
\[
d(p_t,\partial H_\infty) \geq R \quad \text{and} \quad \min_{\gamma \in \Gamma_\infty \setminus \{ \id \}} d(p_t,\gamma.p_t) \geq 2R.
\]

By the Euclidean--hyperbolic distance formulas~\eqref{eq:hyp-eucl2} and~\eqref{eq:hyp-eucl3},
\[
d(p_t,\partial H_\infty) = t
\]
and 
\[
d(p_t,\gamma.p_t) = 2 \arsinh \left( \frac{d_{\partial H_\infty}(p_0,\gamma.p_0)}{2 e^t} \right) \geq 2 \arsinh(r e^{-t})
\]
since $p_t$ and~$\gamma.p_t$ have the same height~$e^t$ for every $\gamma \in \Gamma_\infty$.

Hence $C$ contains an embedded open ball of radius
\[
R_t = \min \{ t, 2 \arsinh(r e^{-t}) \}.
\]
Now, $R_t$ is maximal when $t = 2 \arsinh(r e^{-t})$.
That is, when 
\[
t=\frac{1}{2} \log(1+2r),
\]
which is also the maximal value of~$R_t$.
\end{proof}

\begin{remark} \label{rem:final}
If the cusp boundary of~$C$ contains an embedded Euclidean disk of radius
\[
r>\frac{1}{4} \left( 1+\sqrt{5} \right) = 0.80901...
\]
then 
\[
\RR(M) > \arcosh \left( \frac{\sqrt{5}}{2} \right).
\]
\end{remark}

We will also need the following elementary result from planar Euclidean geometry.

\begin{proposition} \label{prop:hexagon}
Let $\mathcal{H}=ABCDEF$ be a centrally symmetric hexagon in the Euclidean plane with all sides of length~$1$
Assume that the angles at $B,C,E,F$ are equal to $\theta \in [\pi/2,\pi]$. 
Then the maximal distance from a point $P \in \mathcal{H}$ to the set of vertices of~$\mathcal{H}$ is
\[
R=
\begin{cases}
\frac{1}{2\cos \frac{\theta}{2}} & \text{if } \frac{\pi}{2} \leq \theta \leq \frac{2\pi}{3} \\
\sqrt{\frac{1}{4}+\sin^2 \theta} & \text{if } \frac{2\pi}{3} \leq \theta \leq \pi.
\end{cases}
\]
Moreover, for $\theta<\frac{2\pi}{3}$ the maximum is achieved at exactly two symmetric points, whereas for $\theta \geq \frac{2\pi}{3}$ it is achieved uniquely at the center of symmetry~$O$ of $\mathcal{H}$.
\end{proposition}

\begin{proof}
Let $\Delta$ be the perpendicular bisector of~$[B,C]$.
By symmetry with respect to the axis~$(AD)$, it is enough to consider the trapezoid~$ABCD$.
By symmetry with respect to~$\Delta$ and by convexity of the function
\[
P \longmapsto \min\{PA,PB,PC,PD\},
\]
the maximum of this function over~$ABCD$ is attained on~$\Delta$.

If $\theta \leq \frac{2\pi}{3}$, the perpendicular bisectors of~$[A,B]$, $[B,C]$, $[C,D]$ intersect at a point~$P \in \mathcal{H}$.
This point is the unique maximizer of the distance function to the set of vertices within~$ABCD$.
Clearly,
\[
PA=PB=PC=PD= \frac{1}{2 \cos \frac{\theta}{2}}.
\]

If $\theta > \frac{2\pi}{3}$, the perpendicular bisectors of~$[A,B]$, $[B,C]$, $[C,D]$ intersect outside~$\mathcal{H}$.
In this case, the center~$O$ of~$\mathcal{H}$ is the unique maximizer.
The point~$O$ lies at distance~$\sin \theta$ from~$[B,C]$, and we obtain
\[
OA = OB > OB=OC = \sqrt{\frac{1}{4} + \sin^2 \theta}.
\]
\end{proof}

\end{document}